\documentclass[12pt,a4paper,fleqn]{amsart}

\usepackage[utf8]{inputenc}
\usepackage[T1]{fontenc}

\usepackage[foot]{amsaddr}
\usepackage[in]{fullpage}
\usepackage[colorlinks,linkcolor=blue,anchorcolor=blue,citecolor=blue]{hyperref}
\usepackage{amsfonts,amssymb,amsthm}
\usepackage[alphabetic,nobysame]{amsrefs}
\catcode`\@=11

\def\@settitle{%
  \baselineskip14\p@\relax
    {\Large\bfseries
  \@title}}

\def\@setauthors{%
  \begingroup
  \def\thanks{\protect\thanks@warning}%
  \trivlist
  \footnotesize \@topsep45\p@\relax
  \advance\@topsep by -\baselineskip
  \item\relax
  \author@andify\authors
  \def\\{\protect\linebreak}%
  {\sc\fontsize{12}{10}\selectfont\authors}%
  \ifx\@empty\contribs
  \else
    ,\penalty-3 \space \@setcontribs
    \@closetoccontribs
  \fi
  \endtrivlist
  \endgroup
}

\def\@secnumfont{\bfseries}%

\def\section{\@startsection{section}{1}%
  \z@{.7\linespacing\@plus\linespacing}{.5\linespacing}%
  {\normalfont\bf}}

\renewcommand{\BibLabel}{%
    \Hy@raisedlink{\hyper@anchorstart{cite.\CurrentBib}\hyper@anchorend}%
    [\thebib]\hfill%
}

\DefineSimpleKey{bib}{arxiv}

\BibSpec{article}{%
    +{}  {\PrintAuthors}                {author}
    +{,} { \textit}                     {title}
    +{.} { }                            {part}
    +{:} { \textit}                     {subtitle}
    +{,} { \PrintContributions}         {contribution}
    +{.} { \PrintPartials}              {partial}
    +{,} { }                            {journal}
    +{}  { \textbf}                     {volume}
    +{}  { \PrintDatePV}                {date}
    +{,} { \issuetext}                  {number}
    +{,} { \eprintpages}                {pages}
    +{,} { }                            {status}
    +{,} { \PrintDOI}                   {doi}
    +{,} { \arxiv}                      {arxiv}
    +{}  { \parenthesize}               {language}
    +{}  { \PrintTranslation}           {translation}
    +{;} { \PrintReprint}               {reprint}
    +{.} { }                            {note}
    +{.} {}                             {transition}
    +{}  {\SentenceSpace \PrintReviews} {review}
}

\newcommand{\arxiv}[1]{\href{http://arxiv.org/abs/#1}{{\tt arxiv:\hspace{0pt}#1}}}

\catcode`\@=12

\makeatletter
\renewcommand{\tocsection}[3]{%
  \indentlabel{\@ifnotempty{#2}{\ignorespaces#1 \hbox to 15.01021pt{\hfil#2.}\quad}}#3}
\makeatother

\numberwithin{equation}{section}\swapnumbers

\usepackage{parskip}

\usepackage[cmtip,matrix,arrow,curve]{xy} \SelectTips{cm}{10}
\newcommand{\cxymatrix}[1]{\vcenter{\xymatrix@=15pt{#1}}}
\usepackage{rotating}
\newcommand{\xysubseteq}{\ar@{}[r]|{\displaystyle\subseteq}}
\newcommand{\xysubseteqleft}{\ar@{}[r]|{\hspace{10pt}\displaystyle\supseteq}}
\newcommand{\xysubseteqdown}{\ar@{}[d]|{\rotatebox{90}{$\supseteq$}}}

\usepackage{aliascnt}

\usepackage[capitalize]{cleveref}

\newtheorem{theorem}{Theorem}[section]

\newaliascnt{lemma}{theorem}
\newtheorem{lemma}[lemma]{Lemma}
\aliascntresetthe{lemma}
  
\newaliascnt{corollary}{theorem}
\newtheorem{corollary}[corollary]{Corollary}
\aliascntresetthe{corollary}

\newaliascnt{proposition}{theorem}
\newtheorem{proposition}[proposition]{Proposition}
\aliascntresetthe{proposition}

\newtheorem{stheorem}{Theorem}[subsection]

\newaliascnt{slemma}{stheorem}
\newtheorem{slemma}[slemma]{Lemma}
\aliascntresetthe{slemma}

\newaliascnt{scorollary}{stheorem}
\newtheorem{scorollary}[scorollary]{Corollary}
\aliascntresetthe{scorollary}

\newaliascnt{sproposition}{stheorem}
\newtheorem{sproposition}[sproposition]{Proposition}
\aliascntresetthe{sproposition}

\theoremstyle{definition}

\newaliascnt{sremark}{stheorem}
\newtheorem{sremark}[sremark]{Remark}
\aliascntresetthe{sremark}

\newaliascnt{definition}{theorem}
\newtheorem{definition}[definition]{Definition}
\aliascntresetthe{definition}

\newtheorem*{ack}{Acknowledgment}

\newaliascnt{remark}{theorem}
\newtheorem{remark}[remark]{Remark}
\newtheorem{remarks}[remark]{Remarks}
\aliascntresetthe{remark}

\newaliascnt{example}{theorem}
\newtheorem{example}[example]{Example}
\newtheorem{examples}[example]{Examples}
\aliascntresetthe{example}

\usepackage{enumitem}
\setlist[enumerate,2]{label=\textit{\alph*)},ref=\textit{\alph*})}
\setlist[enumerate,1]{label=\textit{\roman*)},ref=\textit{\roman*)}}

\usepackage{microtype}

\renewcommand\tilde{\widetilde}

\newcommand{\fk}{\mathfrak{k}}

\newcommand{\Zq}{{\overline Z}}
\newcommand{\ogamma}{{\bar\gamma}}
\newcommand{\cA}{\mathcal{A}}
\newcommand{\cG}{\mathcal{G}}
\newcommand{\cH}{\mathcal{H}}
\newcommand{\cF}{\mathcal{F}}
\newcommand{\cO}{\mathcal{O}}
\newcommand{\cC}{{\mathcal C}}
\newcommand{\cM}{{\mathcal M}}
\newcommand{\cR}{{\mathcal R}}
\newcommand{\cS}{{\mathcal S}}
\newcommand{\cZ}{{\mathcal Z}}
\newcommand{\cN}{{\mathcal N}}
\newcommand{\RS}{{\widetilde R}}
\newcommand{\fAS}{{\widetilde\fA}}
\newcommand{\cZS}{{\widetilde\cZ}}
\newcommand{\cCS}{{\widetilde\cC}}
\newcommand{\phiS}{{\widetilde\phi}}
\newcommand{\cFS}{{\widetilde\cF}}

\newcommand{\fS}{{\widetilde f}}
\newcommand{\GS}{{\widetilde G}}

\newcommand{\QS}{{\widetilde Q}}
\newcommand{\uS}{{\widetilde u}}

\newcommand{\vS}{{\widetilde v}}
\newcommand{\XS}{{\widetilde X}}
\newcommand{\YS}{{\widetilde Y}}
\newcommand{\ZS}{{\widetilde Z}}
\newcommand{\sigmaS}{{\widetilde\sigma}}

\newcommand{\fa}{\mathfrak{a}}

\newcommand{\fA}{\mathfrak{A}}
\newcommand{\fg}{\mathfrak{g}}
\newcommand{\fh}{\mathfrak{h}}
\newcommand{\fl}{\mathfrak{l}}
\newcommand{\fm}{\mathfrak{m}}
\newcommand{\Fq}{\mathfrak{q}}
\newcommand{\fu}{\mathfrak{u}}

\newcommand{\CC}{\mathbb{C}}
\newcommand{\QQ}{\mathbb{Q}}
\newcommand{\RR}{\mathbb{R}}
\newcommand{\HH}{\mathbb{H}}
\newcommand{\ZZ}{\mathbb{Z}}

\newcommand{\sA}{\mathsf{A}}
\newcommand{\sB}{\mathsf{B}}
\newcommand{\sC}{\mathsf{C}}
\newcommand{\sD}{\mathsf{D}}
\newcommand{\sE}{\mathsf{E}}
\newcommand{\sF}{\mathsf{F}}
\newcommand{\sG}{\mathsf{G}}
\newcommand{\sU}{\mathsf{U}}
\newcommand{\sAq}{\overline{\sA}}
\newcommand{\sXq}{\overline{\sX}}
\newcommand{\sO}{\mathsf{O}}
\newcommand{\sM}{\mathsf{M}}
\newcommand{\sP}{\mathsf{P}}
\newcommand{\sZ}{\mathsf{Z}}
\newcommand{\sX}{\mathsf{X}}
\newcommand{\sYq}{\overline{\mathsf{Y}}}
\newcommand{\sXs}{\tilde{\mathsf{X}}}
\newcommand{\sY}{\mathsf{Y}}
\newcommand{\sH}{\mathsf{H}}
\newcommand{\sK}{\mathsf{K}}

\newcommand{\leer}{\varnothing}
\renewcommand{\rho}{\varrho}
\renewcommand{\mod}{/\!\!/}
\renewcommand{\phi}{\varphi}
\renewcommand{\epsilon}{\varepsilon}
\renewcommand{\hat}{\widehat}

\newcommand{\<}{\langle} 
\renewcommand{\>}{\rangle}

\newcommand{\auf}{\twoheadrightarrow}
\newcommand{\into}{\hookrightarrow}

\newcommand\semidir{\ltimes}
\renewcommand{\[}{\begin{equation}}
\renewcommand{\]}{\end{equation}}
\newcommand{\prim}{^{\rm pr}}
\newcommand{\p}{^{(p)}}
\newcommand{\el}{_{\rm el}}
\newcommand{\an}{_{\rm an}}
\newcommand{\st}{_{\rm st}}

\DeclareMathOperator{\supp}{supp}
\DeclareMathOperator{\rk}{rk}
\DeclareMathOperator{\Hom}{Hom}
\DeclareMathOperator{\Ad}{Ad}
\DeclareMathOperator{\Spec}{Spec}
\DeclareMathOperator{\Aut}{Aut}
\DeclareMathOperator{\res}{res}
\DeclareMathOperator{\Lie}{Lie}

\def\|#1|{\operatorname{#1}}

\newcommand{\osigma}{{\overline\sigma}}
\newcommand{\oalpha}{{\overline\alpha}}
\newcommand{\otau}{{\overline\tau}}
\newcommand{\aq}{{\overline a}}
\newcommand{\eq}{{\overline e}}
\newcommand{\pq}{{\overline p}}
\newcommand{\vq}{{\overline v}}
\newcommand{\wq}{{\overline w}}
\newcommand{\xq}{{\overline x}}
\newcommand{\fn}{{\mathfrak n}}
\newcommand{\fq}{{\overline f}}
\newcommand{\Aq}{{\overline A}}

\newcommand{\Xq}{{\overline X}}

\newcommand{\Yq}{{\overline Y}}
\newcommand{\Gq}{{\overline G}}
\newcommand{\Nq}{{\overline N}}
\newcommand{\FQ}{{\overline F}}
\newcommand{\uq}{{\overline u}}
\newcommand{\Uq}{{\overline U}}
\newcommand{\Lq}{{\overline L}}
\newcommand{\Pq}{{\overline P}}
\newcommand{\Qq}{{\overline Q}}

\newcommand{\G}{{\mathbf{G}}}
\renewcommand{\P}{{\mathbf{P}}}
\newcommand{\A}{{\mathbf{A}}}
\newcommand{\rat}{\!\xymatrix@=10pt{\ar@{.>}[r]&}\!}
\newcommand{\aut}{^{\rm aut}}

\newcommand{\half}{{\textstyle\frac12}}
\newcommand\BIGCUP{\operatorname*\bigcup\limits}

\newcommand\Kwonderful{$K$-adjoint}
\newcommand\kwonderful{$k$-adjoint}

\newcommand\Cite[2]{\cite{#1}*{#2}}

\title[]{Reductive group actions}

\author[]{Friedrich Knop$^1$}
\address[]{$^1$Dept. Mathematik\\
  FAU Erlangen-Nürnberg\\
Cauerstraße 11\\
D-91058 Erlangen}
\author[]{Bernhard Krötz$^2$}
\address[]{$^2$Institut für Mathematik\\
Uni.~Paderborn\\
Warburger Str.~100\\
D-33098 Paderborn}


\newcommand\red{^\natural}
\newcommand\nor{^{\rm n}}
\renewcommand\2{^\sharp}
\renewcommand\b{^\flat}

\begin{document}

\begin{abstract}

  In this paper, we study rationality properties of reductive group
  actions over fields $k$ of characteristic zero. In doing so, we
  unify Luna's theory of spherical systems with the Borel--Tits theory
  of reductive groups. Specifically, for any reductive group action, we
  define a generalized Tits index; its main constituents are a root
  system and a generalization of the anisotropic kernel. For
  $k$-spherical varieties (i.e., those for which a minimal parabolic
  subgroup has an open orbit), we obtain explicit smooth completions
  of the set of rational points. Over local fields, this yields
  genuine compactifications that generalize the maximal Satake
  compactification of symmetric spaces. Our main tool is a $k$-version
  of the Local Structure Theorem.

\end{abstract}

\maketitle

{\setlength{\parskip}{0pt}\tableofcontents}

\section{Introduction}

Let $G$ be a connected reductive group over a field $k$ of
characteristic zero. The main focus of this paper is the study of
rationality properties of $k$-varieties equipped with a
$G$-action. Our goal is to extend the theories of Satake \cite{Satake}
and Borel--Tits \cite{BorelTits} from reductive groups to
\emph{actions by} reductive groups.

The guiding philosophy behind the work of Satake and Borel--Tits is
that a reductive $k$-group is built from its minimal parabolic
$P=MAN$, its anisotropic kernel $M$ (treated as a black box), the
maximal $k$-split subtorus $A$ of $G$, and the restricted root system
$\Phi_k(G)$ on $A$.

In the same spirit, to any $G$-variety $X$ we attach:
\begin{itemize}[noitemsep]

\item a $k$-parabolic $Q_k(X)$ (generalizing $P$) with a Levi
  complement $L$,

\item an $L$-variety $X\an$ (generalizing the anisotropic kernel $M$),

\item a split $k$-torus $A_k(X)$ that is a quotient of $L$
  (generalizing $A$), and

\item a root system $\Phi_k(X)$ on $A_k(X)$ (generalizing
  $\Phi_k(G)$).

\end{itemize}

The $L$-action on $X\an$ factors through an anisotropic quotient of
$L$, which justifies calling $X\an$ the \emph{anisotropic kernel of
  $X$}. Moreover, the Levi part $MA$ of $P$ is generalized not by $L$
but by an $L$-variety $X\el$, which is a principal $A_k(X)$-bundle
over $X\an$; we call it the \emph{elementary kernel of $X$}. The
character group $\Xi_k(X):=\Xi(A_k(X))$ replaces the weight lattice of
$G$; accordingly, the \emph{$k$-rank of $X$} is defined as
\[
  \rk_k X:=\rk\Xi_k(X)=\dim A_k(X)=\dim X\el-\dim X\an.
\]

While the quintuple $(G,Q_k(X),X\an,A_k(X),\Phi_k(X))$ (which
generalizes the Tits index of $G$) does not completely characterize
$X$, it encodes substantial information. This is especially true for
\emph{$k$-spherical varieties}, i.e., those for which the minimal
parabolic $P$ of $G$ has an open dense orbit, or, equivalently, those
for which $L$ acts transitively on $X\an$.

What can be gleaned from the Tits index is, first and foremost, the
asymptotic structure of $X$, namely, its equivariant embeddings
$X\into\Xq$ defined over $k$. This generalizes parts of the theory of
equivariant embeddings over algebraically closed fields, where some of
the milestones were the theory of toroidal embeddings by Kempf et al.\
\cite{KempfTE}, the wonderful embeddings of symmetric varieties by De
Concini--Procesi \cite{DeConciniProcesi}, the groundbreaking work
of Luna and Vust \cite{LunaVust} classifying, in particular, spherical
embeddings, and Brion's discovery \cite{BrionSymetrique} that the
asymptotic behavior of spherical varieties is governed by a reflection
group. The latter was further developed and generalized in
\cite{KnopAB} via momentum map techniques; these methods are central
to this paper as well, which is the reason for our restriction to
characteristic zero.

Over non-algebraically closed fields, much less is known, with the
notable exception of Satake's early work \cite{SatakeComp} on the
compactification of symmetric spaces (where $k=\RR$). Another
important exception is the case of $k$-split groups $G$. In that case,
most of the theory over algebraically closed fields extends readily
and underlies the work of Sakellaridis and Venkatesh
\cite{SakellaridisVenkatesh} on the Plancherel theorem for spherical
varieties of split groups over $p$-adic fields. In fact, that work was
the principal motivation for this paper. Notably, a first step in this
direction was taken for $k=\RR$ in \cite{KKS}, where an $\RR$-version
of the Local Structure Theorem of \cite{BrionLunaVust} was proved and
used to study the geometry of $\RR$-spherical varieties. This became
the basis for our investigation of harmonic analysis on
$\RR$-spherical varieties \cites{KKSS,KKSSvolume,KKStemp,DKKS,GS}.

Because our approach focuses on rationality, our theory only makes
sense if there are enough $k$-rational points, i.e., if $X(k)$ is
Zariski dense in $X$, a property we call being \emph{$k$-dense}, which
we assume henceforth. The same applies to embeddings; we identify
embeddings that become isomorphic after removing closed, non-$k$-dense
$k$-subvarieties. For example, in this setting, orbits of connected
anisotropic groups do not admit any proper embeddings (see
\cref{cor:rational} for a more general statement).

Here is a short outline of the construction of the generalized
index. Let $K$ be an algebraic closure of $k$. Birationally, the
anisotropic and elementary kernels are determined by
$K(X\an)=K(X)^{AN}$ ($AN$-invariants) and $K(X\el)=K(X)^{N}$
($N$-invariants), respectively. The root system is more complex. It is
constructed using the set $\cZ_k(X)$ of $G$-invariant $\QQ$-valued
$K$-valuations of $K(X)$ that vanish on $K(X)^{AN}$. We show that
$\cZ_k(X)$ can be identified with a convex cone in
$\Hom_\ZZ(\Xi_k(X),\QQ)$ (see Sections \ref{sec:quasielementary} and
\ref{sec:ValGeo}). The generalization of Brion's theorem states that
$\cZ_k(X)$ is a Weyl chamber of a root system $\Phi_k(X)$ (see
Sections \ref{sec:WeylGroup} and \ref{sec:RootSystem}).

Our main tool is a $k$-version of the Local Structure Theorem of
Brion--Luna--Vust in \cref{sec:LST}. Here, we mainly follow the
approach of \cite{KKS}, which is in turn an adaptation of the proof in
\cite{KnopAB}. The main difficulty lies in producing sufficiently many
$P$-eigenfunctions on $X$. To this end, we replace the integration and
positivity argument of \cite{KKS} with an argument using Kempf's
theory of optimal one-parameter subgroups \cite{Kempf} (see
\cref{sec:quasielementary}).

Although our primary applications concern $k$-spherical varieties, the
theory does not simplify substantially when restricting to the
spherical case: a $k$-spherical variety need not remain spherical
after base change to $\overline{k}$. Consequently, the embedding
theory of non-spherical varieties is required to a considerable
extent.

As a first application, we construct a family of equivariant
embeddings, called \emph{toroidal}, in \cref{sec:Toroidal}. Each such
embedding is determined by a fan $\cF$ supported in $\cZ_k(X)$. In
general, there is no distinguished fan, except when $\cZ_k(X)$ is
strictly convex. This condition is equivalent to the root lattice
and the weight lattice of $X$ having the same rank. Since a similar
condition characterizes semisimple groups, we call varieties satisfying it
\emph{$k$-semisimple}. Thus, each $k$-semisimple $G$-variety has a
distinguished embedding, called the \emph{standard embedding $X\st$},
obtained by choosing the fan $\cF\st$ of faces of $\cZ_k(X)$.

In \cref{sec:horospherical}, we classify two families of
varieties. First, those of $k$-rank zero, which we show are all
parabolically induced from anisotropic actions; this is relevant
because the closed stratum of the standard embedding has rank
zero. Second, we classify $k$-horospherical varieties, characterized by
$\Phi_k(X)=\leer$; we show that they are parabolically induced from
their elementary kernel. Geometrically, they arise as maximal
degenerations of spherical varieties.

In \cref{sec:wonderful}, we investigate the smoothness of the standard
embedding $X\st$. We call varieties with this property
\emph{$k$-adjoint}, as it is equivalent to the weight lattice being
equal to the root lattice. Following \cite{KnopAut}, we relate
$k$-adjointness to the automorphism group of $X$. In particular, we
show that a sufficient condition for $X$ to be $k$-adjoint is that the
center of $\Aut^G(X)$ is trivial (see \cref{cor:centralunity}).

In \cref{sec:BoundaryDegenerations}, we construct smooth morphisms
whose generic fiber is a given $G$-variety $X$ and determine all
special fibers, the \emph{boundary degenerations} of $X$.

In \cref{sec:polar}, we extend the main result of \cite{KKSS} from
$\RR$ to local fields of characteristic zero, proving a weak polar
decomposition.

Finally, in \cref{sec:Satake}, we explain how the classical maximal
Satake compactification of a Riemannian symmetric space (a manifold
with corners) fits into our framework, and we compare the standard
embedding of $X$ over $k$ with the standard embedding over an
algebraic closure $K$.

\begin{ack}
  We would like to thank the referee of a previous version of this
  paper for catching an error in \cref{sec:polar}. We particularly
  thank the last referee for pointing out numerous mistakes and
  errors. Thanks are also due to Bart Van Steirteghem for many helpful
  remarks on the exposition.
  
\end{ack}


\section{Notation}\label{sec:notation}

Throughout this paper, $k$ denotes a field of characteristic zero. Let
$K$ be an algebraic closure of $k$, and let $\cG:=\|Gal|(K|k)$ be the
Galois group.

We regard a $k$-variety as a variety over $K$ equipped with a
compatible $\cG$-action. Accordingly, the terms ``connected'' and
``irreducible'' always mean ``absolutely connected'' and ``absolutely
irreducible.'' By definition, all $k$-varieties are non-empty and
irreducible.

Let $H$ be a linear algebraic group defined over $K$. Then
$\Xi(H):=\Hom(H,\G_m)$ denotes its character group, which we always
write additively; in particular, $0$ denotes the trivial character. The
value $\chi(h)\in K^*$ of a character $\chi$ at an element $h$ is, of
course, still written multiplicatively. If $H$ is defined
over $k$, then $\Xi(H)$ carries a $\cG$-action. The space
$\cN(H):=\Hom_\ZZ(\Xi(H),\QQ)$ denotes the space of rational cocharacters
of $H$.

Let $V$ be a representation of $H$ and $\chi\in\Xi(H)$. Then
$V^{(H)}_\chi$ denotes the set of $\chi$-eigenvectors in $V$. Note
that $V^{(H)}_\chi$ never contains the zero vector; thus, for the
trivial character $\chi=0$, we have $V^{(H)}_0=V^H\setminus\{0\}$. We
denote by $V^{(H)}$ the union of all $V^{(H)}_\chi$ for
$\chi\in\Xi(H)$. For $v\in V^{(H)}$, let $\chi_v\in\Xi(H)$ be the
unique character such that $v\in V^{(H)}_{\chi_v}$.

Throughout the paper, $G$ is a connected reductive $k$-group. We fix a
minimal parabolic $k$-subgroup $P\subseteq G$ with Levi decomposition
$P=MAN$ where $N=R_uP$ is the unipotent radical, $A$ is a maximal
$k$-split subtorus, and $M$ is connected and anisotropic. Thus, $MA$
is a Levi subgroup of $P$. Let $T\subseteq MA$ be a maximal
$k$-torus. Then $A\subseteq T$, and restriction to $A$ induces a
surjective homomorphism
\[
  \res_A:\Xi(T)\auf\Xi(A).
\]
If $\Omega\subseteq\Xi(T)$ is any subset, we set
\[\label{eq:resstrich}
  \res_A'\Omega:=(\res_A\Omega)\setminus\{0\}.
\]
Let $\Phi=\Phi(G,T)$ be the associated root system and
\[
  \Phi_k:=\res_A'\Phi
\]
the restricted root system (which is, in general, not reduced). The
Weyl group of $\Phi$ is denoted by $W=W(G)$. Choosing a Borel subgroup
$B$ such that $T\subseteq B\subseteq P$ yields a set of simple roots
$S\subseteq \Phi$ and a set of simple restricted roots
\[
  S_k:=\res_A'S\subseteq\Phi_k.
\]
On the other hand, the elements of
\[
  S^0:=\{\alpha\in S\mid\res_A\alpha=0\}
\]
are the compact simple roots; these are the simple roots of
$M$. The reflections corresponding to elements of $S^0$ generate the Weyl
group $W^0\subseteq W$ of $M$.

The Galois group $\cG$ acts on $\Xi(T)$, leaving the root system
$\Phi$ invariant, and hence it normalizes $W$. Here and in the sequel,
the action of an element $\gamma\in\cG$ on an object $x$ is denoted by ${}^\gamma x$. The restriction map
$\res_A$ is $\cG$-invariant. There is also the so-called $*$-action of
$\cG$ on $\Xi(T)$, constructed as follows: for each $\gamma\in\cG$,
there is a unique element $w_\gamma\in W$ such that
$w_\gamma\,{}^\gamma S=S$. Then one defines
\[\label{eq:staraction}
  \gamma*\chi:=w_\gamma\,{}^\gamma\chi.
\]
In what follows, we refer to the action of $\cG$ via the $*$-action as the
$\cG^*$-action.

It is known that $w_\gamma\in W^0$ for all $\gamma\in\cG$
(\Cite{BorelTits}{Proof of Prop.\ 6.7} or \Cite{Springer}{Proof of
  Prop.\ 15.5.3}). This means that if
\[
  \Delta^*:=\{\phi\in\Aut(\Xi(T))\mid\phi(S)=S,\ \phi(S^0)=S^0\}
\]
is the group of ``graph automorphisms'', then the $\cG^*$-action
induces a homomorphism $\cG\to\Delta^*$, whereas the ordinary
$\cG$-action lifts to $\Delta^*\semidir W^0$. Consequently, the
ordinary action and the $*$-action coincide on
$\Xi(T)^{W^0}=\Xi\bigl(T/(T\cap[M,M])^0\bigr)$.

Throughout this paper, we denote algebraic groups by uppercase Latin
letters and their corresponding Lie algebras by lowercase Fraktur
letters, e.g., $\fh$ is the Lie algebra of $H$. The unipotent radical
of an algebraic group $H$ is denoted by $R_uH$.

\section{Invariants of quasi-elementary
  groups}\label{sec:quasielementary}

Following Borel and Tits \Cite{BorelTits}{Def.\ 4.23}, we define:

\begin{definition}

  A reductive $k$-group $G$ is \emph{isotropic} if its $k$-rank is
  positive, i.e., if it contains a non-trivial split
  subtorus. Otherwise, it is called \emph{anisotropic}.

\end{definition}

According to \Cite{BorelTits}{Cor.\ 8.5}, an equivalent condition for
a connected group $G$ to be anisotropic is that all elements of $G(k)$
are semisimple and that $\Hom_k(G,\G_m)=0$. Note that some authors,
such as Springer in \Cite{Springer}{p.\ 271}, omit the second
condition. This alternative notion is also useful, and we define:

\begin{definition}

  A linear algebraic $k$-group $G$ is called \emph{elementary} if all
  of its $k$-rational elements are semisimple.

\end{definition}

Clearly, elementary groups are reductive. Moreover, anisotropic groups
are elementary. Indeed, if $G$ is anisotropic, then $G^0$ is
elementary by Borel--Tits; since in characteristic zero unipotent
elements are contained in $G^0$, it follows that $G$ is elementary as
well.

As with reductive groups, we do not assume anisotropic or elementary
groups to be connected. One reason is that the automorphism group of a
$k$-spherical variety might be disconnected and elementary (see
\cref{prop:anisoconvex}). In any case, a group is anisotropic or
elementary if and only if its identity component is. A further
advantage is that $k$-subgroups of anisotropic or elementary groups
are anisotropic or elementary, respectively. In particular, they are
reductive.

For connected groups, the difference between an anisotropic and an
elementary group is that the latter may have a non-trivial central
split subtorus. Because of this, connected elementary groups are often
called ``anisotropic modulo center.'' For example, the group
$G=GL(1,\HH)$ is elementary but isotropic as an $\RR$-group.  Over
$K$, the algebraic closure of $k$, the only connected anisotropic
group is the trivial group, whereas the elementary groups are
precisely the tori.

These definitions extend to non-reductive groups as follows:

\begin{definition}

  A $k$-group $H$ with unipotent radical $R_uH$ is
  \emph{quasi-anisotropic} or \emph{quasi-elementary} if $H/R_uH$ is
  anisotropic or elementary, respectively.

\end{definition}

Here and throughout the paper, a \emph{parabolic subgroup} of a
(possibly disconnected) $k$-group $H$ means a \emph{connected}
subgroup $Q\subseteq H$ such that $H/Q$ is complete, i.e., a parabolic
subgroup of $H^0$ in the usual sense. This convention differs from
\Cite{BorelTits}{4.1}, where a parabolic subgroup of a disconnected
group is not required to be connected. With this convention, since any
parabolic subgroup of $H$ contains $R_uH$, it follows from
\Cite{BorelTits}{Cor.\ 4.17} that $H$ is quasi-elementary if and only
if $H^0$ is its only parabolic $k$-subgroup. In particular, a minimal
parabolic $k$-subgroup of any group is quasi-elementary, and its Levi
component is elementary.

Every connected quasi-elementary group $H$ has a factorization
\[\label{quasi-factor}
  H=MAN,
\]
where $M$ is connected and anisotropic, $A$ is a split torus, and $N$
is its unipotent radical. This factorization is unique up to
conjugation by $N(k)$. The group $H$ is quasi-anisotropic, elementary,
or anisotropic if and only if $A$ is trivial, $N$ is trivial, or both
$A$ and $N$ are trivial, respectively.

For any connected group $H$, let $H\el\subseteq H$ be the Zariski
closure of the subgroup generated by all $k$-rational unipotent
elements. Then $H\el$ is the smallest normal subgroup of $H$ such that
$H/H\el$ is elementary. We call it the \emph{elementary radical of
  $H$}. Concretely, the elementary radical can be described as
follows: if $H=LR_uH$ is a Levi decomposition of $H$, then
$H\el=L_nR_uH$, where $L_n\subseteq L$ is the product of all isotropic
simple factors of $L$.

This has the following consequence: let $P=MAN\subseteq H$ be a
minimal parabolic $k$-subgroup. Since $MA$ contains all anisotropic
simple factors as well as the center of $L$, we have
\[\label{eq:H=PHel}
  H=MA\,H\el.
\]

There is an analogous notion for anisotropic groups: every group $H$
has a smallest normal subgroup $H\an$ such that $H/H\an$ is
anisotropic. We call $H\an$ the \emph{anisotropic radical of $H$}. As
a group, it is generated by $R_uH$ and all split subtori. If $H$ is
connected and $P=MAN$ is a minimal parabolic subgroup of $H$, then
\[\label{eq:H=PHan}
  H\an=AH\el\text{ and }H=MH\an.
\]

The characterization of quasi-elementary groups in terms of parabolic
subgroups can be extended as follows:

\begin{proposition}\label{prop:affineorbit}

  For a $k$-group $H$, the following are equivalent:

  \begin{enumerate}[noitemsep]

  \item\label{it:affineorbit1} $H$ is quasi-elementary.

  \item\label{it:affineorbit2} For any $k$-subgroup $F \subseteq H$,
    the quotient $H/F$ is an affine variety.

  \item\label{it:affineorbit3} The only parabolic $k$-subgroup of $H$
    is $H^0$.

  \end{enumerate}

\end{proposition}

\begin{proof}
  We may assume that $H$ is connected. Assume that $H=MAN$ is
  quasi-elementary and let $F\subseteq H$ be a $k$-subgroup. Consider
  the morphism
  \[
    \pi:H/F\to H/FN\cong(H/N)/(FN/N).
  \]
  The group $FN/N$ is a subgroup of the elementary group $H/N=MA$ and
  is therefore reductive. This implies that $H/FN$ is an affine
  variety. The fiber of $\pi$ is $FN/F\cong N/(F\cap N)$. It is an
  orbit of a unipotent group, hence affine as well. Therefore, $\pi$
  is an affine morphism to an affine variety; thus, $H/F$ is affine.

  Now assume \ref{it:affineorbit2}. Let $P\subseteq H$ be a minimal
  parabolic $k$-subgroup. Applying \ref{it:affineorbit2} to $F=P$
  implies that $H/P$ is both affine and complete, and therefore a
  point. Thus, $P=H$.

  Finally, the equivalence of \ref{it:affineorbit1} and
  \ref{it:affineorbit3} follows from \Cite{BorelTits}{Cor.\ 4.17}, as
  mentioned above.
\end{proof}

\begin{definition}

  A $k$-variety $X$ is called \emph{$k$-dense} if the set $X(k)$ of
  $k$-rational points is Zariski dense in $X$.

\end{definition}

Clearly, affine and projective spaces are $k$-dense. If $X\to Y$ is a
dominant $k$-morphism and $X$ is $k$-dense, then $Y$ is also
$k$-dense. The product $X\times Y$ of two $k$-dense varieties is
$k$-dense. More generally, if $X\to Y$ is a Zariski-locally trivial
fiber bundle with fiber $Z$, then $X$ is $k$-dense if and only if both
$Y$ and $Z$ are $k$-dense. This holds, in particular, for the total
space of a line bundle over a $k$-dense variety.

In the remainder of this section, we study actions of quasi-elementary
groups on (mostly) affine $k$-dense varieties. The point of departure
is the following fundamental theorem of Kempf:

\begin{theorem}[\Cite{Kempf}{Cor.\ 4.3}]\label{thm:Kempf}

  Let $G$ be a connected reductive $k$-group acting on an affine
  $k$-variety $X$. Let $x\in X(k)$ be a $k$-rational point and let
  $Y\subseteq X$ be a closed $G$-invariant $k$-subset with
  $\overline{Gx}\cap Y\ne\leer$. Then there exists a $k$-homomorphism
  $\lambda:\G_m\to G$ such that $\lim_{t\to0}\lambda(t)x$ exists and
  lies in $Y(k)$.

\end{theorem}

This implies:

\begin{corollary}[\Cite{Kempf}{Remark after Cor.\ 4.4}]
  \label{cor:Kempf}

  Let $G$ be an anisotropic group acting on a quasi-affine $k$-variety
  $X$. Then the $G$-orbit $Gx$ of any $k$-rational point $x\in X(k)$
  is closed in $X$.

\end{corollary}

\begin{proof}

  We may assume that $G$ is connected. Since $X$ is quasi-affine, it
  is isomorphic to a locally closed subvariety of an affine
  $G$-variety $\Xq$. Replacing $X$ by $\Xq$, we may assume that $X$ is
  affine. Now apply Kempf's theorem to $Y:=\overline{Gx}\setminus Gx$
  and observe that any homomorphism $\lambda:\G_m\to G$ is trivial
  since $G$ is anisotropic. We conclude that $Y$ is empty, i.e., $Gx$
  is closed in $X$.
\end{proof}

Rosenlicht proved in \cite{Ros2} that orbits of unipotent groups in
quasi-affine varieties are also closed. The following can be regarded
as a combination of both cases:

\begin{theorem}\label{MNquotient}

  Let $H=MAN$ be a connected quasi-elementary group acting on a
  $k$-dense quasi-affine variety $X$. Then there exists an affine
  $H$-variety $Y$ on which $MN$ acts trivially, an $H$-equivariant
  morphism $\pi:X\to Y$, and an $H$-stable open dense subset
  $Y_0\subseteq Y$ such that for all $y\in Y_0$, the fiber $X_y$ is an
  $MN$-orbit.

\end{theorem}

\begin{proof}
  Choose an open equivariant embedding of $X$ into an affine
  $H$-variety $\Xq$. Since it suffices to prove the assertion for
  $\Xq$, we may assume from the outset that $X$ is affine.

  Now choose generators $r_1,\ldots,r_n$ of $k(X)^N$, the field of
  $N$-invariant rational functions. Since $N$ is unipotent and the
  ideal
  \[
    \{f\in k[X]\mid fr_1,\ldots,fr_n\in k[X]\}
  \]
  is non-zero, it contains a non-zero $N$-invariant function $f_0$. By
  replacing $r_1,\ldots,r_n$ with $f_0,f_0r_1,\ldots,f_0r_n$, we may
  assume that each $r_i$ lies in $k[X]^N$.

  Let $V\subseteq k[X]^N$ be the (finite-dimensional) $H$-submodule of
  $k[X]^N$ generated by $\{r_1,\ldots,r_n\}$, let
  $\cA\subseteq k[X]^N$ be the (finitely generated) subalgebra
  generated by $V$, and let $Z:=\Spec \cA$ be the variety associated
  to $\cA$. Then $Z$ is an affine $H$-variety with $N$ acting
  trivially. Moreover, the inclusion $\cA\into k[X]$ induces a
  dominant $H$-morphism $\pi_N:X\to Z$. By construction, this morphism
  induces an isomorphism of fields
  $\pi_N^*:k(Z)\overset\sim\to k(X)^N$.  Rosenlicht's theorem
  \Cite{Rosenlicht}{Thm.~2} implies that the generic $N$-orbits are
  separated by rational $N$-invariants. This means that the generic
  fibers of $\pi_N$ contain an open dense $N$-orbit. Since orbits of a
  unipotent group on an affine variety are closed (by \cite{Ros2}, see
  also \Cite{Borel}{Prop.\ 4.10}), the generic fibers of $\pi_N$ are
  actually $N$-orbits.

  Now consider the $M$-action on $Z$ and let $\pi_M:Z\to Y:=Z\mod M$
  be the categorical quotient. Observe that $Y$ carries an action of
  $H$ with $MN$ acting trivially. All fibers of $\pi_M$ contain a
  unique closed $M$-orbit. It follows from \cref{cor:Kempf} that the
  orbits $Mz$ with $z\in Z(k)$ are closed. The $k$-density of $X$
  implies that $Z$ is also $k$-dense. This means that the generic
  $M$-orbits in $Z$ are closed. Thus, the generic fibers of $\pi_M$
  are precisely the generic $M$-orbits. Set
  $\pi:=\pi_M\circ\pi_N:X\to Y$.
\end{proof}

From this, we obtain the following generalization of \cref{cor:Kempf}:

\begin{corollary}

  Let $H$ be a quasi-anisotropic group acting on a quasi-affine
  $k$-variety $X$. Then:

  \begin{enumerate}

  \item\label{it:quasi1} The orbit $Hx$ of any $x\in X(k)$ is closed
    in $X$.

  \item\label{it:quasi2} Assume, moreover, that $X$ is $k$-dense. Then
    the generic $H$-orbits are closed in $X$, i.e., there is an
    $H$-stable open dense subset $X_0\subseteq X$ such that all
    $H$-orbits in $X_0$ are closed in $X$.
  \end{enumerate}

\end{corollary}

\begin{proof}
  \ref{it:quasi1} It is well known that $H(k)$ is dense in $H$
  (\Cite{Borel}{Thm.~18.3}). Therefore, the orbit closure
  $\overline{Hx}$ is $k$-dense. Applying \cref{MNquotient} to
  $\overline{Hx}$, the variety $Y$ is necessarily a point. Thus,
  $\overline{Hx}$ is an $H$-orbit, i.e., $Hx$ is closed.

  \ref{it:quasi2} Since fibers are closed, this follows directly from
  the theorem.
\end{proof}

Next, we establish a criterion for the existence of semiinvariants.

\begin{proposition}\label{SemiinvIdeal}

  Let $H$ be a connected quasi-elementary group acting on a $k$-dense
  affine variety $X$. Let $I\subseteq K[X]$ be a non-zero $H$-stable
  ideal (not necessarily defined over $k$). Then $I$ contains an
  $H$-semiinvariant function defined over $k$.

\end{proposition}

\begin{proof}
  Let $H=MAN$ be the factorization from \eqref{quasi-factor} and let
  $\pi:X\to Y$ be a morphism as in \cref{MNquotient}. The zero locus
  $Z$ of $I$ in $X$ is an $H$-stable proper closed subset. Its image
  $\pi(Z)$ in $Y$ cannot be dense since otherwise $Z$ would contain
  the generic $MN$-orbits and would therefore coincide with $X$. The
  ideal of all functions on $Y$ which vanish on $\pi(Z)$ is non-zero
  and therefore contains an $A$-semiinvariant function $f_0$. Then
  $f_0$ is also $H$-semiinvariant since $MN$ acts trivially on
  $Y$. Now replace $f_0$ by the (finite) product of all functions of
  the form ${}^\gamma f_0$, where $\gamma$ runs through the Galois
  group $\cG$. Then $f_0$ is defined over $k$ as well. The pull-back
  $f_1$ of $f_0$ to $X$ vanishes on $Z$. Hilbert's Nullstellensatz
  asserts that some power $f$ of $f_1$ is in $I$.
\end{proof}

\begin{corollary}\label{ExistPsemi}

  Let $H=MAN$ be a connected quasi-elementary group acting on a
  $k$-dense affine variety $X$. Then either $MN$ acts transitively on
  $X$ or $X$ carries a non-constant $H$-semiinvariant function defined
  over $k$.
\end{corollary}

\begin{proof}
  If $MN$ acts transitively, then all $H$-semiinvariants, being
  $MN$-invariant, are constant. Assume, conversely, that $MN$ does not
  act transitively, and let $Y\subsetneq X$ be a closed orbit. Now,
  apply \cref{SemiinvIdeal} to the vanishing ideal of $Y$.
\end{proof}

\begin{corollary}\label{rationalPinvariant}

  Let $H$ be a connected quasi-elementary group acting on a $k$-dense
  affine variety $X$. Then every $H$-semiinvariant rational function
  over $k$ is the quotient of two $H$-semiinvariant regular functions
  over $k$.

\end{corollary}

\begin{proof}
  Let $f\in k(X)^{(H)}$. Apply \cref{SemiinvIdeal} to the ideal
  $I=\{h\in k[X]\mid hf\in k[X]\}$.
\end{proof}

For reductive groups, one can prove the following refinement of
\cref{SemiinvIdeal}.

\begin{proposition}\label{SemiinvIdealElementary}

  Let $G$ be a connected reductive $k$-group acting on an affine
  $k$-variety $X$. Let $I\subseteq K[X]$ be a non-zero $G$-stable
  ideal and let $x\in X(k)$ be a point that is fixed by $G\el$ and
  is not contained in the zero locus of $I$. Then $I$ contains a
  $G$-semiinvariant function $f\in k[X]$ with $f(x)\ne0$.

\end{proposition}

\begin{proof}
  The group $G$ decomposes uniquely as $G=MAG\el$, where $M$ is
  anisotropic and $A$ is a split torus. Set $H:=MG\el$. Then $Hx=Mx$
  is closed by \cref{cor:Kempf}. Moreover, it is disjoint from $Z$,
  the zero locus of $I$, since $x\not\in Z$. Thus, there is a function
  $F\in K[X]$ which vanishes on $Z$ and which is non-zero at
  $x$. Replacing $F$ by the product of its Galois translates, we may
  assume that $F$ is defined over $k$. Replacing $F$ by a suitable
  power, we may further assume that $F\in I$. Since $A$ is a split
  torus and $I$ is $A$-stable, $F$ decomposes as a sum of
  $A$-semiinvariants $f_{\chi_i}\in I$. At least one of the
  $f_{\chi_i}$ does not vanish at $x$, which then has all the asserted
  properties.
\end{proof}

If $P$ is a minimal parabolic subgroup of $G$, then $P$ is connected
and quasi-elementary, and all the results above apply to actions of
$P$. In case the $P$-action actually arises from a $G$-action, one can
be more specific about the character of a semiinvariant.

\begin{lemma}\label{lemma:MAsemi}

  Let $G$ be a connected reductive $k$-group acting on an affine
  $k$-dense variety $X$, and let $P=MAN$ be a minimal parabolic
  $k$-subgroup of $G$. Let $\chi\in\Xi(A)$ be the character of an
  $AN$-semiinvariant function $f\in k[X]$. Then $n\chi$ is, for some
  $n\in\ZZ_{>0}$, the character of a $P$-semiinvariant function
  $\fq\in k[X]$.

\end{lemma}

\begin{proof}
  We start with a couple of reduction steps. Let $V_0\subseteq k[X]$
  be the $AN$-eigenspace for the character $\chi$. Because $V_0$ is
  $MA$-stable and non-zero by assumption, it contains an irreducible
  $MA$-submodule $V_1$. Then the $G$-module
  $V:=\<G\cdot V_1\>\subseteq k[X]$ generated by $V_1$ is
  irreducible. The inclusion $V\into k[X]$ yields a $G$-morphism
  $X\to V^*$ (where we consider $V^*$ as an affine $k$-variety). Let
  $X_1$ be the closure of its image, and let $\ell:V^*\to k$ be the
  evaluation at any fixed $0\ne f\in V_1$. Then $\ell$ is a linear
  $AN$-semiinvariant on $V^*$ with character $\chi$, which is non-zero
  on $X_1$. Clearly, it suffices to prove the assertion for $X_1$.

  Now, let $\Xq:=\overline{\G_mX_1}\subseteq V^*$ be the closed cone
  generated by $X_1$. It carries an action of $\Gq=\G_m\times G$. Then
  $\Pq=M\,\Aq\, N$ is a minimal parabolic subgroup of $\Gq$, where
  $\Aq=\G_m\times A$. Moreover, $\ell$ is an $\Aq N$-semiinvariant for
  the character $\overline\chi=(1,\chi)$. Suppose the assertion is
  true for $\Xq$. Then $\Xq$ carries a $\Pq$-semiinvariant $\fq$ with
  character $n(1,\chi)=(n,n\chi)$. The restriction $f_1$ of $\fq$ to
  $X_1$ is non-zero because $\fq\ne0$ is homogeneous and $\G_mX_1$ is
  dense in $\Xq$. Then $f_1$ is the desired semiinvariant.

  Replacing $X$, $G$ by $\Xq$, $\Gq$, we are reduced to the situation
  where $X$ is a closed cone in an irreducible $G$-module $V^*$, that
  $G$ contains all scalars, and that $f$ is the restriction of a
  linear function $\ell$ on $V^*$.

  Let $N^-\subseteq G$ be the unipotent subgroup opposite to $N$. Then
  $PN^-=MANN^-=NMAN^-$ is a dense open subset of $G$. Let
  $\P(X)\subseteq\P(V^*)$ be the projective variety corresponding to
  $X$. Since $AN^-$ is a split solvable group and $X$ is $k$-dense, it
  has a fixed $k$-rational point in $\P(X)$, which means that $X$
  contains an $AN^-$-stable line $k\cdot w\ne0$. Suppose that
  $\ell(Mw)=0$. Then $\ell(NMAN^-w)=0$, and therefore $\ell(Gw)=0$,
  which is impossible since $V^*$ is irreducible. Thus, there exists
  $g\in M(k)$ with $\ell(gw)\ne0$. Replacing $w$ by $gw$, we may
  assume that $\ell(w)\ne0$. Then $A$ acts on $kw$ via the character
  $-\chi$.

  Let $X_2:=\overline{Gw}\subseteq X$, which is a cone since $G$
  contains the scalars. Because $0\in X_2$, the group $MN$ cannot act
  transitively on $X_2$. Thus, \cref{ExistPsemi} implies that $X_2$
  carries a non-constant $P$-semiinvariant $\fS$. Suppose
  $\fS(w)=0$. Then from $\fS(PN^-w)=0$ we get $\fS(Gw)=0$ and
  therefore $\fS=0$, a contradiction. Thus, the restriction of $\fS$
  to $kw$ is an $A$-semiinvariant that is homogeneous of degree
  $n>0$. Its character therefore equals $n\chi$. Finally, since the
  restriction map $k[X]\to k[X_2]$ splits as a $G$-homomorphism, the
  function $\fS$ extends to a $P$-semiinvariant $\fq$ on $X$ with the
  same character.
\end{proof}

\begin{remark}

  This lemma replaces the integration argument in \Cite{KKS}{Thm.\
    2.2} (Local Structure Theorem) for the case $k=\RR$.

\end{remark}

As mentioned in the introduction, the main philosophy of this paper is
to extend the Borel--Tits theory of reductive groups to actions of
reductive groups on varieties. We begin by defining the analogues of
elementary and anisotropic groups.

\begin{definition}

  Let $H$ be a $k$-group acting on a $k$-dense variety $X$. The action
  (or, by abuse of language, $X$) is called \emph{elementary} or
  \emph{anisotropic} if the elementary radical $H\el$ or the
  anisotropic radical $H\an$, respectively, acts trivially on $X$.

\end{definition}

Clearly, the action is elementary if and only if all $k$-rational
unipotent elements of $H$ act trivially on $X$. Likewise, the action
is anisotropic if $R_uH$ and all split subtori of $H$ act
trivially. Let $P=MAN\subseteq H$ be a minimal parabolic subgroup. It
follows from \eqref{eq:H=PHel} or \eqref{eq:H=PHan} that an elementary
$H$-action is the same as an $MA$-action that is trivial on
$MA\cap H\el$. Similarly, \eqref{eq:H=PHan} implies that an
anisotropic $H$-action is the same as an $M$-action that is trivial on
$M\cap H\an$. Observe that every anisotropic action is elementary.

Next, we provide a representation-theoretic characterization of
elementary actions, which is key to the proof of the Local Structure
Theorem \ref{kLST}.

\begin{proposition}\label{cor:CharacterizationElementaryAction}

  Assume that the action of the connected reductive group $G$ on the
  $k$-dense quasi-affine variety $X$ is \emph{not} elementary. Then
  there is a $P$-semiinvariant $f\in k[X]$ whose character $\chi_f$
  does not extend to a character of $G$.

\end{proposition}

\begin{proof}
  According to \Cite{BorelTits}{Prop.\ 8.4}, every $k$-rational
  unipotent element is conjugate to an element of $N$. This implies
  that the action of $N$ on $X$ is non-trivial. Therefore, there is a
  simple $G$-submodule $V$ of $k[X]$ with a non-trivial
  $N$-action. For a character $\chi$ of $A$, let $V_\chi\subseteq V$
  be the corresponding weight space. Choose $\chi$ such that
  $V_\chi^N=V_\chi\cap V^N\ne0$. Observe that $V_\chi^N$ is
  $M$-stable.

  The non-triviality of the $N$-action on $V$, the
  Poincaré--Birkhoff--Witt theorem, and the simplicity of $V$ imply
  that
  \[\label{eq:VVUV}
    V_\chi^N\ne V=U(\fn^-)V_\chi^N.
  \]
  Hence, in particular, $V_\chi^N=V_\chi$.

  Let $w_0$ be the longest element of the restricted Weyl group of
  $G$. Because $w_0\fn^-=\fn$, applying $w_0$ to \eqref{eq:VVUV}
  yields $w_0V_\chi\ne V_\chi$. Thus, there is a simple $k$-root
  $\alpha$ such that $V_{s_\alpha\chi}=s_\alpha V_\chi\ne
  V_\chi$. This is equivalent to $\<\chi,\alpha^\vee\>\ne 0$.

  Applying \cref{lemma:MAsemi} to any $0\ne f\in V_\chi^N$ yields a
  $P$-semiinvariant $\fq\in V\subseteq k[X]$ whose character is
  $\chi_\fq=n\chi$ for some $n\in\ZZ_{>0}$. Then
  $\<\chi_\fq,\alpha^\vee\>\ne0$, which means that $\chi_\fq$ cannot
  be the restriction of a character of $G$ to $A$.
\end{proof}

\begin{remark}

  Using \eqref{eq:H=PHel}, it is easy to see that the converse of the
  \lcnamecref{cor:CharacterizationElementaryAction} also holds.

\end{remark}

\section{The Local Structure Theorem}\label{sec:LST}

Before stating and proving the Local Structure Theorem, we need to
discuss a regularity property for algebraic group actions.

\begin{definition}

  Let $H$ be a linear algebraic $k$-group acting on a $k$-variety
  $X$. Then $X$ is called \emph{locally $k$-linear} if for every
  $x\in X(k)$, there is an $H$-stable open neighborhood $X_0$ of $x$,
  a finite-dimensional representation $V$ of $H$, and an
  $H$-equivariant embedding $X_0\into\P(V)$, where everything is
  defined over $k$.

\end{definition}

If $H$ is connected and $X$ is normal at every $k$-rational point,
then Sumihiro's Theorem \cite{Sumihiro} implies that $X$ is locally
$k$-linear. Thus, the condition is relatively mild, as at least the
normalization of any $H$-variety is locally $k$-linear. On the other
hand, local linearity possesses better functorial properties than
normality, since $H$-stable subvarieties of locally $k$-linear
varieties are themselves locally $k$-linear. In particular, all
$H$-stable subvarieties of $\P(V)$ are locally $k$-linear.

We proceed with the main tool of this paper. Its purpose is to reduce
arbitrary $G$-actions to elementary ones.

\begin{theorem}[The $k$-Local Structure Theorem]\label{kLST}

  Let $G$ be a connected reductive $k$-group, let $X$ be a locally
  $k$-linear $G$-variety, and let $Y\subseteq X$ be a $G$-stable,
  $k$-dense subvariety. Then there exists a parabolic $k$-subgroup
  $Q\subseteq G$ with a Levi decomposition $Q=LU$ (where $U=R_uQ$),
  and an $L$-stable affine $k$-subvariety $R\subseteq X$ such that:

  \begin{enumerate}

  \item\label{it:slice1} The intersection $Y\el:=R\cap Y$ is a
    $k$-dense affine $L$-variety (and thus, in particular, non-empty),
    all $L$-orbits in $Y\el$ are closed, and the action of $L$ on
    $Y\el$ is elementary.

  \item\label{it:slice2} The morphism
    \[
      U\times R=Q\times^LR\to X:[g,x]\mapsto gx
    \]
    is an open embedding.

  \end{enumerate}

\end{theorem}

For the proof, we need the following lemma:

\begin{lemma}

  Let $X_0\subseteq X$ be a $G$-stable open $k$-subvariety with
  $Y_0:=Y\cap X_0\ne\leer$. Then the $k$-Local Structure Theorem holds
  for $(X,Y)$ if and only if it holds for $(X_0,Y_0)$.

\end{lemma}

\begin{proof}
  One direction is trivial: if the $k$-LST holds for $(X_0,Y_0)$, then
  it holds for $(X,Y)$. Conversely, assume that it holds for $(X,Y)$,
  and let $R\subseteq X$ be the $L$-stable slice. Set $R':=R\cap X_0$.
  Because $U(R\cap Y)$ is dense and open in $Y$, we see that $R$ meets
  every $U$-stable open subset of $Y$. In particular,
  $R'\cap Y_0\ne\leer$. One readily sees that $R'\subseteq X_0$
  inherits all requirements for a slice from $R\subseteq X$ except
  possibly the requirement that $R'$ be affine. To satisfy this final
  property, it suffices to construct an affine $L$-stable open subset
  $R_0\subseteq R'$ with $R_0\cap Y_0\ne\leer$.

  To this end, we choose a $k$-rational point $y\in R\cap Y_0$, which
  is possible since $R\cap Y$ is $k$-dense. Note that
  $y\not\in Z:=R\setminus R'$. Because the $L$-action on $Y\el$ is
  elementary, the point $y$ is fixed by $L\el$. Hence,
  \cref{SemiinvIdealElementary} yields an $L$-semiinvariant function
  $f$ on $R$ that vanishes on $Z$ and is non-zero at $y$. Now
  $R_0:=\{x\in R\mid f(x)\ne0\}\subseteq R'$ is an affine open
  subvariety with $R_0\cap Y_0\ne\leer$.
\end{proof}

\begin{proof}[Proof of the $k$-LST]
  We begin with several reductions. First, by the lemma and the
  assumption of local $k$-linearity, we may assume that
  $X\subseteq\P(V)$, where $V$ is a finite-dimensional
  $G$-module. Next, we apply the lemma to $X$ and its closure in
  $\P(V)$. Thus, we may assume that $X$ is closed in $\P(V)$.

  Now let $\XS\subseteq V$ and $\YS\subseteq V$ be the affine cones of
  $X$ and $Y$, respectively. Furthermore, set
  $\XS^*:=\XS\setminus\{0\}$ and $\YS^*:=\YS\setminus\{0\}$. All these
  varieties carry an action of $\GS:=G\times\G_m$. Assume that the
  $k$-LST is true for $(\GS,\XS,\YS)$. Then it holds for
  $(\GS,\XS^*,\YS^*)$ by the lemma. Moreover, $\XS^*\to X$ and
  $\YS^*\to Y$ are principal $\G_m$-bundles. It follows easily that
  all assertions of the theorem descend to $(G,X,Y)$. Thus, by
  replacing $(G,X,Y)$ with $(\GS,\XS,\YS)$, we may assume that $X$ is
  an affine $G$-variety.

  We proceed by induction on $\dim X$. Assume first that the
  $G$-action on $Y$ is elementary. Then $Q=L=G$ and $R=X$ possess all
  the required properties, except possibly the closedness of the
  $G$-orbits in $Y$. To ensure this, let $Z\subseteq Y$ be the set of
  all $y\in Y$ whose orbit $Gy$ is not of maximal dimension. This is a
  proper closed $G$-stable subset of $Y$. Moreover, all $G$-orbits in
  $Y\setminus Z$ have the same dimension and are thus closed. By
  \cref{SemiinvIdealElementary}, we can choose a $G$-semiinvariant
  function $\fq$ on $Y$ vanishing on $Z$. Since $k[Y]$ is, as a
  $G$-module, a direct summand of $k[X]$, the function $\fq$ extends
  to a $G$-semiinvariant function $f$ on $X$. Let $X_f$ be the
  non-vanishing set of $f$. Then $Y_f=Y\cap X_f\ne\leer$ and all of
  its $G$-orbits are closed. Thus, we can take $Q=G$ and $R=X_f$.

  Finally, assume that the action of $G$ on $Y$ is not
  elementary. Then \cref{cor:CharacterizationElementaryAction} implies
  the existence of a $P$-semiinvariant function $\fq\in k[Y]$ whose
  character does not extend to $G$. As above, $\fq$ extends to a
  $P$-semiinvariant function $f\in k[X]$. Let $X_f\subseteq X$ be its
  non-vanishing set and $Y_f:=X_f\cap Y=Y_\fq\ne\leer$. Let
  $\Qq:=\{g\in G\mid gf\in K^*f\}$ be the normalizer of $f$ in
  $G$. Since $f$ is a highest weight vector, $\Qq$ is a parabolic
  subgroup that acts on the line $kf$ by a dominant character
  $\chi_f$. Moreover, the parabolic $\Qq$ is a proper subgroup of $G$
  because $\chi_f$ does not extend to a character of $G$. Let
  $\Uq=R_u\Qq$ be its unipotent radical. From highest weight theory,
  we have
  \[\label{eq:chialpha}
    \<\chi_f,\alpha^\vee\>>0\text{ for all roots $\alpha$ in
      $\overline\fu:=\Lie\Uq$}.
  \]

  The rest of the argument is almost identical to the proof of
  \Cite{KnopAB}{Thm.\ 2.3}. On the $\Qq$-stable open set $X_f$ one can
  define the $\Qq$-equivariant morphism
  \[\label{eq:moment}
    m:X_f\to \fg^*:x\mapsto m_x,
  \]
  where
  \[
    m_x:\fg\to K:\xi\mapsto\frac{\xi f(x)}{f(x)}.
  \]
  Choose a $k$-rational point $y\in Y_f$ (possible due to $k$-density)
  and set $a:=m(y)\in\fg^*$. The $\Qq$-invariance of $f$ implies that
  the image of $m$ is contained in the affine subspace
  \[\label{eq:a+u}
    a+\overline\Fq^\perp=a+\overline\fu
  \]
  where we identify $\overline\Fq^\perp$ with $\overline\fu$ via an
  isomorphism $\fg^* \cong \fg$.

  It follows from \eqref{eq:chialpha} that the stabilizer of $a$ in
  $G$ is a Levi subgroup $\Lq$ of $\Qq$. We claim that
  $a+\overline\fu=\Qq a=\Qq/\Lq$ is a single $\Qq$-orbit. Indeed,
  since $\Uq\cap\Lq=1$, the orbit morphism $\Uq\to a+\overline\fu$ is
  injective and, for dimensional reasons, dominant. Thus, it is also
  surjective, as orbits of unipotent groups on affine varieties are
  closed.

  In summary, we obtain a $\Qq$-equivariant morphism
  $m:X_f\to\Qq/\Lq$. Defining $X':=m^{-1}(a)$, the morphism
  \[\label{eq:QXX}
    \Uq\times X'=\Qq\times^\Lq X'\overset\sim\to X_f:[q,x]\mapsto qx
  \]
  is an isomorphism.

  Since $\Uq\ne1$, we have $\dim X'<\dim X$. Therefore, we can apply
  the induction hypothesis to $\Lq$ acting on $X'$ and $Y':=X'\cap Y$.
  This yields a parabolic $k$-subgroup $Q'\subseteq\Lq$, a Levi
  decomposition $Q'=L'U'$, and an $L'$-stable affine slice
  $R\subseteq X'$ such that
  \[\label{eq:LSX}
    U'\times R=Q'\times^{L'}R\to X'
  \]
  is an open embedding. Setting
  $Q:=Q'\,\Uq\subseteq\Lq\,\Uq\subseteq G$, which is a parabolic
  subgroup of $G$, and combining \eqref{eq:QXX} and \eqref{eq:LSX}, we
  find that
  \[
    Q\times^LR=R_uQ\times R=\Uq\times U'\times R\into\Uq\times X'\into
    X
  \]
  is an open embedding.
\end{proof}

Since $Q\times^LR=U\times R$ is affine, we obtain the following:

\begin{corollary}\label{cor:P-stable-affine}

  Let $Y$ be a $G$-stable $k$-dense subvariety of a locally $k$-linear
  $G$-variety $X$. Then there is an affine open $k$-subvariety
  $X_0\subseteq X$ that is stable under the minimal parabolic $P$ and
  satisfies $X_0\cap Y\ne\leer$.

\end{corollary}

Although the slice $R$ is not unique, it satisfies certain uniqueness
properties. For any $k$-dense $G$-variety $X$, define
\[
  Q_k(X):=\{g\in G\mid gPx=Px\text{ for $x$ in a dense open subset of
    $X$}\}.
\]
Since $P\subseteq Q_k(X)$, this is a parabolic $k$-subgroup.

\begin{proposition}\label{prop:uniqueQ}

  In the setting of \cref{kLST}, assume without loss of generality
  that $P\subseteq Q$. Then $Q=Q_k(Y)$. Furthermore, as an
  $L$-variety, the slice $R$ is unique up to a unique $L$-equivariant
  birational isomorphism that is regular at a generic point of $Y\el$.

\end{proposition}

\begin{proof}
  Set $\Qq=Q_k(Y)$. From \eqref{eq:H=PHel} and the $k$-LST, the
  generic $P$-orbits and $Q$-orbits on $Y$ coincide, showing
  $Q\subseteq\Qq$.

  Conversely, let $y\in Y\el$ be generic. Then $\Qq y=Py$ by
  assumption and $Py=Qy$ by the $k$-LST. Thus, $\Qq/\Qq_y=Q/Q_y$,
  which implies $Q_y=Q\cap\Qq_y$ and $\Qq=Q\Qq_y$. This in turn
  implies $\Qq_y/Q_y=\Qq/Q$. Since the $L$-action on $Y\el$ is
  elementary, $Q_y=L_y$ is reductive, so $\Qq_y/Q_y$ is an affine
  variety. On the other hand, $\Qq/Q$ is projective, which implies
  $\Qq=Q$.

  Finally, let $X_0=QR\subseteq X$. This is an open, dense, $Q$-stable
  subset with $X_0\cap Y\ne\leer$ such that the orbit space $X_0/U$
  exists. Moreover, $R$ is isomorphic to this quotient. If
  $X_1\subseteq X$ is another such set, then $R$ is birationally
  equivalent to $(X_0\cap X_1)/U$, and therefore to $X_1/U$.
\end{proof}

One of the most important special cases is that of $Y=X$. We formulate
it explicitly:

\begin{corollary}[Generic Structure Theorem]\label{cor:LSTgeneric}

  Let $X$ be a $k$-dense $G$-variety and let $Q_k(X)=LU$ be a Levi
  decomposition. Then there exists a smooth affine $L$-subvariety
  $X\el\subseteq X$ such that:

  \begin{enumerate}

  \item The action of $L$ on $X\el$ is elementary, all orbits are
    closed, and the categorical quotient $X\el\to X\el\mod L$ is a
    locally trivial fiber bundle in the étale topology.

  \item The natural morphism $U\times X\el=Q_k(X)\times^LX\el\to X$ is
    an open embedding.

  \end{enumerate}

  The slice $X\el$ is unique up to a unique $L$-equivariant birational
  isomorphism. Specifically, its field of rational functions is given
  by (where $N=R_uP$)
  \[\label{eq:kXN}
    k(X\el)=k(X)^U=k(X)^N.
  \]

\end{corollary}

\begin{proof}
  Apply \cref{kLST} to the smooth part of $X$, which is locally
  $k$-linear by Sumihiro's Theorem. This proves everything except the
  local triviality statement and the second equality in
  \eqref{eq:kXN}. The former is a consequence of Luna's slice theorem
  (see \Cite{LunaSlice}{III.2, Cor.~5}). The latter follows because
  $N=(L\cap N)U$ and $L\cap N$ acts trivially on $X\el$; thus, the
  generic $N$- and $U$-orbits coincide on $X$.
\end{proof}

Here and in the sequel, the Levi subgroup $L$ of $Q_k(X)$ is always
chosen such that $MA\subseteq L$. This is possible since
$P\subseteq Q_k(X)$, and it makes $X\el$ an $MA$-variety. Due to the
factorization \eqref{eq:H=PHel}, no information is lost by considering
$X\el$ as an $MA$-variety. Since all $A$-orbits in $X\el$ are closed,
we may form the orbit space.

\begin{definition}

  The quotient $X\an:=X\el\mod A$ is called the \emph{anisotropic
    kernel of $X$}.

\end{definition}

By construction, the anisotropic kernel is a $k$-dense affine
$M$-variety whose orbits are all closed and of the same dimension. In
particular, all isotropy groups are reductive and, over $K$, conjugate
to each other. In our theory, $X\an$ plays a role similar to the
anisotropic kernel of a reductive group in Borel--Tits theory
\cite{BorelTits}.

From \eqref{eq:kXN}, we have $k(X\an)=k(X)^{AN}$ and thus
$k(X\an)^M=k(X)^P$, providing another view of the anisotropic kernel
as a birational invariant.

We now introduce an important numerical invariant. Recall that the
\emph{cohomogeneity} of an action is the codimension of a generic
orbit.

\begin{definition}

  The \emph{$k$-complexity of $X$}, denoted $c_k(X)$, is the
  cohomogeneity of the $P$-action on $X$. It equals the transcendence
  degree of $k(X)^P$ (by Rosenlicht), the transcendence degree of
  $k(X\an)^M$, and the cohomogeneity of the $M$-action on $X\an$.

\end{definition}

The most important special case is:

\begin{definition}

  A locally $k$-linear, $k$-dense $G$-variety $X$ with $c_k(X)=0$ is
  called \emph{$k$-spherical}. Equivalently, $X$ is $k$-spherical if
  it is locally $k$-linear, $k$-dense, and $P$ has an open orbit.

\end{definition}

Another characterization of $k$-spherical varieties is that $M$ acts
transitively on the anisotropic kernel. In this case, $X\an$ is unique
up to a biregular isomorphism.

\begin{remarks}
\emph{i)} Let $B\subseteq P$ be a Borel subgroup. Then $X$ is
    $K$-spherical if $B$ has a dense orbit in $X$ (spherical in the
    usual sense). We call such $X$ \emph{absolutely spherical}. Key
    examples of absolutely spherical varieties for $k=\RR$ are
    symmetric spaces $G/K$, where $K(\RR)$ is a maximal compact
    subgroup of $G(\RR)$. On the other hand, the triple spaces
    $SO(1,n)^3/SO(1,n)$ are $\RR$-spherical but not absolutely
    spherical for $n\ge4$.

\emph{ii)} A $k$-variety is absolutely spherical if and only if its
    anisotropic kernel is. For $k=\RR$, anisotropic absolutely
    spherical varieties were classified by Krämer \cite{Kraemer},
    Mikityuk \cite{Mikityuk}, and Brion \cite{BrionClassification}.

\emph{iii)} If $G$ is quasi-split, then $P=B$, and every $k$-spherical
    variety is absolutely spherical. Otherwise, $X=G/AN$ is an
    example of a $k$-spherical variety that is not absolutely
    spherical; its anisotropic kernel is $MA/A=M/M\cap A$.
\end{remarks}

For quasi-affine varieties, we have the following generalization of
the representation-theoretic criterion by Vinberg--Kimel'feld
\cite{VinbergKimelfeld}.

\begin{proposition}

  For a quasi-affine, $k$-dense $G$-variety $X$, the following are
  equivalent:

  \begin{enumerate}

  \item $X$ is $k$-spherical.

  \item For every dominant character $\chi$ of $P$, the simple
    $G$-module with highest weight $\chi$ appears with multiplicity at
    most one in $K[X]$.

  \end{enumerate}

\end{proposition}

\begin{proof}
  Let $M_1, M_2\subseteq K[X]$ be two distinct simple $G$-submodules
  with highest weight $\chi$. Let $f_i\in M_i$ be a highest weight
  vector. Then $f=\frac{f_1}{f_2}$ is a non-constant $P$-invariant
  rational function on $X$, so $X$ cannot have an open $P$-orbit.

  Conversely, if $X$ does not have an open $P$-orbit, there exists a
  non-constant $P$-invariant rational function $f$ on $X$ (by
  Rosenlicht). By \cref{rationalPinvariant}, there are
  $P$-semiinvariants $f_1,f_2\in k[X]$ with $f=\frac{f_1}{f_2}$. Their
  characters are the same, say $\chi$, so the $G$-modules $M_i$
  generated by $f_i$ are distinct, isomorphic, irreducible submodules
  of $K[X]$.
\end{proof}

For $k$-spherical varieties, \cref{cor:LSTgeneric} becomes:

\begin{corollary}\label{cor:LSTspherical}

  Let $X$ be a $k$-spherical $G$-variety and $Q:=Q_k(X)=LU$ be a Levi
  decomposition. Then there is a point $x\in X(k)$ such that $Qx$ is
  open and affine in $X$, and the isotropy subgroup satisfies
  $L\el\subseteq Q_x\subseteq L$.

\end{corollary}

Next, we compare the complexity of a variety with that of its
subvarieties.

\begin{proposition}\label{prop:cYcX}

  Let $X$ be a locally $k$-linear, $k$-dense $G$-variety and let
  $Y\subseteq X$ be a $G$-stable $k$-dense subvariety. Then
  $c_k(Y)\le c_k(X)$ with equality if and only if
  $K(X)^P\subseteq\cO_{X,Y}$ (the local ring of $X$ at $Y$).

\end{proposition}

\begin{proof}
  Apply \cref{kLST} to $Y\subseteq X$. Since all $L$-orbits in $Y\el$
  are closed, the morphism
  \[\label{eq:YelRel}
    Y\el\mod L\to R\mod L
  \]
  is a closed embedding. Thus,
  \[
    c_k(Y)=c_k(Y\el)=\dim Y\el\mod L\le\dim R\mod L\le c_k(R)=c_k(X).
  \]
  In the case of equality, \eqref{eq:YelRel} is an isomorphism and
  \[
    K(X)^P=K(R)^{P\cap L}=K(R)^L=K(Y\el)^L=K(Y)^P\subseteq
    K(Y),
  \]
  hence $P$-invariant rational functions are regular in
  $Y$. Conversely, $K(X)^P\subseteq K(Y)$ clearly implies
  $c_k(Y)\ge\|trdeg|_K K(X)^P=c_k(X)$.
\end{proof}

\begin{corollary}\label{cor:sphericalfinite}

  Let $X$ be a $k$-spherical $G$-variety. Then every $k$-dense
  $G$-stable subvariety is $k$-spherical, and $X$ contains only
  finitely many such subvarieties. In particular, the number of
  $G$-orbits $Y$ with $Y(k)\ne\leer$ is finite.

\end{corollary}

\begin{proof}
  The first part follows from \cref{prop:cYcX}. If $X$ contained
  infinitely many $k$-dense subvarieties, a minimal closed $G$-stable
  subvariety $Y$ would exist with the same property. Since $Y$ is
  spherical, $P$, and therefore $G$, has a dense open orbit $Y^0$, and
  the irreducible components of $Y\setminus Y^0$ would yield smaller
  counterexamples.
\end{proof}

\begin{corollary}\label{cor:localfinite}

  If $k$ is a local field and $X$ is a $k$-spherical $G$-variety, then
  the number of $G(k)$-orbits in $X(k)$ is finite.

\end{corollary}

\begin{proof}
  This follows from \cref{cor:sphericalfinite} and
  \Cite{BorelSerre}{Cor.\ 6.4}: if $k$ is local of characteristic $0$,
  then for any homogeneous $G$-variety $Y$, the orbit set $Y(k)/G(k)$
  is finite.
\end{proof}

Next we discuss the relationship between $X\el$ and $X\an$. Let
$A_0\subseteq A$ be the kernel of the action of $A$ on $X\el$. Then
$A_k=A_k(X):=A/A_0$ acts freely on $X\el$, and the quotient
\[\label{eq:Abundle}
  \pi:X\el\to X\an
\]
is an $M$-equivariant principal $A_k$-bundle. The characters of $A_k$
are of the form $\chi_f$, where $f$ is an $A$-semiinvariant rational
function on $X\el$. The Local Structure Theorem implies that these $f$
correspond to $AN$-semiinvariant rational functions on $X$. Thus, we
obtain
\[\label{eq:defXik}
  \Xi_k(X):=\Xi(A_k(X))=\{\chi_f\in\Xi(A)\mid f\in K(X)^{(AN)}\}.
\]
Since $\chi_f=0$ is equivalent to $f$ being $AN$-invariant, we get the
short exact sequence
\[\label{eq:short-exact}
  1\to K(X\an)^*\to K(X)^{(AN)}\to\Xi_k(X)\to0.
\]

\begin{definition}

  The \emph{$k$-rank} of a $k$-dense $G$-variety $X$ is
  \[
    \rk_kX:=\rk\Xi_k(X)=\dim A_k(X)=\dim X\el-\dim X\an.
  \]

\end{definition}

Coming back to \eqref{eq:Abundle}, the algebra $k[X\el]$ decomposes as
an $A_k$-module into isotypic components
\[
  k[X\el]=\bigoplus_{\chi\in\Xi_k(X)}\cS_\chi.
\]
Each $\cS_\chi$ is an $M$-linearized invertible
$\cS_0=k[X\an]$-module, inducing a homomorphism
\[\label{eq:XiPic}
  \Xi_k(X)\to\|Pic|^M(X\an):\chi\mapsto [\cS_\chi].
\]
This homomorphism encodes the non-triviality of the bundle
\eqref{eq:Abundle}.

If $X$ is $k$-spherical, the situation simplifies. After fixing a
point $y\in X\an(k)$, we obtain $X\an\cong M/M_y$. Then, for any
$m\in M_y$ there is a unique $a\in A_k$ with $mx=ax$ for all
$x\in\pi^{-1}(y)$. This yields a homomorphism
\[
  \phi:M_y\to A_k:m\mapsto a
\]
defined over $k$. Using $\phi$, one can recover $X\el$ as
\[
  X\el\cong M\times^{M_y}A_k.
\]

\begin{remarks}

  \emph{i)} Since $M_y^0$ is connected and anisotropic, the
  homomorphism $\phi$ factors through the finite group
  $\pi_0(M_y)=M_y/M_y^0$. Thus, $\phi(M_y)$ is finite.

  \emph{ii)} In addition to $\Xi_k(X)$, one can consider the group
  \[
    \Xi_k'(X):=\{\chi_f\in\Xi(A)\mid f\in K(X)^{(P)}\}.
  \]
  It is the character group of the torus
  $A_k'(X):=A_k(X)/\phi(M_y)$. Thus, $\Xi_k'(X)$ is of finite index in
  $\Xi_k(X)$. In general, it is the kernel of \eqref{eq:XiPic}.

\end{remarks}

We conclude this section by stating a comparison result between the
theories over $k$ and $K$. Choose a maximal $k$-torus $T\subseteq
MA$. Since $A$ is the maximal split subtorus of $T$, we obtain a
restriction map $\res_A:\Xi_K(T)\to\Xi_k(A)$. Let $B$ be a Borel
subgroup of $P$ (and therefore of $G$) containing $T$. Unless $G$ is
quasi-split, $B$ is not defined over $k$.

\begin{lemma}\label{lemma:XiResXi}

  Let $X$ be a $k$-dense $G$-variety. Then
  \[\label{eq:resXiK}
    \Xi_k(X)=\res_A\Xi_K(X).
  \]

\end{lemma}

\begin{proof}
  The inclusion ``$\supseteq$'' follows from the fact that every
  $B$-semiinvariant is also an $AN$-semiinvariant for the restricted
  character. We show the opposite inclusion. Since
  $\Xi_k(X)=\Xi_k(X\el)$, it suffices to prove the statement for $MA$
  acting on $X\el$. By \cref{rationalPinvariant}, every rational
  $A$-semiinvariant on $X\el$ is the quotient of two regular
  $A$-semiinvariants. Thus, it suffices to show that for every
  $f\in K[X\el]^{(A)}$ there exists $\fq\in K[X\el]^{(B_0)}$ with
  $\chi_f=\res_A\chi_\fq$, where $B_0:=B\cap MA$. For this, consider
  the $A$-eigenspace $\cS_{\chi_f}:=K[X\el]_{\chi_f}$. This space is
  $MA$-stable and non-zero by assumption. Then any highest weight
  vector $\fq\in\cS_{\chi_f}$ is a $B_0$-semiinvariant with the
  required property.
\end{proof}

This result can be rephrased. By construction, the tori $A_k(X)$ and
$A_K(X)$ are quotients of $A$ and $T$, respectively. Then the lemma
means:

\begin{corollary}\label{cor:AkAK}

  The torus $A_k(X)$ is the image of $A\subseteq T$ in $A_K(X)$.

\end{corollary}

There is a slightly weaker result. Recall the notation
$\cN(A)=\Hom(\Xi(A),\QQ)$ and set
\[
  \cN_k(X)=\cN(A_k(X))=\Hom(\Xi_k(X),\QQ).
\]
This is a $\QQ$-vector space of dimension $\rk_kX$.

\begin{corollary}\label{cor:NkNK}

  The space $\cN_k(X)$ is the image of $\cN(A)\subseteq\cN(T)$ in
  $\cN_K(X)$.

\end{corollary}

\begin{remarks}

  1. Observe that $T$ is not unique, neither up to conjugacy nor
  up to isomorphism. An exception is the field $k=\RR$ of real numbers
  where $T$ is unique up to conjugation by $M(\RR)$.

  2. The character group $\Xi_k(X)$ can also be defined as
  \[
    \Xi_k(X):=\{\chi_f\in\Xi(A)\mid f\in k(X)^{(AN)}\},
  \]
  i.e., with semiinvariants defined over $k$ instead of $K$. To see
  this, one can again reduce to the case where $G$ is elementary, $X$
  is affine, and $f$ is regular. The point is that, since $A$ is
  split, the character $\chi_f$ is defined over $k$ even if $f$ is
  not. Then $k[X]_{\chi_f}\otimes_kK\cong K[X]_{\chi_f}\ne0$ implies
  $k[X]_{\chi_f}\ne0$.

\end{remarks}

\newcommand{\ue}{{\underline e}} \newcommand{\uu}{{\underline u}}
\newcommand{\uv}{{\underline v}}

\begin{example}

  We illustrate the theory with a simple but instructive example. Let
  $D$ be a finite-dimensional division algebra over $k$. For $n\ge2$,
  let
  \[
    H_1:=1\times GL(n-1,D)\subset GL(1,D)\times GL(n-1,D)\subseteq
    G:=GL(n,D).
  \]
  Then $X:=G/H_1$ can be realized as follows: let $V=D^n$ be the space
  of column vectors with $G$-action on the left, $g\cdot\uv:=g\uv$,
  and $V^\top$ be the space of row vectors with action
  $g\cdot\uu^\top:=\uu^\top g^{-1}$. Then the map
  $b:V^\top\times V\to D$ given by $b(\uu^\top,\uv)=\sum_i u_iv_i$ is
  $G$-invariant. Let $\ue_1,\ldots,\ue_n$ denote the canonical basis
  vectors of $D^n$. Since $H_1$ is the isotropy group of
  $(\ue_1^\top,\ue_1)$, we get
  \[
    X=\{(\uu^\top,\uv)\in V^\top\times V\mid b(\uu^\top,\uv)=1_D\}.
  \]
  Let $Q\subseteq G$ be the parabolic subgroup fixing the subspaces
  $D\ue_1\subseteq V$ and $\ue_n^\top D\subseteq V^\top$. Then $Q=LU$
  with
  \[
    L=GL(1,D)\times GL(n-2,D)\times GL(1,D)
  \]
  and $U\cong D^{2n-3}$ (as a vector space). Let $X_0\subseteq X$ be
  the open set of $(\uu^\top,\uv)$ with
  $u_1,v_n\in D^\times=GL(1,D)$. Then it is easy to see that $U$
  acts freely on $X_0$ and that the map
  \[
    \pi:X_0\to D^\times\times D^\times:(\uu^\top,\uv)\mapsto (u_1,v_n)
  \]
  is a quotient by $U$. This shows that $Q_k(X)=Q$ and
  $X\el=D^\times\times D^\times$ with $L$-action
  \[
    (d_1,d_2,d_3)(u_1,v_n)=(u_1d_1^{-1},d_3v_n).
  \]
  Since this action is transitive and elementary, we see that $X$ is
  $k$-spherical with $\rk_kX=2$. The anisotropic kernel is
  $X\an=PGL(1,D)\times PGL(1,D)$. On the other hand, $X\an$ is
  $K$-spherical if and only if $D$ is commutative.

\end{example}

\section{Invariant valuations}\label{sec:InvVal}

Let $X$ be any $K$-variety. In this section, we study a particular
type of valuation on the field $K(X)$ of rational functions on $X$. A
$K$-valuation (or, by abuse of terminology, a valuation) is a map
$v:K(X)\to\QQ\cup\{\infty\}$ satisfying

\begin{enumerate}[noitemsep]
  
\item $v(f_1f_2)=v(f_1)+v(f_2)$,
\item $v(f_1+f_2)\ge\min\{v(f_1),v(f_2)\}$,
\item $v(f)=\infty\Leftrightarrow f=0$,
\item $v(K^*)=0$,

\end{enumerate}

where $f,f_1,f_2$ range over $K(X)$. We also allow the trivial
valuation $v=0$, where $v(f)=0$ for all $f\ne 0$. Two valuations
$v_1,v_2$ are called \emph{equivalent} ($v_1\sim v_2$) if there exists
$c\in\QQ_{>0}$ such that $v_1=c\,v_2$. One could also require the
image of $v$ to be discrete, but this condition holds automatically
for all valuations considered here. For a general introduction to
valuation theory, see \Cite{ZS}{Ch.\ VI}. Invariant valuations are
covered in \cites{LunaVust,KnopIB}. We now recall some basic facts.

If $X$ is normal, then any irreducible divisor $D\subset X$ induces a
valuation $v_D$ via the order of vanishing along $D$. More generally,
let $\Xq$ be a normal variety with a fixed birational isomorphism to
$X$. Since $K(\Xq)=K(X)$, any irreducible divisor $D\subset\Xq$
induces a valuation $v_D$ of $K(X)$. Valuations that are equivalent to
some $v_D$ are called \emph{geometric}. It is useful to allow $D=X$
with $v_X=0$ and to call it geometric as well. Geometric valuations
are characterized by having a residue field whose transcendence degree
over $K$ is $\ge\dim X-1$.

A pair $(\Xq,D)$ with $v\sim v_D$ is called a \emph{model for
  $v$}. The model is unique in the following sense: if $(X',D')$ is
another model for $v$, then the composition of birational maps
$\Xq\rat X\rat X'$ is regular at the generic point of $D$, its inverse
is regular at the generic point of $D'$, and it maps $D$ to $D'$.

If $X$ is defined over $k$, then $\Xq$ can be chosen to be defined
over $k$. Indeed, if $\{v_1,\ldots,v_n\}$ is the $\cG$-orbit of $v$,
then each $v_i$ defines a model $(\Xq_i,D_i)$ with
$v_{D_i}=v_i$. These models can be glued along $X$ to obtain a model
$\XS$ that is defined over $k$ but which is possibly
inseparable. Since the inseparability locus does not contain any of
the divisors $D_i$, it has codimension $\ge2$ and can thus be removed,
yielding $\Xq$. If, in addition, $v$ is invariant under the action of
the Galois group, then gluing is unnecessary, and $D$ is also defined
over $k$.

Now assume that a connected linear algebraic group $H$ acts on
$X$. Then $H(K)$ acts on the set of valuations of $K(X)$. In
particular, $v$ is called $H$-invariant if
\[
  v(hf)=v(f)\text{ for all $h\in H(K)$ and $f\in K(X)$}.
\]
An infinitesimal characterization is as follows: the Lie algebra $\fh$
of $H$ acts on $K(X)$ via derivations. Then a valuation $v$ is
$H$-invariant if and only if
\[\label{eq:invariant}
  v(\xi f)\ge v(f)\text{ for all $\xi\in\fh$ and $f\in K(X)$}
\]
(see \Cite{LunaVust}{3.2.\ Cor.\ 3}).

The correspondence between divisors and valuations also holds in the
equivariant setting: if $v$ is an $H$-invariant geometric valuation of
$K(X)$, then there exists a normal model $\Xq$ of $X$ equipped with an
$H$-action and an $H$-stable irreducible divisor $D\subset \Xq$ such
that $v\sim v_D$ (\Cite{KnopIB}{Kor.\ 7.2}). Of particular importance
is the case where $X$ is a homogeneous variety. Then an equivariant
model of $X$ is equivalent to an open embedding $X\into\Xq$, a
so-called \emph{equivariant embedding} of $X$. Moreover, any invariant
irreducible divisor is an irreducible component of the boundary
$\Xq\setminus X$.

In this paper, we are only interested in a very specific type of
valuation.

\begin{definition}

  A $G$-invariant valuation of a $k$-dense $G$-variety $X$ is called
  \emph{$k$-central} if it is trivial on the subfield
  $K(X\an)=K(X)^{AN}$. The set of $k$-central valuations is denoted by
  $\cZ_k(X)$.

\end{definition}

It follows directly from the definition and the short exact sequence
\eqref{eq:short-exact} that a $k$-central valuation induces a
homomorphism
\[\label{eq:lambdav}
  \lambda_v:\Xi_k(X)\to\QQ:\chi_f\mapsto v(f)
\]
which yields a map
\[\label{eq:ZtoN}
  \iota_k:\cZ_k(X)\to\cN_k(X).
\]
Over $K$, central valuations have been extensively studied in
\cite{KnopIB}. To transfer properties from $K$ to $k$, we choose
$T\subseteq B\subseteq P$ as at the end of \cref{sec:LST}. Recall (see
\cref{cor:NkNK}) that $\cN_k(X)$ can be viewed as a subspace of
$\cN_K(X)$.

\begin{proposition}\label{prop:kKcentral}

  Let $X$ be a $k$-dense $G$-variety. Then:

  \begin{enumerate}

  \item The map $\iota_k$ in \eqref{eq:ZtoN} is injective.

  \item If we consider $\cZ_k(X)$ and $\cZ_K(X)$ as subsets of
    $\cN_k(X)$ and $\cN_K(X)$, respectively, then
    \[\label{eq:diagXi22}
      \cZ_k(X)=\cZ_K(X)\cap\cN_k(X).
    \]

  \end{enumerate}

\end{proposition}

\begin{proof}
  Since $AN\subseteq B$, it follows that $K(X)^{AN}\supseteq K(X)^B$.
  Thus, every $k$-central valuation is also $K$-central. This provides
  the following commutative diagram:
  \[
    \cxymatrix{
      \cZ_k(X)\ar@{^(->}[r]\ar[d]^{\iota_k}&\cZ_K(X)\ar[d]^{\iota_K}\\
      \cN_k(X)\ar@{^(->}[r]&\cN_K(X)}.
  \]
  Hence $\iota_k$ is injective because $\iota_K$ is injective by
  \Cite{KnopIB}{Kor.\ 3.6}.

  It remains to be shown that
  $\cZ_K(X)\cap\cN_k(X)\subseteq\cZ_k(X)$. Let $v$ lie in the
  intersection. Then $v$ is trivial on all $B$-semiinvariants $f$ with
  $\res_A\chi_f=0$, i.e., on $(K(X)^{(B)})^A$. We have
  \[
    (K(X)^{(B)})^A=(K(X)^{AN})^{(M\cap B)}=K(X\an)^{(M\cap B)}.
  \]
  It follows from \Cite{KnopIB}{Kor.\ 3.6} that $v$ is trivial on
  $K(X\an)$. Thus, $v$ is $k$-central.
\end{proof}

From this, we obtain:

\begin{corollary}\label{cor:valcone}

  Let $X$ be a $k$-dense $G$-variety. Then $\cZ_k(X)$ is a finitely
  generated convex cone in $\cN_k(X)$.

\end{corollary}

\begin{proof}
  This follows from \eqref{eq:diagXi22} and the corresponding
  statement for $\cZ_K(X)$ (see \Cite{KnopIB}{Kor.\ 6.5}).
\end{proof}

The next result shows, in particular, that $\cZ_k(X)$ is quite large
in the sense that it contains an open subset of $\cN_k(X)$. To this
end, let
\[
  \cN^-(A)=\{a\in\cN(A)\mid\alpha(a)\le0\text{ for all }\alpha\in
  S_k\}
\]
be the anti-dominant Weyl chamber with respect to the restricted root
system of $G$.

\begin{proposition}\label{prop:chamber}
 
  Let $\pi:\cN(A)\auf\cN_k(X)$ be the canonical projection. Then
  \[
    \pi(\cN^-(A))\subseteq\cZ_k(X).
  \]
  In particular, the interior of $\cZ_k(X)$ is non-empty.

\end{proposition}

\begin{proof}
  Let $\cN^-(T)$ be the anti-dominant Weyl chamber of $G$ over
  $K$. Then $\cN^-(A)=\cN^-(T)\cap\cN(A)$. By \Cite{KnopIB}{Kor.\
    5.2.4}, the image of $\cN^-(T)$ in $\cN_K(X)$ lies in
  $\cZ_K(X)$. Thus, $\cZ_k(X)$ contains the image of $\cN^-(A)$.
\end{proof}

\section{Valuations and geometry}\label{sec:ValGeo}

Next, we study the generic structure of invariant divisors defined by
valuations. To this end, we require a refinement of
\cref{cor:LSTgeneric} that also holds ``at infinity.''

Let $G$ be a connected reductive group and let $X$ be a normal
$k$-dense $G$-variety. According to \cref{cor:LSTgeneric}, there is an
open embedding $U\times X\el\into X$ with $U=R_uQ_k(X)$. By projecting
to the first factor, we obtain a dominant rational map $u:X\rat U$
with the characteristic property that $u(x)^{-1}x\in X\el$ for a
generic $x\in X$. This map induces a pull-back homomorphism of
function fields $u^*:K(U)\into K(X)$.

\begin{lemma}\label{lem:u*val}
  Let $X$ be a $k$-dense $G$-variety. Then there exists an open
  embedding $U\times X\el\into X$ as in \cref{cor:LSTgeneric}, with
  induced embedding of fields $u^*:K(U)\into K(X)$ as above, such that
  $v(u^*(h))\ge0$ for every $G$-invariant valuation $v$ of $X$ and all
  regular functions $h\in K[U]$.
\end{lemma}

\begin{proof}
  We take for $U\times X\el\into X$ the embedding produced by the
  explicit construction of $X\el$ in the proof of \cref{kLST}, and we
  let $v$ be an arbitrary $G$-invariant valuation of $X$. First, we
  constructed a $G$-stable open subset
  $X_0\subseteq X$ that is embeddable in some $\P(V)$ and considered
  the affine cone $\XS\subseteq V$ over $X_0$. We claim that it
  suffices to prove our assertion for $\XS$ and the group
  $\GS=\G_m\times G$. First of all, observe that the slice
  $\XS\el\subseteq\XS$ is homogeneous. This implies that the
  projection $\uS:\XS\rat U$ factors through $u:X\rat U$. Second, it
  is known from \Cite{KnopIB}{Kor.\ 3.2} that the valuation $v$ on $X$
  can be lifted to a $\G_m\times G$-invariant valuation $\vS$ on
  $\XS$. Then
  \[
    v(u^*(h))=\vS(\uS^*(h))\ge0
  \]
  proves the claim. Replacing $X$ by $\XS$, we may assume that $X$ is
  a subvariety of $V$.

  Next, we used a $P$-semiinvariant regular function $f$ on $X$. The
  stabilizer of the line $Kf$ is the parabolic $\Qq=\Lq\,\Uq$. Then a
  $\Qq$-equivariant map $m:X_f\to a+\overline\fu$ with fiber
  $X':=m^{-1}(a)$ was defined such that
  $\overline\iota:\Uq\times X'\overset\sim\to X_f$ is an
  isomorphism. Finally, applying the induction hypothesis yielded an
  open embedding $\iota':U'\times R\into X'$.

  The open embeddings $\overline\iota$ and $\iota'$ allow us to define
  projections $\uq:X\rat\Uq$ and $u':X'\rat U'$ onto the first
  factors. Let $\pi(x):=\uq(x)^{-1}x$ be the projection of $X$ to the
  second factor $X'$. In this notation, $u$ has the form
  \[
    u:X\rat U=\Uq\rtimes U':x\mapsto (\uq(x),u'(\pi(x))).
  \]
  It clearly suffices to prove the non-negativity of $v(u^*(h))$ for
  generators $h$ of the algebra $K[U]$. Thus, we may assume that $h$
  is either a function on $\Uq$ or on $U'$. In the first case, observe
  that $\uq$ is the composition of $m:X_f\to a+\overline\fu$ with the
  isomorphism $a+\overline\fu\cong\Uq$. The functions on the affine
  space $a+\overline\fu$ are generated by the linear functions
  $\ell_\xi:=(\xi,\cdot)$ with $\xi\in\fg$, where $(\cdot,\cdot)$ is
  an invariant scalar product on $\fg$. From the definition of $m$, we
  obtain
  \[
    v(m^*(\ell_\xi))=v\left(\frac{\xi f}{f}\right)=v(\xi f)-v(f).
  \]
  Now the $G$-invariance of $v$ implies that the right-hand side is
  non-negative (see \eqref{eq:invariant}).

  Finally, assume that $h$ is a regular function on $U'$. We restrict
  the valuation $v$ to $X'$ via $\pi$; that is, for any $h'\in K(X')$,
  we define $v'(h'):=v(\pi^*(h'))$. It is easy to see that $v'$ is an
  $\Lq$-invariant valuation of $X'$. The induction hypothesis then
  yields
  \[
    v((u'\circ\pi)^*(h))=v'((u')^*(h))\ge0.\qedhere\hfill\qed
  \]
\end{proof}

As a consequence, we obtain the existence of ``very big'' slices,
provided we relax the affineness requirement of the slice and restrict
to normal varieties.

\begin{theorem}\label{lem:bigLST}

  Let $X$ be a normal $k$-dense $G$-variety and let $Q_k(X)=LU$ be a
  Levi decomposition. Then there exists an $L$-stable subvariety
  $\Xq\el\subseteq X$ (all defined over $k$) such that:

  \begin{enumerate}

  \item\label{it:bigLST1} the action of $L$ on $\Xq\el$ is elementary;

  \item\label{it:bigLST2} the natural morphism
    $Q_k(X)\times^L\Xq\el\to X$ is an open embedding;

  \item\label{it:bigLST3} the boundary $\partial X_0=X\setminus X_0$
    of the open set $X_0:=U\cdot\Xq\el\subseteq X$ is of pure
    codimension one in $X$, and none of its irreducible components is
    $G$-stable.

  \end{enumerate}

\end{theorem}

\begin{proof}
  Let $U\times X\el\into X$ be an open embedding as in
  \cref{lem:u*val} and let $X_0\subseteq X$ be the largest open subset
  on which the corresponding rational map $u:X\rat U$ is regular. Let
  $\Xq\el:=u^{-1}(1)\subseteq X_0$. Since $u$ is $U$-equivariant, we
  infer that $U\times\Xq\el\to X_0$ is an isomorphism, which proves
  \ref{it:bigLST2}. Property \ref{it:bigLST1} follows from the fact
  that $X\el$ is open in $\Xq\el$.

  The main point is to show \ref{it:bigLST3}. Since $U$ is affine, a
  point $x\in X$ is in $X_0$ if and only if $u^*(h)$ is regular in a
  neighborhood of $x$ for all $h\in K[U]$. In a normal variety, this
  is the case if and only if $x$ is not contained in the pole divisor
  of any function $u^*(h)$. Thus, $\partial X_0$ is the union of these
  pole divisors and is therefore of pure codimension one.

  Finally, suppose $\partial X_0$ has an irreducible component $D$
  that is $G$-stable (though possibly not defined over $k$). Since
  $v_D$ is a $G$-invariant valuation, \cref{lem:u*val} implies
  $v_D(u^*(h))\ge0$ for all $h\in K[U]$. Thus, the generic point of
  $D$ lies in $X_0$, contradicting the inclusion
  $D\subseteq\partial X_0$.
\end{proof}

From \cref{prop:uniqueQ}, we obtain:

\begin{corollary}

  Let $X$ be a $k$-dense $G$-variety and let $D\subseteq X$ be a
  $k$-dense irreducible divisor. Then $Q_k(D)=Q_k(X)$.

\end{corollary}

Next, we relate $k$-central valuations to toroidal embeddings. Fix a
$k$-central valuation $v$, which induces an element
$a:=\lambda_v\in\cN_k(X)=\cN(A_k)$ (with $A_k:=A_k(X)$).  Let
$\cR_a:=\QQ_{\ge0}a\subseteq\cN(A_k)=\cN_k(X)$ be the ray generated by
$a$. This $a$ defines a unique $A_k$-invariant valuation $v_a$ of
$K(A_k)$ by $v_a(\chi)=a(\chi)$ for all characters
$\chi:A_k\to\G_m$. Toroidal embedding theory asserts the existence of
an embedding $\Aq=A_k\cup D(a)$, where $D(a)$ is a homogeneous divisor
characterized by the property $v_{D(a)}\sim v_a$.

Alternatively, one can relate $\Aq$ to $a$ as follows: The space
$\cN(A_k)$ contains the group $\Hom(\G_m,A_k)$ of cocharacters as a
sublattice. Thus, there is a homomorphism $\aq:\G_m\to A_k$ with
$\aq=c\,a$ for some $c\in\QQ_{>0}$. Then
\[\label{eq:limpi}
  \pi_a:\Aq\to \Aq: x\mapsto\lim_{t\to\infty}\aq(t)x
\]
is a well-defined projection onto $D(a)$. Here the limit is taken for
$t\to\infty$ because $f_\chi(\aq(t)x)=t^{-\aq(\chi)}f_\chi(x)$ for
every $A_k$-semiinvariant $f_\chi$, so that precisely the functions
with $v_a(f_\chi)=a(\chi)>0$ tend to zero.

There is also a slightly more general version of this
construction. For a principal $A_k$-bundle $X\to Y$ (defined over
$k$), let $X(a):=X\times^{A_k}\Aq\to Y$ be the associated fiber
bundle. Then $X$ is open in $X(a)$ and the complement is an
irreducible divisor $D_a$. Since $K(X)^{A_k}=K(Y)$, the associated
valuation $v_{D_a}$ is $k$-central with
$v_{D_a}(f_\chi)=a(\chi)$. This shows that $\cZ_k(X)=\cN_k(X)$, and
thus every $k$-central valuation is of this type. Clearly,
\eqref{eq:limpi} still defines a projection $\pi_a:X(a)\to D_a$.

We now treat the general case:

\begin{lemma}\label{lem:generic-central}

  Let $X$ be a normal $k$-dense $G$-variety and let $D\subseteq X$ be
  an irreducible $G$-stable divisor such that $v_D$ is
  $k$-central. Put $a=\lambda_{v_D}$. Then there exists a slice
  $X\el\subseteq\Xq\el$ as in \cref{lem:bigLST} such that
  $X\el(a)\subseteq\Xq\el$ is open and $D_a\subseteq D\cap\Xq\el$.

\end{lemma}

\begin{proof}
  Let $D\el:=D\cap\Xq\el$. We must show that the valuation $v_{D\el}$
  of $X\el$ is equal to the valuation $v_{D_a}$. For any
  $f\in K(X\el)$, let $t^*f\in K(X)$ be its pull-back via the rational
  map $t:X\rat X\el$. Then
  \[
    v_D(t^*f)=v_{D\el}(f)
  \]
  by \cref{lem:bigLST}. Let $f\in K(X\el)^A$. Then $t^*f\in K(X)^{AN}$
  and therefore $v_{D\el}(f)=0$ since $v_D$ is $k$-central. This means
  that $v_{D\el}$ is $k$-central for the $A$-action on $X\el$. By the
  injectivity of the map \eqref{eq:ZtoN}, it suffices to check that
  $v_{D\el}(f_\chi)=v_{D_a}(f_\chi)$ for all $A$-semiinvariants
  $f_\chi$. This is clear:
  \[
    v_{D_a}(f_\chi)=a(\chi)=\lambda_{v_D}(\chi)=v_D(t^*f_\chi)=v_{D\el}(f_\chi).
    \qedhere\hfill\qed
  \]
\end{proof}

Now we introduce another important property of invariant valuations.

\begin{definition}

  Let $X$ be a $k$-dense variety. A geometric valuation $v$ of $X$ is
  called \emph{$k$-dense} if either it is trivial or the corresponding
  divisor $D$ in any $k$-model of $v$ is defined over $k$ and is
  $k$-dense. If $v$ is $G$-invariant and $k$-dense, we define the
  complexity of $v$ as $c_k(v)=c_k(D)$.

\end{definition}

The following theorem characterizes $k$-central valuations. It
implies, in particular, that all $k$-central valuations are $k$-dense.

\begin{theorem}\label{thm:valdense}

  Let $X$ be a $k$-dense $G$-variety. For a $G$-invariant valuation
  $v$ of $K(X)$, the following are equivalent:

  \begin{enumerate}[noitemsep]

  \item\label{it:valdense-ii} $v$ is $k$-central.

  \item\label{it:valdense-i} $v$ is $k$-dense and the restriction of
    $v$ to $K(X)^P=K(X\an)^M$ is trivial.

  \item\label{it:valdense-iii} $v$ is $k$-dense with $c_k(v)=c_k(X)$.

  \item\label{it:valdense-iv} $v$ is $k$-dense and $K$-central.

  \end{enumerate}
\end{theorem}

\begin{proof}
  If $v$ is trivial, all conditions hold. Assume $v\ne0$. We claim all
  conditions imply $v$ is geometric. For \ref{it:valdense-ii}, this
  follows from \Cite{KnopIB}{Satz 4.4} and the fact that $v$ is
  $K$-central. For the others, this is part of the definition of
  $k$-density. Thus, we may assume that $X$ is normal and that it
  contains a divisor $D$ with $v=v_D$.  We may also remove from $D$
  all $G$-orbits of non-maximal dimension, ensuring that all $G$-orbits in
  $D$ are closed.

  \ref{it:valdense-ii}$\Rightarrow$\ref{it:valdense-i}. Since
  $AN\subseteq P$, we have $K(X)^P\subseteq K(X)^{AN}$. Thus, $v$ is
  trivial on $K(X)^P$. In the setting of \cref{lem:generic-central},
  since $X$ is $k$-dense, $X\el$ is also $k$-dense. By the projection
  $\pi_a:X\el\to D(a)$ (see \eqref{eq:limpi}), $D(a)$ is also
  $k$-dense. Thus $D$, being birational to $U\times D(a)$, is
  $k$-dense.

  \ref{it:valdense-i}$\Rightarrow$\ref{it:valdense-ii}. We first treat
  the case where the action of $G$ on $X$ is elementary.  We may
  assume that $G=MA$ is elementary itself. In this case, $P=G$ and
  $N=1$. Then, using \cref{cor:P-stable-affine}, we may assume that
  $X$ and $D$ are affine and smooth.

  Consider the categorical quotient $\pi:X\to Y:=X\mod A$. Then
  $Z:=\pi(D)=D\mod A$ is a closed subset of $Y$. If $Z\subsetneq Y$,
  there exists an $M$-invariant function $f_0\ne0$ on $Y$ that
  vanishes on $Z$ (\cref{SemiinvIdealElementary}). Its pull-back
  $f=\pi^*(f_0)$ would be a $G$-invariant function on $X$ with
  $v(f)>0$, contradicting the assumption.

  Thus, $\pi:D\to Y$ is surjective. This shows that $v(f)=0$ for all
  non-zero $f\in K[Y]=K[X]^A$. It remains to show that $K(X)^A$ is the
  field of fractions of $K[X]^A$. This is equivalent to the generic
  fiber $X_y=\pi^{-1}(y)$ of $\pi$ containing a dense open
  $A$-orbit. Since $X_y$ is smooth and irreducible, it suffices to
  show that $X_y$ contains an orbit of dimension $r=\dim X-\dim
  Y$. Recall that $X_y$, as a fiber of a categorical quotient,
  contains exactly one closed $A$-orbit. By construction, all
  $G$-orbits of $D$ are closed, which implies that $X_y\cap D$ is the
  closed $A$-orbit in $X_y$. Since $y$ is generic, the codimension of
  $X_y\cap D$ is $1$. Moreover, it lies in the closure of every
  $A$-orbit in $X_y$. This shows that $X_y$ contains an orbit of
  dimension $r$, proving the claim.

  Now let $X$ be any $G$-variety. \cref{lem:bigLST} yields an
  $L$-stable elementary slice $\Xq\el\subseteq X$ such that
  $D\el:=\Xq\el\cap D\ne\leer$. Put $v\el:=v_{D\el}$. Since $D$ is
  birational to $U\times D\el$, also $v\el$ is $k$-dense. Since
  $K(X)^P=K(X\el)^{MA}$, $v\el$ meets both conditions of
  \ref{it:valdense-i}. Thus, by the above, $v\el$ is
  $k$-central. Finally, $K(X)^{AN}=K(X\el)^A$ shows that $v$ is
  $k$-central.

  The equivalence of \ref{it:valdense-i} and \ref{it:valdense-iii}
  follows from \cref{prop:cYcX}. The equivalence of
  \ref{it:valdense-iv} with the other conditions follows from
  $K(X)^P\subseteq K(X)^B\subseteq K(X)^{AN}$.
\end{proof}

\begin{remark}

  One can show that the complexity of a non-central $k$-dense
  valuation equals $c_k(X)-1$.

\end{remark}

For a $k$-spherical variety $X$, we have $K(X)^P=K$. Hence:

\begin{corollary}

  Let $X$ be a $k$-spherical $G$-variety. Then a $G$-invariant
  valuation on $X$ is $k$-central if and only if it is $k$-dense.

\end{corollary}

We conclude this section with the observation that subvarieties of
maximal $k$-complexity are, in a sense, controlled by $k$-central
valuations. Recall that a subvariety $Y$ of $X$ is the center of a
valuation $v$ if $v\ge0$ on the local ring $\cO_{X,Y}$ and $v>0$ on
its maximal ideal $\fm_{X,Y}$. If $v$ is geometric with model
$(\XS,D)$, this means that the birational map $\XS\rat X$ maps $D$
dominantly onto $Y$. The center is unique if it exists. More
generally, if $\XS\to X$ is a dominant morphism and $v$ is a valuation
of $\XS$, then $Y\subseteq X$ is the center of $v$ if it is the center
of the restricted valuation $v|_{K(X)}$.

\begin{lemma}\label{lem:centercentral}

  Let $\pi:\XS\to X$ be a dominant $G$-morphism where $\XS$ is
  $k$-dense and $X$ is locally $k$-linear. Let $Y\subseteq X$ be a
  $k$-dense closed $G$-subvariety with $c_k(Y)=c_k(\XS)$. Then $Y$ is
  the center of a $k$-central valuation of $\XS$.

\end{lemma}

\begin{proof}
  The strategy is to construct a dominant rational $G$-map
  $\Phi:Z\times\A_k^1\rat\XS$ where $Z$ is a $k$-dense $G$-variety,
  $\pi\circ\Phi$ is defined at $D_0:=Z\times 0$, and $\pi\circ\Phi$
  maps $D_0$ dominantly onto $Y$. In this case, $\Phi$ induces an
  embedding $k(\XS)\into k(Z\times\A_k^1)$. We claim that we can take
  $v$ to be the restriction of $v_{D_0}$ to $k(\XS)$.

  First, $v$ is geometric and $G$-invariant since $v_{D_0}$ is. Let
  $(\Xq,D)$ be a model of $v$. Then the rational map
  $Z\times\A_k^1\rat\Xq$ maps $D_0$ dominantly onto $D$. This implies
  that $v$ is $k$-dense since $Z$ is $k$-dense. Furthermore, since
  $\pi\circ\Phi$ maps $D_0$ dominantly onto $Y$, the divisor $D$ is
  mapped by $\pi$ dominantly onto $Y$; i.e., $Y$ is the center of $v$
  in $X$. Finally, since $c_k(v)\ge c_k(Y)=c_k(\XS)$, the valuation
  $v$ is $k$-central by \cref{thm:valdense}.

  To construct $\Phi$, we first apply the Local Structure \cref{kLST}
  to $Y\subseteq X$. In the notation of that theorem, we have
  $\dim Y\el\mod L=\|trdeg|_k K(Y)^P=c_k(Y)$. On the other hand, we
  have $\dim R\mod L\le c_k(X)\le c_k(\XS)$. Since $c_k(Y)=c_k(\XS)$
  and since $Y\el\mod L\to R\mod L$ is a closed embedding, we see that
  actually $Y\el\mod L=R\mod L=:S$. This means that every $L$-orbit of
  $R$ contains a unique $L$-orbit of $Y\el$ in its closure.

  Let $\RS:=\pi^{-1}(R)\subseteq\XS$ such that $U\times\RS\to\XS$ is
  birational. We now apply Kempf's \cref{thm:Kempf} simultaneously to
  all points of the form $\pi(x)$ with $x\in\RS$. To this end,
  consider the field extension $F:=k(\RS)$ of $k(S)$ and the affine
  $F$-variety $R_F:=R\times_S\Spec F$ obtained by base change. Then
  $R_F$ contains exactly one closed $L_F$-orbit, namely
  $Y_F:=Y\el\times_S\Spec F$. The generic point $\eta=\Spec F$ of
  $\RS$ can be considered as an $F$-rational point of $R_F$.  Kempf's
  theorem applied to $L_F$ acting on $R_F$ at the point $\eta$ and the
  closed subset $Y_F$ yields a one-parameter subgroup
  $\lambda:\G_{m,F}\to L_F$ such that the orbit map
  $\G_{m,F}\to R_F:t\mapsto \lambda(t)\eta$ extends to a morphism
  $\Phi':\A_F^1\to R_F$ with $\Phi'(0)\in Y_F$.

  Unraveling the definitions, $\lambda$ corresponds to a rational
  morphism $\RS\times\G_m\rat L:(x,t)\mapsto\lambda_x(t)$.  The orbit
  map $\Phi_0:\RS\times\A_k^1\rat\RS:(x,t)\mapsto\lambda_x(t)x$ has
  the property that
  $\Phi'=\pi\circ\Phi_0:\RS\times\A_k^1\to
  R:(x,t)\mapsto\lambda_x(t)\pi(x)$ is defined at $t=0$, with
  $Y_0:=\Phi'(\RS\times0)\subseteq Y\el$. The fact that $Y_F$ is an
  $L_F$-orbit means that $Y_0$ meets the generic $L$-orbit of
  $Y\el$. Moreover, $\Phi_0$ is dominant because $\Phi_0(x,1)=x$ for
  all $x\in\RS$.

  Finally, we extend $\Phi_0$ to $\XS$ by using the birational map
  $U\times\RS\overset\sim\rat\XS$. For this, we define $Z:=G\times\XS$
  and $\Phi$ as the composition of
  \[
    \begin{split}G\times\XS\times\A_k^1&\overset\sim\rat G\times
      U\times\RS\times\A_k^1 \xrightarrow{\|id|\times\Phi_0} G\times
      U\times\RS \overset\sim\rat G\times\XS\to\XS:\\(g,u,x,t)&\mapsto
      gu\lambda_x(t)x.
    \end{split}
  \]
  Then $\Phi$ is equivariant and dominant. Moreover, $D_0$ is mapped
  dominantly onto $Y$ because $GY_0\supseteq ULY_0$ is dense in $Y$.
\end{proof}

A first application is:

\begin{corollary}\label{lem:GAcentral}

  Let $\phi:X\to Y$ be a dominant $G$-equivariant morphism between
  $k$-dense $G$-varieties. Then $\phi$ induces a surjective map
  $\phi_*:\cN_k(X)\to\cN_k(Y)$ such that
  $\phi_*(\cZ_k(X))\subseteq\cZ_k(Y)$ with equality if
  $c_k(X)=c_k(Y)$. In particular, if $c_k(X)=c_k(Y)$ and
  $\rk_kX=\rk_kY$ then $\phi_*:\cZ_k(X)\to\cZ_k(Y)$ is bijective.

\end{corollary}

\begin{proof}
  The morphism $\phi$ induces an injective homomorphism
  $\Xi_k(Y)\into\Xi_k(X)$ whose dual is $\phi_*$, which is therefore
  surjective. Since $\phi^* K(Y)^{AN}\subseteq K(X)^{AN}$, the
  restriction of a central valuation on $K(X)$ to $K(Y)$ is central,
  hence $\phi_*\cZ_k(X)\subseteq\cZ_k(Y)$.

  Now assume $c_k(X)=c_k(Y)$ and let $v\in\cZ_k(Y)$. Let $(\Yq,D)$ be
  a model of $v$. Since $c_k(D)=c_k(Y)=c_k(X)$,
  \cref{lem:centercentral} implies that $D$ is the center of some
  $\vq\in\cZ_k(X)$. This proves $\phi_*\cZ_k(X)\supseteq\cZ_k(Y)$. The
  last assertion is clear.
\end{proof}

\begin{remark}

  If $c_k(Y)<c_k(X)$ then $\phi_*$ is in general not surjective on
  $k$-central valuations. Let, for example, $k=\CC$, $G=SL(2)$, and
  $X=SL(2)$ with $G$ acting on the left. Then $\cZ_k(X)=\QQ_{>0}$ and
  $c_k(X)=1$ (easy, see e.g., \cite{LunaVust}). Now let
  $Y=\CC^2\setminus\{0\}=G/N$. Then $\cZ_k(Y)=\QQ$ and $c_k(Y)=0$ (see
  section \S\ref{sec:horospherical}). In this case, the surjective
  morphism $X\to Y$ is not surjective on central valuations.

\end{remark}

\section{Toroidal embeddings}\label{sec:Toroidal}

We have seen in \cref{lem:generic-central} that the neighborhood of a
$k$-central divisor is modeled on a toroidal embedding corresponding
to a ray, i.e., a one-dimensional cone. In this section, we generalize
this construction to other toroidal embeddings.

Recall that the theory of toroidal embeddings (see \cite{KempfTE} or
\cite{Oda}) associates to every fan $\cF$ in $\cN_k:=\cN_k(X)$ a
normal equivariant embedding $A_k\into A(\cF)$. Recall that a fan is a
finite collection of finitely generated strictly convex cones
$\cC\subseteq\cN_k$ satisfying certain axioms (see, e.g.,
\cite{KempfTE}). The embedding $A(\cF)$ has the property that each
$\cC\in\cF$ corresponds to an orbit $A(\cC)\subseteq A(\cF)$. Its
closure $\Aq(\cC)=\overline{A(\cC)}$ is given by
\[\label{eq:bigcup}
  \Aq(\cC)=\bigcup_{\cC\subseteq\cC'\in\cF}A(\cC').
\]
The \emph{support $\supp\cF$ of $\cF$} is the union of all cones
$\cC\in\cF$.

More generally, if $Z\to Y$ is a principal $A_k$-bundle (e.g., the
fibration $X\el\to X\an$ for a $G$-variety $X$), define
$Z(\cF):=Z\times^{A_k}A(\cF)$ to be the corresponding relative
toroidal embedding over $Y$. For arbitrary $k$-dense $G$-varieties, we
have:

\begin{theorem}\label{thm:CEmbedding}

  Let $X$ be a normal $k$-dense $G$-variety with LST-slice $X\el$ and
  assume that $X=G\cdot X\el$. Let $\cF$ be a fan with
  $\supp\cF\subseteq\cZ_k(X)$. Then there exists a $G$-variety
  $X(\cF)$ and an equivariant embedding $X\into X(\cF)$ such that the
  open embedding $U\times X\el\to X$ extends to an open embedding
  $U\times X\el(\cF)\to X(\cF)$ with $X(\cF)=G\cdot X\el(\cF)$.

\end{theorem}

\begin{proof}
  Let $X_0$ be any normal but not necessarily equivariant model of
  $X$. Since each $\xi\in\fg$ acts as a derivation on $K(X)$, it
  induces a rational vector field $\xi_*$ on $X_0$. Luna and Vust have
  shown (\Cite{LunaVust}{Prop.\ 1.4}) that there is a normal
  $G$-variety $\Xq$ containing $X_0$ as an open subset if and only if,
  for every $\xi\in\fg$, the vector field $\xi_*$ is regular on
  $X_0$. The condition $\Xq=G\cdot X_0$ determines this model
  uniquely.

  We claim that the Luna--Vust condition is satisfied for
  $X_0=U\times X\el(\cF)$. Since $X_0$ is normal, it suffices to show
  that $\xi_*$ is regular in codimension one. Regularity certainly
  holds on $U\times X\el$, since that set is open in the $G$-variety
  $X$. The complement consists of divisors $D$ corresponding to the
  one-dimensional cones (i.e., rays) $\cR\in\cF$. Because, by
  assumption, $\cR\subseteq\cZ_k(X)$, the set $(U\times X\el)\cup D$
  is isomorphic to an open subset of a $G$-variety by
  \cref{lem:generic-central}. Thus, $\xi_*$ is regular along $D$ as
  well, proving the claim.
\end{proof}

\begin{remark}

  The condition $X=G\cdot X\el$ is quite mild. For example, it is
  clearly satisfied when $G$ acts transitively on $X$. In general, let
  $X':=G\cdot X\el=G\cdot (U\cdot X\el)$. Then $X'$ is a $G$-stable
  open subset to which the theorem applies.

\end{remark}

We now study the orbit structure of $X(\cF)$. For every $\cC\in\cF$,
let
\[
  X\el(\cC):=X\el\times^{A_k}A(\cC)\text{ and }
  \Xq\el(\cC):=X\el\times^{A_k}\Aq(\cC)
\]
be the relative $A_k$-orbit corresponding to $\cC$ and its closure,
respectively. The sets $X\el(\cC)$ form a stratification of
$X\el(\cF)$ that extends to $X(\cF)$:

\begin{theorem}\label{prop:boundary}

  Let $X\subseteq X(\cF)$ be as in \cref{thm:CEmbedding}. For each
  $\cC\in\cF$, define subvarieties of $X(\cF)$ by
  \[
    X(\cC):=G\cdot X\el(\cC)\text{ and }\Xq(\cC):=G\cdot \Xq\el(\cC).
  \]
  \begin{enumerate}

  \item\label{it:boundary1} The sets $X(\cC)$ form a stratification of
    $X(\cF)$. More precisely, each $X(\cC)$ is locally closed, the
    closure of $X(\cC)$ is $\Xq(\cC)$, and $\Xq(\cC)$ is the disjoint
    union of all $X(\cC')$ with $\cC\subseteq\cC'\in\cF$.

  \item\label{it:boundary2} The stratification of $X(\cF)$ is
    compatible with that of $X\el(\cF)$, i.e.,
    $X(\cC)\cap X\el(\cF)=X\el(\cC)$ and
    $\Xq(\cC)\cap X\el(\cF)=\Xq\el(\cC)$ for all
    $\cC\in\cF$. Moreover, $U\times X\el(\cC)\to X(\cC)$ and
    $U\times \Xq\el(\cC)\to\Xq(\cC)$ are open embeddings.

  \item\label{it:boundary3} Each stratum $X(\cC)$ is $k$-dense with
    \begin{align*}\hspace{20pt}
      &\dim X(\cC)=\dim X-\dim\cC,\
        c_k(X(\cC))=c_k(X),\ \\&\rk_kX(\cC)=\rk_kX-\dim\cC,\\
      &Q_k(X(\cC))=Q_k(X),\ X(\cC)\el=X\el(\cC),\ X(\cC)\an=X\an.
    \end{align*}

  \item\label{it:boundary4} Let $\<\cC\>\subseteq\cN_k(X)$ be the
    subspace spanned by $\cC$. Then
    \[
      \Xi_k(X(\cC))=\Xi_k(X)\cap\<\cC\>^\perp,\
      \cN_k(X(\cC))=\cN_k(X)/\<\cC\>
    \]
    and $\cZ_k(X(\cC))$ is the image of $\cZ_k(X)$.

  \end{enumerate}

\end{theorem}

\begin{proof}
  We start with a general observation: because
  $X(\cF)=G\cdot X\el(\cF)$, we have
  \[\label{eq:remark}
    Z=G\cdot(Z\cap X\el(\cF))\text{ for all $G$-stable subsets
    }Z\subseteq X(\cF).
  \]

  Let $\XS(\cC)$ be the closure of $U\times \Xq\el(\cC)$ in
  $X(\cF)$. We claim that $\XS(\cC)=\Xq(\cC)$. For this, we first
  consider the case where $\dim\cC=1$. Let $a\in\cC$ be a generator, let
  $v\in\cZ_k(X)$ be the corresponding $k$-central valuation, and let
  $(\Xq,D)$ be a model of $v$. It follows from
  \cref{lem:generic-central} that the birational map $\Xq\rat X(\cF)$
  is regular along $D$ and maps $D$ to $U\times X\el(\cC)$. Because
  $D$ is $G$-stable, this implies that $\XS(\cC)$ is also
  $G$-stable. By construction, we have
  $\XS(\cC)\cap X\el(\cF)=\Xq\el(\cC)$. Thus, \eqref{eq:remark} with
  $Z=\XS(\cC)$ yields $\XS(\cC)=\Xq(\cC)$.

  Next, let $\cC\in\cF$ be arbitrary, let
  $\cR_1,\ldots,\cR_s\subseteq\cC$ be its extremal rays, and put
  \[
    X'(\cC):=\Xq(\cR_1)\cap\ldots\cap\Xq(\cR_s),
  \]
  which is a $G$-stable subset of $X(\cF)$. Then
  \[
    X'(\cC)\cap(U\times X\el(\cF))=\bigcap_{i=1}^s\left(U\times
      \Xq\el(\cR_i)\right) =U\times\Xq\el(\cC)=\XS(\cC)\cap(U\times
    X\el(\cF)).
  \]
  This implies that $\XS(\cC)$ is an irreducible component of
  $X'(\cC)$, and therefore $G$-stable as well. From \eqref{eq:remark}
  with $Z=\XS(\cC)$, we conclude that $\XS(\cC)=\Xq(\cC)$, proving the
  claim.

  We have shown that $\Xq(\cC)$ is closed, with
  \[\label{eq:inter1}
    \Xq(\cC)\cap(U\times X\el(\cF))=U\times\Xq\el(\cC).
  \]
  Next, we define the $G$-stable subvariety
  \[
    X''(\cC):=\Xq(\cC)\setminus\BIGCUP_{\cC\subsetneq\cC'\in\cF}
    \Xq(\cC').
  \]
  Then \eqref{eq:inter1} implies $X''(\cC)\cap X\el(\cF)=X\el(\cC)$
  and therefore $X''(\cC)=X(\cC)$ by \eqref{eq:remark}. This proves
  \ref{it:boundary2} and shows that $X(\cC)$ is an open subset of
  $\Xq(\cC)$. The rest of \ref{it:boundary1} follows easily with the
  help of \eqref{eq:bigcup}. All assertions in \ref{it:boundary3}
  follow from \ref{it:boundary2} and the corresponding assertions for
  $X\el(\cC)$. The same holds for the assertions in
  \ref{it:boundary4}, except for the final one regarding
  $\cZ_k$. Using equation \eqref{eq:diagXi22}, the assertion can be
  reduced to the case $k=K$. Then the case where $\cC$ is a ray is
  handled in \Cite{KnopIB}{Satz\ 7.4.5}. In the general case, choose
  any extremal ray $\cR$ of $\cC$. Then $\Xq(\cR)$ is a toroidal
  embedding of $X(\cR)$ containing $X(\cC)$, and we conclude by
  induction.
\end{proof}

\begin{remarks}\label{rem:XcF}

  \emph{i)} By shrinking $X\an$ and thereby $X(\cF)$, we can ensure
  that any finite list of ``generic'' conditions holds for all
  strata. For example: all strata $X(\cC)$ admit an orbit space
  $X(\cC)/G$; each stratum, its orbit space, and the quotient map are
  smooth; all isotropy groups in a stratum have $K$-isomorphic Levi
  complements, etc.

  \emph{ii)} In general, $X(\cF)$ is only defined if $X$ meets the
  conditions of \cref{thm:CEmbedding}. Thus, when we speak about
  $X(\cF)$, we mean $X_0(\cF)$ where $X_0=G\cdot X\el$ is an open
  subset depending on the choice of $X\el$. In this sense, $X(\cF)$ is
  only unique up to a birational map defined on all
  strata. Observe that in the most important case, namely when $X$ is
  homogeneous and $k$-spherical, $X(\cF)$ is an unambiguously defined
  embedding of $X$, since all strata are orbits; see
  \cref{cor:spherXF} below.

  \emph{iii)} The dimension of the generic $G$-orbits in $X(\cC)$ is
  not easy to control. An upper bound for the codimension is
  $c_k(X)$. Moreover, we show in \cref{sec:horospherical} that this
  bound is attained as soon as $\cC$ contains an interior point of
  $\cZ_k(X)$. This is, in particular, the case when $\dim\cC=\rk_kX$.

\end{remarks}

As usual, for $k$-spherical varieties, we can be more specific:

\begin{corollary}\label{cor:spherXF}

  Let $X$ be a homogeneous $k$-spherical variety and let
  $X\into X(\cF)$ be the embedding corresponding to a fan $\cF$ with
  $\supp\cF\subseteq\cZ_k(X)$. Then the strata $X(\cC)$ are precisely
  the $G$-orbits. In particular, all $G$-orbits in $X(\cF)$ are
  $k$-dense.

\end{corollary}

\begin{proof}
  The group $MA$ acts transitively on $X\an$ and therefore on
  $X\el(\cC)$. Thus, $G$ acts transitively on
  $X(\cC)=G\cdot X\el(\cC)$.
\end{proof}

For local fields, the set $X(\cF)(k)$ is also stratified:

\begin{proposition}\label{prop:stratification}

  Assume that $k$ is a local field. Then the sets $X(\cC)(k)$ with
  $\cC\in\cF$ form a stratification of $X(\cF)(k)$ for the Hausdorff
  topology.

\end{proposition}

\begin{proof}
  Clearly, each set $X(\cC)(k)$ is locally closed and its closure is
  contained in the union of all sets $X(\cC')(k)$ with
  $\cC\subseteq\cC'\in\cF$. It remains to be shown that, conversely,
  each such set $X(\cC')(k)$ is contained in the closure of
  $X(\cC)(k)$. Hilberts Satz 90 implies that the fibration
  $X\el\to X\an$ is Zariski locally trivial. Thus, the construction of
  $X(\cF)$ shows that it is covered by open subsets of the form
  $Y\times A(\cF)$ where $Y$ is some $k$-variety.

  Toroidal theory shows that the choice of a one-parameter subgroup
  $a:\G_m\to A_k(X)$ in the relative interior of $\cC'$ yields a
  surjective morphism
  \[
    \pi:A(\cC)\to A(\cC'):x\mapsto \lim_{t\to\infty}a(t)x.
  \]
  Moreover, $\pi$ is a homomorphism of split tori with connected
  kernel. Hence, it has a section, showing that the induced map
  $A(\cC)(k)\to A(\cC')(k)$ on $k$-rational points is surjective as
  well. This implies the claim.
\end{proof}

Next, we study morphisms to and from toroidal embeddings. Let
$\phi:\cN_1\to\cN_2$ be a homomorphism between two finite-dimensional
$\QQ$-vector spaces. Let $\cF_1$ and $\cF_2$ be fans supported in
$\cN_1$ and $\cN_2$, respectively. We say that \emph{$\cF_1$ dominates
  $\cF_2$} if for every $\cC_1\in\cF_1$, there is a $\cC_2\in\cF_2$
with $\phi(\cC_1)\subseteq\cC_2$. On the other hand, let $X\to\Xq_1$
and $X\to\Xq_2$ be two dominant $G$-morphisms (e.g., embeddings). Then
\emph{$\Xq_1$ dominates $\Xq_2$} if the identity on $X$ extends to a
morphism $\Xq_1\to\Xq_2$. It is well known (by reduction to the torus
case) that $X(\cF_1)$ dominates $X(\cF_2)$ if and only if $\cF_1$
dominates $\cF_2$. We emphasize that here $X(\cF_1)$ and $X(\cF_2)$
are toroidal embeddings of one and the same variety $X$, so that
$\cF_1$ and $\cF_2$ are fans in the same space
$\cN_1=\cN_2=\cN_k(X)$ and $\phi$ is the identity. In this case,
$\cF_1$ dominates $\cF_2$ simply means that every cone of $\cF_1$ is
contained in some cone of $\cF_2$. More generally:

\begin{proposition}\label{prop:dominate}

  Let $X$ be a homogeneous $k$-spherical variety. Let $X\into\Xq$ be a
  normal equivariant embedding dominating the toroidal embedding
  $X(\cF)$. Then $\Xq=X(\overline\cF)$ for some fan $\overline\cF$
  dominating $\cF$.

\end{proposition}

\begin{proof}
  Let $\phi:\Xq\to X(\cF)$ be the dominating morphism and put
  $\Xq_0:=\phi^{-1}(X\el(\cF))$. Then $U\times \Xq_0\to\Xq$ is an open
  embedding and $\Xq=G\cdot\Xq_0$. Moreover, since $X$ is spherical,
  $X\an\cong\Xq\an$ is homogeneous. Therefore, $\Xq_0$ is of the form
  $X\el\times_{X\an}\Aq$ where $\Aq$ is an embedding of $A_k(X)$
  dominating $A(\cF)$. It then follows from the theory of torus
  embeddings that $\Aq=A(\overline\cF)$ for some fan $\overline\cF$
  dominating $\cF$. This implies that the identity on $X$ extends to a
  $G$-morphism $X(\overline\cF)\to\Xq$ which maps $X\el(\overline\cF)$
  isomorphically onto $\Xq_0$. This morphism is surjective and an
  open embedding, thus an isomorphism.
\end{proof}

The reverse direction is more involved.

\begin{theorem}\label{prop:morphism2}

  Let $X$ be a homogeneous, $k$-spherical $G$-variety and let
  $\phi:X\to Y$ be a $G$-equivariant $k$-morphism to a locally
  $k$-linear $G$-variety $Y$. For a fan $\cF$ supported in $\cZ_k(X)$,
  let $X_0\subseteq X(\cF)$ be the largest open subset to which $\phi$
  can be extended to a morphism $\phi_0:X_0\to Y$. Then an orbit
  $X(\cC)$ is contained in $X_0$ if and only if there is a
  Hausdorff-dense subset $\cS\subseteq\cC$ such that all $v\in\cS$
  have the same non-empty center in $Y$.

\end{theorem}

\begin{proof}
  Assume first that $X_0\cap X(\cC)\ne\leer$. We claim that
  $\cS=\cC^0$, the relative interior of $\cC$, suffices. Indeed,
  $v\in\cC^0$ implies that $X(\cC)$ is the center of $v$ in
  $X_0$. Thus, the closure of $\phi_0(X_0\cap X(\cC))$ in $Y$ is the
  center of $v$, independently of the choice of $v$.
 
  To show the converse, observe that the assertion is local in
  $X(\cF)$. We may assume that $\cF$ consists of the faces of $\cC$ or,
  equivalently, that $X(\cC)$ is the only closed stratum of $X(\cF)$.

  The assertion is also local in $Y$, so we may assume that $Y$ is a
  subvariety of a projective space $\P^n$ with a linear $G$-action. It
  suffices to prove the assertion for the closure $\Yq$ of $Y$ in
  $\P^n$. Indeed, let $\phi_1:X_1\to\Yq$ be the maximal extension of
  $\phi$ and put $X_0:=\phi_1^{-1}(Y)$. By assumption, there is a
  dense subset $\cS\subseteq\cC$ and a closed subvariety
  $Z\subseteq Y$ which is the common center of every $v\in\cS$. Its
  closure $\Zq\subseteq\Yq$ is the center of $v$ in $\Yq$. Since
  $X(\cC)$ is the center of $v$ in $X(\cF)$ and because
  $X(\cC)\cap X_1\ne\leer$, the morphism
  $\phi_1:X(\cC)\cap X_1\to \Zq$ is dominant. Hence, the preimage
  $X(\cC)\cap X_0=\phi_1^{-1}(\Zq\cap Y)=\phi_1^{-1}(Z)$ is
  non-empty. We may therefore assume from now on that $Y$ is
  projective.

  Let $\XS$ be the normalization of the closure of the graph of
  $\phi:X\to Y$ in $X(\cF)\times Y$. Since $\XS\to X(\cF)\times Y$ is
  finite, the restriction of the second projection $\XS\to Y$ to any
  fiber of the first projection $\XS\to X(\cF)$ is also finite.

  The variety $\XS$ dominates $X(\cF)$ via the projection $\pi$ to the
  first factor. \cref{prop:dominate} implies that $\XS$ is of the form
  $X(\cFS)$ where $\cFS$ is a fan supported in $\cC$. Thus, we get
  \[\label{eq:morphism2}
    \cxymatrix{X\ar@{^(->}[r]\ar[dr]_\phi&
      X(\cFS)\ar[r]^\pi\ar[d]_\phiS& X(\cF)\ar@{.>}[dl]^{\phi_0}\\&Y}
  \]
  Since $Y$ is complete, the morphism $\pi$ is proper. Hence, the
  support of $\cFS$ is all of $\cC$. The assertion follows if we show
  that $\cFS=\cF$, since then $\pi$ is invertible over an open subset
  $X_0\subseteq X(\cF)$ with $X_0\cap X(\cC)\ne\leer$.

  Suppose $\cFS\ne\cF$. Then $\cFS$ contains more than one cone of
  maximal dimension and therefore also a cone $\cCS$ of codimension
  one separating two of them, say $\cC_1$ and $\cC_2$.  The closure
  $\ZS$ of $X(\cCS)$ in $X(\cFS)$ consists of three strata, namely,
  $X(\cCS)$, $X(\cC_1)$, and $X(\cC_2)$. The morphism $\pi$ maps $\ZS$
  to $X(\cC)$. The fiber $\ZS_x$ over any point $x\in X(\cC)(k)$ is
  isomorphic to the projective line $\P_k^1$. If $x$ is chosen inside
  the subset $U\times X\el(\cF)$, as in \cref{thm:CEmbedding}, then
  $P_x$ acts on $\ZS_x$ with an open orbit and two fixed points:
  $\ZS_x\cap X(\cC_i)=\{p_i(x)\}$.

  Since $\cS$ is dense in $\cC$, we may choose valuations
  $v_i\in\cS\cap\cC_i^0$. The center of $v_i$ in $X(\cFS)$ is
  $X(\cC_i)$ while, by assumption, $v_1$ and $v_2$ have the same
  center $Z$ in $Y$. It follows that $\phiS$ maps both $X(\cC_1)$ and
  $X(\cC_2)$ dominantly to $Z$.

  By \cref{cor:P-stable-affine}, there exists a non-empty $P$-stable
  affine open subset $Y^0\subseteq Y$ with $Y^0\cap Z\ne\leer$. Thus,
  there is an $x\in X(\cC)(k)\cap(U\times X\el(\cF))$ such that the
  images of both fixed points $p_i(x)$ of the fiber $\ZS_x$ lie in
  $Y^0\cap Z$. Since this condition is open, a non-fixed point of
  $\ZS_x$ also maps to $Y^0$. By $P$-equivariance, the entire fiber
  $\ZS_x$ is mapped into $Y^0$. However, this is impossible because
  $\ZS_x$ is complete, $Y^0$ is affine, and $\ZS_x\to Y$ is finite.
\end{proof}

\begin{corollary}\label{cor:morphism}

  Let $X\into X(\cF)$ and $Y\into Y(\cF')$ be two toroidal embeddings
  and let $\phi:X\to Y$ be a dominant equivariant morphism of
  homogeneous $k$-spherical $G$-varieties. Assume that $\cF$ dominates
  $\cF'$ with respect to $\phi_*:\cN_k(X)\to\cN_k(Y)$. Then there
  exists a $G$-invariant, dense, open subset $X_0\subseteq X(\cF)$
  that meets all strata and dominates $Y(\cF')$. In particular, if $X$
  is $k$-spherical then $X(\cF)$ dominates $X(\cF')$.

\end{corollary}

A fan $\cF$ whose support is all of $\cZ_k(X)$ is called
\emph{complete (for $X$)}. While $X(\cF)$ is not always complete in
the algebro-geometric sense, it satisfies the following completeness
property for $k$-rational points:

\begin{theorem}\label{thm:CompleteGeneral}

  Let $\cF$ be a complete fan, let $X(\cF)\into\Xq$ be a locally
  $k$-linear equivariant embedding, and let $Y\subseteq\Xq$ be a
  $k$-dense, $G$-stable subvariety with $c_k(Y)=c_k(X)$. Then
  $Y\cap X(\cF)\ne\leer$.

\end{theorem}

\begin{proof}
  \cref{lem:centercentral} asserts that $Y$ is the center of some
  $v\in\cZ_k(X)$. Since $v\in\supp\cF$, it has a center in $X(\cF)$
  (specifically, $\Xq(\cC)$ where $v$ lies in the relative interior of
  $\cC$). Thus, $Y\cap X(\cF)\ne\leer$.
\end{proof}

For $k$-spherical varieties, this implies that the complement of
$X(\cF)$ cannot contain $k$-rational points. Thus, $X(\cF)(k)$ is
maximal in the following sense:

\begin{corollary}\label{cor:rational}

  Let $X$ be a $k$-spherical variety and let $X(\cF)\into\Xq$ be a
  locally $k$-linear equivariant embedding. Then $X(\cF)(k)=\Xq(k)$.

\end{corollary}

For local fields, we obtain a true compactification:

\begin{corollary}\label{cor:compact}

  Assume that $k$ is a local field and let $X$ be a $k$-spherical
  variety. Then $X(\cF)(k)$ is compact in the Hausdorff topology if
  and only if $\cF$ is complete.

\end{corollary}

\begin{proof}
  If $\cF$ is complete, then by \cite{Sumihiro} there exists an
  equivariant embedding of $X(\cF)$ into a normal complete variety
  $\Xq$. Then $X(\cF)(k)=\Xq(k)$ is compact. Conversely, suppose $\cF$
  is not complete. Then there is a fan $\cF'$ properly containing
  $\cF$. We claim that $X(k)$ is Hausdorff-dense in $X(\cF')(k)$. By
  \cref{prop:stratification}, it suffices to show that $A_k(k)$ is
  dense in $\Aq_k=\Aq_k(\cF')$. This is clear, because for every
  $\cC\in\cF'$, any one-parameter subgroup of $A_k$ in the relative
  interior $\cC^0$ contracts $A_k$ to $A_k(\cC)$. Thus, $X(\cF)(k)$ is
  a proper, dense subset of $X(\cF')(k)$ and cannot be compact.
\end{proof}

Since a fan consists of strictly convex cones, there is generally no
canonical complete fan unless the valuation cone itself is strictly
convex.

\begin{definition}\label{def:standard}

  A $k$-dense $G$-variety is called \emph{$k$-semisimple} if its
  valuation cone $\cZ_k(X)$ is strictly convex. In that case, we
  define the \emph{standard fan $\cF\st$} as the set of faces of
  $\cZ_k(X)$ and the \emph{standard embedding $X\st:=X(\cF\st)$}.

\end{definition}

Note that Remark~\ref{rem:XcF}\emph{ii)} also applies to the standard
embedding, i.e., $X\rat X\st$ is merely birational and unique in a
birational sense unless $X$ is homogeneous and $k$-spherical. We show
in \cref{sec:horospherical} that the condition of $k$-semisimplicity
is not very restrictive.

By construction, $X\st$ has a unique closed stratum $X(\cC)$ for
$\cC=\cZ_k(X)$. If $X$ is $k$-spherical, then $X\st$ is a \emph{simple
  embedding}, meaning that it has a unique closed $G$-orbit.

\begin{proposition}

  Let $X\into X\st$ be the standard embedding of a homogeneous,
  $k$-semisimple, $k$-spherical variety $X$. Let $Y$ be a complete
  $G$-variety that contains at most one $k$-dense orbit of rank
  zero. Then every $G$-equivariant $k$-morphism $\phi:X\to Y$ extends
  uniquely to a morphism $\phi_0:X\st\to Y$.

\end{proposition}

\begin{proof}
  As in the proof of \cref{prop:morphism2}, there is a fan $\cFS$
  fitting into the diagram
  \[
    \cxymatrix{X(\cFS)\ar[r]^\pi\ar[d]_\phiS&
      X(\cF\st)\ar@{.>}[dl]^{\phi_0}\\Y}
  \]
  with $\|supp|\cFS=\|supp|\cF\st=\cZ_k(X)$. Let
  $\cC_1,\ldots,\cC_n\in\cFS$ be the cones of maximal dimension and
  let $Z_i\subseteq Y$ be the closure of the image of $X(\cC_i)$ in
  $Y$. Since $X(\cC_i)$ is $k$-dense of rank $0$, the same holds for
  the $Z_i$. By assumption, all $Z_i$ are equal, so every
  $v\in\cS:=\cC_1^0\cup\ldots\cup\cC_n^0$ has the same center in
  $Y$. The assertion follows from \cref{prop:morphism2}.
\end{proof}

Following \Cite{BorelTits}{12.3}, we call an irreducible
representation $V$ of $G$ \emph{strongly rational} if it admits a
$P$-eigenvector. In this case, $\P(V)^{AN}$ consists of exactly one
point, and its orbit is the unique orbit of rank zero.

\begin{corollary}

  Let $X=G/H$ be $k$-semisimple and $k$-spherical. Let $V$ be a
  strongly rational representation of $G$. Then every equivariant
  $k$-morphism $G/H\to\P(V)$ extends to a morphism $X\st\to\P(V)$.

\end{corollary}

\section{Horospherical varieties}
\label{sec:horospherical}

We aim to understand the closed strata of $X(\cF)$ when $\cF$ is a
complete fan. Since, by \cref{prop:boundary}, these strata are
$G$-varieties of $k$-rank $0$, we study rank-$0$ varieties in
general. Over an algebraically closed field, a homogeneous variety of
rank zero is known to be a flag variety. Over arbitrary fields, one
obtains a wider class of spaces that nevertheless share many
properties with flag varieties.

To formulate our result, we need to introduce the process of parabolic
induction, which is also of independent interest.

\begin{definition}

  Let $Q=LU\subseteq G$ be a parabolic subgroup containing $P$, and
  let $Q^-=U^-L$ be the opposite parabolic subgroup.

  \begin{enumerate}

  \item Let $Y$ be an $L$-variety. Consider it as a $Q^-$-variety via
    the projection $Q^-\auf L$. Then $G\times^{Q^-}Y$ is the
    \emph{parabolic induction of $Y$ to a $G$-variety}.

  \item Similarly, the parabolic induction of a $k$-subgroup
    $H\subseteq L$ is the subgroup $U^-H$ of $G$. In that case, the
    homogeneous variety $G/U^-H$ is parabolically induced from $L/H$.

  \end{enumerate}

\end{definition}

Parabolic induction preserves most properties:

\begin{proposition}\label{prop:induction}

  Let $Q=LU\subseteq G$ be a parabolic subgroup containing $P$ and
  $Q^-=U^-L$ its opposite. Let $Y$ be a $k$-dense $L$-variety, and let
  $X=G\times^{Q^-}Y$ be its parabolic induction. Then:

  \begin{enumerate}[noitemsep]

  \item\label{it:induction0} $X$ is $k$-dense.

  \item\label{it:induction1} $Q_k(X)=UQ_k(Y)$.

  \item\label{it:induction2} $X\el=Y\el$ and $X\an=Y\an$.

  \item\label{it:induction4} $c_k^G(X)=c_k^L(Y)$ and
    $\rk_k^G(X)=\rk_k^L(Y)$.

  \item\label{it:induction3} $\Xi_k(X)=\Xi_k(Y)$, $\cN_k(X)=\cN_k(Y)$,
    and $\cZ_k(X)=\cZ_k(Y)$.

  \end{enumerate}

\end{proposition}

\begin{proof}
  The Local Structure Theorem for $Y$ yields an open embedding
  $R_uQ_k(Y)\times Y\el\into Y$. Combined with the open embedding
  $U\into G/Q^-$ (the big cell), we obtain the open embeddings
  \[
    U\times R_uQ_k(Y)\times Y\el\into U\times Y\into X.
  \]
  Assertions \ref{it:induction0}, \ref{it:induction1}, and the first
  part of \ref{it:induction2} follow immediately from
  \cref{prop:uniqueQ}. These imply all other assertions except the
  last regarding valuation cones. Because $K(X)^U=K(Y)$, the
  restriction of any $k$-central valuation to $K(Y)$ is $k$-central as
  well. This shows $\cZ_k(X)\subseteq\cZ_k(Y)$. Now let
  $v_0\in\cZ_k(Y)$ and let $(\Yq,D_0)$ be a model of $v_0$. Let
  $D:=G\times^{Q^-}D_0\subseteq\Xq:=G\times^{Q^-}\Yq$. Then the
  restriction of $v_D$ to $K(Y)$ is equivalent to $v_0$, which shows
  the opposite inclusion.
\end{proof}

We proceed with the classification of varieties of $k$-rank $0$.

\begin{theorem}\label{thm:rank0}

  For a $k$-dense $G$-variety $X$, the following are equivalent:

  \begin{enumerate}[noitemsep]

  \item\label{it:rank0-1} $\rk_kX=0$.

  \item\label{it:rank0-2} An open dense $G$-stable subset of $X$ is
    induced from an anisotropic action.

  \item\label{it:rank0-3} $X=G\cdot X^{AN}$.

  \end{enumerate}

\end{theorem}

\begin{proof}
  \ref{it:rank0-1}$\Rightarrow$\ref{it:rank0-2} We apply the Generic
  Structure Theorem \ref{cor:LSTgeneric} to $X$. Here, $\rk_kX=0$
  means that $A$ acts trivially on $X\el$; i.e., the action of $L$ on
  $X\el$ is anisotropic. Let $Q^-=LU^-$ be the opposite parabolic of
  $Q$, and let $\fu^-=\Lie U^-$. We claim that $X\el$ consists of
  $U^-$-fixed points.

  For this, it suffices to show that $\fu^-x=0$, where $x\in X\el$ is
  any smooth point. Since $T_xX=\fu x\oplus T_xX\el$, all weights of
  $A$ on $T_xX$ are either zero or positive restricted roots. On the
  other hand, all weights of $A$ on $\fu^- x$ are negative
  roots. Thus, $\fu^-x=0$.

  From this claim, we infer that $X\el$ is a $Q^-$-variety with $U^-$
  acting trivially. Now consider the (proper) $G$-morphism
  \[
    \phi:\Xq:=G\times^{Q^-}X\el\to X:[g,x]\mapsto gx.
  \]
  Since the multiplication map $U\times Q^-\to G$ is an open
  embedding, we can think of $U\times X\el$ as an open subset of
  $\Xq$. By the Generic Structure Theorem \ref{cor:LSTgeneric}, the
  restriction of $\phi$ to $U\times X\el$ is an open embedding, which
  shows that $\phi$ is birational. Thus, the largest open subset
  $X_0\subseteq\Xq$ on which $\phi$ is an open embedding is
  non-empty. Since $X_0$ is also $G$-stable, it is induced from
  $Y=X_0\cap X\el$, which proves the assertion.

  \ref{it:rank0-2}$\Rightarrow$\ref{it:rank0-3} Let
  $X_0:=G\times^{Q^-}Y$ be an open subset of $X$, where $Y$ is an
  anisotropic $L$-variety. Since $Y$ is also anisotropic as a
  $Q^-$-variety and $AN^-\subseteq P^-\subseteq Q^-$, it follows that
  $AN^-$ acts trivially on $Y$. This implies that $G\cdot X^{AN^-}$ is
  dense in $X$. On the other hand, the morphism
  $G\times^{Q^-}X^{AN^-}\to X$ is proper. Therefore, its image
  $G\cdot X^{AN^-}$ is closed, and hence coincides with $X$. Since
  $AN^-$ is $G(k)$-conjugate to $AN$, we have
  $G\cdot X^{AN}=G\cdot X^{AN^-}=X$.
  
  \ref{it:rank0-3}$\Rightarrow$\ref{it:rank0-1} Since
  $X=G\cdot X^{AN}$, there is an irreducible component $Y$ of $X^{AN}$
  such that the morphism $\XS:=G\times^PY\to X$ is surjective. Then
  $\rk_k^G\XS=\rk_k^{MA}Y=0$ implies $\rk_kX=0$
  (\cref{lem:GAcentral}).
\end{proof}

\begin{remark}

  One can show that if a locally $k$-linear $X$ has $k$-rank zero,
  then there even exists an open subset \emph{containing $X(k)$} which
  is induced from an anisotropic action.

\end{remark}

For homogeneous varieties, we have:

\begin{corollary}\label{cor:homrk0}

  Let $H\subseteq G$ be a $k$-subgroup. Then the following are
  equivalent:

  \begin{enumerate}[noitemsep]

  \item\label{it:homrk0-1} $\rk_kG/H=0$.

  \item\label{it:homrk0-3} There is a parabolic subgroup
    $Q\subseteq G$ with $Q\an\subseteq H\subseteq Q$.

  \item\label{it:homrk0-2} $H$ contains a $G(k)$-conjugate of $AN$.

  \item\label{it:homrk0-4} Let $\phi:G/H\into\Xq$ be any locally
    $k$-linear $G$-equivariant open embedding. Then $\Xq(k)=(G/H)(k)$.

  \end{enumerate}

  If $k$ is a local field, then these conditions are also equivalent
  to:

  \begin{enumerate}[noitemsep,resume]

  \item\label{it:homrk0-5} $(G/H)(k)$ is compact.

  \item\label{it:homrk0-6} $G(k)/H(k)$ is compact.

  \end{enumerate}

\end{corollary}

\begin{proof}
  \ref{it:homrk0-1}$\Rightarrow$\ref{it:homrk0-3} By \cref{thm:rank0},
  there is an isomorphism $G/H\cong G\times^{Q^-}Y$, where $Y$ is an
  anisotropic homogeneous $L$-variety. From this, we obtain a
  projection $\pi:G/H\to G/Q^-$. Since $G(k)$ acts transitively on
  $(G/Q^-)(k)$ (see \Cite{BorelTits}{Thm.\ 4.13 a)}), we may assume
  that $\pi(eH)=eQ^-$, i.e., $H\subseteq Q^-$ and $Q^-/H\cong
  Y$. Thus, $Q^-\an\subseteq H\subseteq Q^-$.

  \ref{it:homrk0-3}$\Rightarrow$\ref{it:homrk0-2} There exists
  $g\in G(k)$ such that $gPg^{-1}\subseteq Q$ (\Cite{BorelTits}{Thm.\
    4.13 b)}). Then $gANg^{-1}\subseteq Q\an\subseteq H$.

  \ref{it:homrk0-2}$\Rightarrow$\ref{it:homrk0-1} Use the implication
  \ref{it:rank0-3}$\Rightarrow$\ref{it:rank0-1} from \cref{thm:rank0}.

  \ref{it:homrk0-1}$\Rightarrow$\ref{it:homrk0-4} Since $G/H$ is
  $k$-spherical (by \ref{it:homrk0-2}), the assertion follows from
  \cref{cor:rational}.

  \ref{it:homrk0-4}$\Rightarrow$\ref{it:homrk0-1} Suppose
  $\rk_kG/H>0$. Since $\cZ_k(G/H)\ne\{0\}$ by \cref{cor:valcone},
  there is a non-trivial $k$-central valuation.  It belongs to a
  smooth embedding $G/H\into X=G/H\cup D$, where $D$ is
  $k$-dense. Then $X(k)\not\subseteq (G/H)(k)$.

  The equivalence of \ref{it:homrk0-2}, \ref{it:homrk0-5}, and
  \ref{it:homrk0-6} is due to Borel--Tits \Cite{BorelTits}{Prop.\
    9.3}.
\end{proof}

Our main application is to toroidal embeddings of $k$-spherical
varieties:

\begin{corollary}

  Assume $X$ is a homogeneous $k$-spherical variety and let
  $X\into X(\cF)$ be the embedding corresponding to a fan $\cF$ with
  $\supp\cF=\cZ_k(X)$. Then all closed orbits are isomorphic to
  $G/H_0$, where $H_0=M_0Q_k(X)^-\an$ with $X\an\cong M/M_0$.

\end{corollary}

Our next goal is to understand varieties with $\cZ_k(X)=\cN_k(X)$. For
this, we start with the more general problem of studying
\[
  \cZ_k^0=\cZ_k(X)\cap(-\cZ_k(X)),
\]
the largest subspace contained in $\cZ_k(X)$. The importance of this
space lies in the fact that $\cZ_k^0$, when non-zero, prevents $X$
from having a standard embedding as in \cref{def:standard}.

\begin{definition}\label{def:central}

  Let $X$ be a $k$-dense $G$-variety. An equivariant $K$-automorphism
  $\phi$ of $K(X)$ is called \emph{$k$-central} if
  $K(X)^{(AN)}\subseteq K(X)^{(\phi)}$; i.e., if every
  $AN$-semiinvariant is a $\phi$-semiinvariant. The group of
  $k$-central automorphisms will be denoted by $\fA_k(X)$.

\end{definition}

Since every $B$-semiinvariant is also an $AN$-semiinvariant, all
$k$-central automorphisms are also $K$-central. The latter have been
studied in detail in \Cite{KnopIB}{Sec.\ 8} and \Cite{KnopAut}{Sec.\
  5}. It was shown that the group of $K$-central automorphisms is the
set of $K$-rational points of a subgroup $\fA_K(X)\subseteq
A_K(X)$. More precisely, a $K$-central automorphism $\phi$ is related
to an element $a_\phi\in A_K(X)$ by the formula
\[\label{eq:phif}
  f^\phi:=\phi(f)=\chi_f(a_\phi)f\quad\text{for all }f\in
  K(X)^{(B)}.
\]

\begin{lemma}\label{lemma:definedoverk}
  Let $X$ be a $k$-dense $G$-variety. Then $A_K=A_K(X)$,
  $\fA_K=\fA_K(X)$, the inclusion $\fA_K\into A_K$, and the action of
  $\fA_K$ on $X$ are all defined over $k$. Moreover,
  \[\label{eq:fAA}
    \fA_k(X)=\fA_K(X)\cap A_k(X),
  \]
  which implies that $\fA_k(X)$ is a diagonalizable $k$-subgroup of
  $\fA_K(X)$.
\end{lemma}

\begin{proof}
  Recall from \cref{sec:notation} the $*$-action
  $\gamma*\chi=w_\gamma\,{}^\gamma\chi$ of $\cG$ on $\Xi(T)$. Let
  $n_\gamma\in N_G(T)$ be a lift of $w_\gamma$ and let
  $\ogamma=n_\gamma\gamma$, regarded as an element of
  $G(K)\rtimes\cG$. By construction, we have
  ${}^\ogamma B:=n_\gamma \gamma(B)n_\gamma^{-1}=B$. Moreover, $\ogamma$
  acts on a function $f\in K(X)$ by
  ${}^\ogamma f(x)={}^\gamma f(\ogamma^{-1}x)$. A direct check shows that
  if $f_\chi$ is a $B$-semiinvariant on $X$ for the character $\chi$,
  then ${}^\ogamma f_\chi$ is a $B$-semiinvariant for the character
  $\gamma*\chi={}^\ogamma\chi$. This shows that $\Xi(A_K)$ is a
  $\cG^*$-stable subgroup of $\Xi(T)$. Hence $A_K$ and the projection
  $T^*\to A_K$ are defined over $k$, where $T^*$ is the absolute
  maximal torus of $G$, i.e., the maximal torus equipped with the
  $\cG^*$-action.

  Let $\phi\in\fA_K$ be $K$-central. Applying \eqref{eq:phif} to
  $f={}^{\ogamma^{-1}}\!f_\chi$ and evaluating it at $\ogamma^{-1}x$
  yields
  \[\label{eq:fchiogamma}
    f_\chi(\ogamma\phi\ogamma^{-1}\ x)=c\, f_\chi(x)\text{ with
    }c:=\gamma\left((\gamma^{-1}*\chi)(a_\phi)\right)= \chi({}^\ogamma
    a_\phi)\in K.
  \]
  This implies that $\gamma\phi\gamma^{-1}=\ogamma\phi\ogamma^{-1}$ is
  $K$-central. Therefore, $\fA_K$ and the action of $\fA_K$ on $X$ are
  defined over $k$. From \eqref{eq:fchiogamma}, we then have
  $\chi(a_{\gamma\phi\gamma^{-1}})=\chi({}^\ogamma a_\phi)$, hence
  $a_{\gamma\phi\gamma^{-1}}={}^\ogamma a_\phi$. This proves that the
  inclusion $\fA_K\into A_K$ is defined over $k$.

  Finally, if $\phi$ is $k$-central, it must act trivially on
  $K(X)^{AN}$. In view of \cref{lemma:XiResXi}, this means that $\phi$
  is $k$-central if and only if $\chi(a_\phi)=1$ whenever
  $\res_A\chi=0$. From this, \eqref{eq:fAA} follows.
\end{proof}

Central automorphisms are related to central valuations in the
following way:

\begin{theorem}\label{thm:auto}

  Let $X$ be a $k$-dense $G$-variety. Then $\cN(\fA_k(X))=\cZ_k^0(X)$.

\end{theorem}

\begin{proof}
  From \Cite{KnopIB}{Satz\ 8.2}, we get
  $\cN(\fA_K(X))=\cZ_K^0(X)$. Now use \eqref{eq:diagXi22} and
  \eqref{eq:fAA}.
\end{proof}

\begin{corollary}\label{cor:finiteconvex}

  A $k$-dense $G$-variety $X$ is $k$-semisimple if and only if
  $\fA_k(X)$ is finite.

\end{corollary}

From this, we derive a sufficient condition for the existence of a
standard embedding:

\begin{corollary}\label{cor:centralsplit}

  Let $X=G/H$ be homogeneous and assume that $\Aut^G(X)=N_G(H)/H$ does
  not contain a non-trivial central split torus (e.g., if $H$ is of
  finite index in its normalizer). Then $X$ is $k$-semisimple.

\end{corollary}

\begin{proof}
  The identity component $\fA_k(X)^0$ is a split torus sitting in the
  center of $\Aut^G(X)$ (see \Cite{KnopAut}{Cor.\ 5.6}). Thus,
  $\fA_k(X)^0=1$.
\end{proof}

For $k$-spherical varieties, one can be more specific:

\begin{proposition}\label{prop:anisoconvex}

  Let $X=G/H$ be a homogeneous $k$-spherical variety. Then
  $\Aut^G(X)=N_G(H)/H$ is an elementary group such that $\fA_k(X)^0$
  is its unique maximal split subtorus. In particular, $X$ is
  $k$-semisimple if and only if $N_G(H)/H$ is anisotropic.

\end{proposition}

\begin{proof}
  Since $X$ is $k$-spherical, the parabolic $Q:=Q_k(X)=LU$ has an open
  dense orbit $Qx$ in $X$ such that the action of $MA\subseteq L$ on
  $X\el=Qx/U$ is transitive (\cref{cor:LSTspherical} and
  \eqref{eq:H=PHel}). In this way, we obtain a homomorphism from
  $\Aut^G(X)$ to a subquotient of $MA$. Any unipotent element
  $\nu\in\Aut^G(X)$ therefore acts trivially on $X\el$ and hence on
  $K(X\el)^{(A)}=K(X)^{(AN)}$. Thus, $\nu$ is both $k$-central and
  unipotent, hence trivial. This shows that $\Aut^G(X)$ is elementary.
  Its (unique) largest split subtorus $S$ maps into
  $S':=\Aut^{MA}(X\el)$, which is an elementary group containing $A$
  as its largest split subtorus. Thus, $S$ maps into $A$, which means
  that all $A$-semiinvariants on $X\el$ are $S$-semiinvariants; this
  implies $S\subseteq\fA_k(X)^0\subseteq S$. The last assertion
  follows from \cref{thm:auto}.
\end{proof}

\cref{thm:auto} also provides a way to make any $G$-variety
$k$-semisimple by taking the quotient by $\fA_k(X)$. More generally,
the following holds:

\begin{lemma}\label{lem:auto}

  Every $k$-dense $G$-variety $X$ contains an open
  $G\times\fA_k(X)$-stable subset $X_0$ for which the orbit space
  $X_0/\fA_k(X)$ exists. Moreover, for any $k$-subgroup
  $\fA\subseteq\fA_k(X)$, the quotient $X':=X_0/\fA$ has the following
  properties:
  \begin{align*}
    &\Xi_k(X')=\{\chi\in\Xi_k(X)\mid\res_\fA\chi=0\},\ A_k(X')=A_k(X)/\fA,\\
    &X'\el=X\el/\fA,\ X'\an=X\an,\ c_k(X')=c_k(X),\ \rk_kX'=\rk_kX-\dim\fA,\\
    &\cN_k(X')=\cN_k(X)/\cN(\fA),\ \cZ_k(X')=\cZ_k(X)/\cN(\fA).
  \end{align*}

\end{lemma}

\begin{proof}
  By \Cite{KnopAut}{Cor.\ 5.4}, there is an open $G$-stable subset of
  $X$ on which the birational action of $\fA_k(X)$ is regular. Then
  the existence of $X_0$ is a consequence of Rosenlicht's theorem.

  The first two properties of $X'$ follow from \eqref{eq:phif}. That
  formula also implies the compatibility of the $\fA$-action on $X\el$
  with the action induced from that of $A_k(X)$. This shows
  $X'\el=X\el/\fA$, which implies all other assertions except the last
  one regarding the valuation cone. That follows from its
  $K$-counterpart, \Cite{KnopIB}{Satz\ 8.1 (4.)}.
\end{proof}

\begin{corollary}

  Let $X$ be a $k$-dense $G$-variety and let $\fA\subseteq\fA_k(X)$ be
  an open $k$-subgroup. Then there is a $G\times\fA$-stable open dense
  subset $X_0\subseteq X$ such that $X_0/\fA$ has a standard
  embedding.

\end{corollary}

Now we study another important class of $G$-varieties that generalize
rank-$0$ varieties:

\begin{definition}

  A $k$-dense $G$-variety $X$ is called \emph{$k$-horospherical} if
  $\cZ_k(X)=\cN_k(X)$.

\end{definition}

One reason for the importance of $k$-horospherical varieties is the
fact that a stratum $X(\cC)$ of an embedding $X(\cF)$ is
$k$-horospherical if and only if $\cC$ contains an interior point of
$\cZ_k(X)$. This follows immediately from \cref{prop:boundary}
\ref{it:boundary4}.

\cref{thm:auto} implies that $k$-horospherical varieties can also be
characterized by the equality $\fA_k(X)=A_k(X)$, which means that such
a variety carries an action of $A_k(X)$.

\begin{theorem}\label{thm:horo}

  For a $k$-dense $G$-variety $X$, the following are equivalent:

  \begin{enumerate}[noitemsep]

  \item\label{it:horo1} $X$ is $k$-horospherical.

  \item\label{it:horo2} An open dense $G$-stable subset of $X$ is
    induced from an elementary action.

  \item\label{it:horo3} $X=G\cdot X^N$.

  \end{enumerate}

\end{theorem}

\begin{proof}
  \ref{it:horo1}$\Rightarrow$\ref{it:horo2} Let $X$, $X_0$, and $X'$
  be as in \cref{lem:auto} with $\fA=A_k(X)$. Since $\rk_kX'=0$,
  \cref{thm:rank0} implies that there is an open subset
  $X'_1\subseteq X'$ which is induced from an anisotropic action:
  $X'_1=G\times^{Q^-}Y'$. Let $Y\subseteq X_0$ be the preimage of
  $Y'$. Because $Y\to Y'$ is an $A_k(X)$-bundle, the action on $Y$ is
  elementary, and $X_1=G\times^{Q^-}Y$ is open in $X_0$.

  The proofs of the implications
  \ref{it:horo2}$\Rightarrow$\ref{it:horo3}$\Rightarrow$\ref{it:horo1}
  are the same as those of the analogous statements in
  \cref{thm:rank0}.
\end{proof}

For homogeneous varieties, this translates into:

\begin{corollary}\label{cor:horo-spherical}

  Let $H\subseteq G$ be a $k$-subgroup and $X=G/H$. Then the following
  are equivalent:

  \begin{enumerate}[noitemsep]

  \item\label{it:homhoro1} $X$ is $k$-horospherical.

  \item\label{it:homhoro2} There is a parabolic $Q\subseteq G$ with
    $Q\el\subseteq H\subseteq Q$.

  \item\label{it:homhoro3} $H$ contains a $G(k)$-conjugate of $N$.

  \item\label{it:homhoro4} $X$ is $k$-spherical and $H$ is normalized
    by a split torus $A_0$ with $\dim A_0=\rk_kX$ and
    $|A_0\cap H|<\infty$.

  \end{enumerate}

\end{corollary}

\begin{proof}
  The proof of the equivalence of the first three parts is analogous
  to that of \cref{cor:homrk0}. Now assume that they hold. Then $G/H$
  is $k$-spherical since $G/N$ is. Moreover, it is well known that the
  split torus $A_k(X)\subseteq\Aut^G(X)=N_G(H)/H$ can be lifted to an
  isogenous torus $A_0\subseteq N_G(H)$. Hence \ref{it:homhoro4}
  holds. Conversely, when $X$ is $k$-spherical, the action of $A_0$ on
  $X$ is automatically $k$-central and locally effective. Thus,
  $\fA_k(X)=A_k(X)$ and $X$ is $k$-horospherical.
\end{proof}

From this, we derive a characterization of $k$-spherical varieties:

\begin{proposition}

  For a homogeneous $k$-dense $G$-variety $X$, the following are
  equivalent:

  \begin{enumerate}[noitemsep]

  \item $X$ is $k$-spherical.

  \item The number of $k$-dense $G$-orbits in any normal equivariant
    embedding $X\into\Xq$ is finite.

  \end{enumerate}

\end{proposition}

\begin{proof}
  One implication follows from \cref{cor:sphericalfinite}. For the
  converse, choose a valuation $v$ in the interior of $\cZ_k(X)$ and
  let $X\into\Xq=X\cup D$ be the associated embedding (with $D=X$ if
  $v=0$). It follows from \cref{prop:boundary} that $D$ is
  $k$-horospherical with $c_k(D)=c_k(X)$. Moreover, by assumption $D$
  contains only finitely many $k$-dense orbits. Since $D$ is $k$-dense
  itself, one of these orbits must be open. Then
  \cref{cor:horo-spherical} implies that $D$ is $k$-spherical. Hence
  $c_k(X)=c_k(D)=0$.
\end{proof}

\section{The Weyl group}\label{sec:WeylGroup}

Our next goal is to study the valuation cone $\cZ_k(X)$. Our strategy
is to compare it with $\cZ_K(X)$, which is known to be the Weyl
chamber of a reflection group (\cites{BrionSymetrique,KnopIB}). First,
however, we compare the canonical parabolic subgroups $Q_k(X)$ and
$Q_K(X)$.

\begin{proposition}\label{prop:QkQK}

  Let $X$ be a $k$-dense $G$-variety. Then the conjugacy class of
  $Q_K(X)$ is defined over $k$, and $Q_k(X)=Q_K(X)P$.

\end{proposition}

\begin{proof}

  The first part is clear, since $X$ is defined over $k$. For the
  second, we apply \cref{cor:LSTgeneric} in two steps. First, we
  obtain a Levi subgroup $L_k\subseteq Q_k(X)$ and an $L_k$-stable
  slice $X\el\subseteq X$ such that $Q_k(X)\times^{L_k}X\el\to X$ is
  an open embedding. Considering $X\el$ as a $K$-variety, in a second
  step we obtain a parabolic $K$-subgroup $\Qq\subseteq L_k$ with Levi
  subgroup $L_K$ and an $L_K$-stable slice $\Xq\el\subseteq X\el$ such
  that $\Qq\times^{L_K}\Xq\el\to X\el$ is an open embedding. Combining
  these, we obtain an open embedding
  \[
    \QS\times^{L_K}\Xq\el\to X\text{ with }\QS:=R_uQ_k\;\Qq.
  \]
  Thus, by uniqueness (\cref{prop:uniqueQ}), we have
  $Q_K(X)=\QS$. Next, let $\Pq$ be the image of $P$ in $L_k$. Then
  $\Pq$ contains all the anisotropic simple factors of $L_k$ (see
  \eqref{eq:H=PHel}). On the other hand, all isotropic factors act
  trivially on $X\el$ and are therefore contained in $\Qq$. Together,
  this implies $L_k=\Qq\,\Pq$ and therefore
  \[
    Q_k(X)=R_uQ_k(X)\;\Qq\;\Pq=Q_K(X)P.\qedhere\hfill\qed
  \]
\end{proof}

Observe that $Q_k(X)$ equals the product $Q_K(X)P$ rather than merely
the subgroup generated by $Q_K(X)$ and $P$. We translate
\cref{prop:QkQK} into a combinatorial statement. Recall the set
$S^0\subseteq S$ of compact roots. By definition, these are the simple
roots of $M$. On the other hand, the parabolic $Q_k(X)$ is defined
over $k$ and therefore corresponds to a subset $S_k\p(X)\subseteq S_k$
of $k$-parabolic simple roots. Replacing $k$ with $K$, we also obtain a
set $S\p(X)\subseteq S$ of $K$-parabolic simple roots corresponding to
$Q_K(X)$. These sets are related as follows:

\begin{corollary}
With the notation above, we have:
  \begin{enumerate}[noitemsep]

  \item\label{it:Sp1} $S\p(X)$ is $\cG^*$-stable {\rm(see
      \eqref{eq:staraction} for the notation).}

  \item\label{it:Sp2} $S_k\p(X)=\res_A'S\p(X)$ {\rm(see
      \eqref{eq:resstrich} for the notation).}

  \item\label{it:Sp3} Let $C$ be a connected component of
    $S\p(X)\cup S^0$ in the Dynkin diagram of $S$. Then
    $C\subseteq S\p(X)$ or $C\subseteq S^0$.
  \end{enumerate}

\end{corollary}

\begin{proof}

  Assertion \ref{it:Sp1} simply means that the conjugacy class of
  $Q_K(X)$ is defined over $k$. Assertion~\ref{it:Sp2} is then
  equivalent to $Q_k(X)$ being generated by $Q_K(X)$ and $P$. Finally,
  assertion~\ref{it:Sp3} holds because $Q_k(X)$ is indeed the product of
  $Q_K(X)$ and $P$.
\end{proof}

As mentioned above, it is known that $\cZ_K(X)$ is the fundamental
domain for a finite reflection group $W_K(X)$, the so-called
\emph{little Weyl group of $X$} (see \Cite{KnopAB}{Thm.\
  7.4}). Following the methods of \cite{KnopAB}, we show an analogous
fact for $\cZ_k(X)$. For this, we need the following geometric result:

\begin{theorem}\label{prop:flat}

  Let $X$ be a normal $k$-dense $G$-variety. Then there exists a
  $G$-stable, $k$-dense open subset $X^0\subseteq X$
  such that for every $x\in X^0(K)$ there exists a Borel subgroup $B$
  of $G$ such that the $K$-parabolic $Q_K(X)$ determined by $B$ has
  the following properties:

  \begin{enumerate}

  \item\label{it:flat1} There is a Levi $K$-subgroup
    $L_K\subseteq Q_K(X)$ such that $A_K(X)=L_K/(L_K)_x$.

  \item\label{it:flat3} The action of $W_K(X)$ on $\cN_K(X)$ induces
    an action on $A_K(X)$. This action extends to the closure of
    $L_Kx\cong A_K(X)$ in $X$.

  \item\label{it:flat4} Let $a:\G_m\to A_K(X)$ be a one-parameter
    $K$-subgroup, considered as an element of $\cN_K(X)$. Let
    $w\in W_K(X)$ with $wa\in\cZ_K(X)$, and let $v$ be the $K$-central
    valuation corresponding to $wa$. Then
    $x_0:=\lim_{t\to\infty}a(t)x$ exists in $X$ if and only if $v$ has
    a center $Y$ in $X$, in which case $x_0\in Y$. Moreover, if $X$ is
    toroidal as a $K$-variety with fan $\cF$, then the closure of $A_K(X)x$
    in $X$ is toroidal with fan $W_K(X)\cdot\cF$.

  \item\label{it:flat2}
    Assume additionally $x\in X^0(k)$. Then $B$ and $L_K$ can be
    chosen such that $B$ lies in a minimal $k$-parabolic $P$ and in a
    canonical $k$-parabolic $Q_k(X)$, and such that there exist a Levi
    complement $L_k\subseteq Q_k(X)$, a Levi complement $MA$ of $P$,
    and a maximal torus $T\subseteq L_K$ for which all groups below,
    except for $Q_K(X)$, are defined over $k$ and all indicated
    inclusions hold.
    \[\label{eq:square}
      \cxymatrix{
      A\xysubseteq\xysubseteqdown&T\xysubseteq\xysubseteqdown&MA\xysubseteqdown\\
      (L_k)\an\xysubseteq\xysubseteqdown&L_K\xysubseteq\xysubseteqdown&L_k\xysubseteqdown\\
      Q_k(X)\an\xysubseteq&Q_K(X)\xysubseteq&Q_k(X)}
  \]
    In particular, the equality
    $A_K(X)=L_K/(L_K)_x$ furnishes a $k$-structure on $A_K(X)$ such
    that the inclusion $A_k(X)\subseteq A_K(X)$ from \cref{cor:AkAK}
    is defined over $k$.
  \end{enumerate}

\end{theorem}

\begin{proof}

  First, by an affine cone construction as in the proof of
  \cref{kLST}, it is easy to see that we may assume $X$ is
  quasi-affine and, therefore, non-degenerate in the sense of
  \Cite{KnopAB}{\S3}. In fact, since the assertions can be verified on
  a finite covering of $X$ by $G$-stable quasi-projective open subsets
  (intersecting the resulting sets $X^0$ and $T^r_X$ (see below)), we
  may assume $X$ to be quasi-projective. Let $\XS\to X$ be a punctured
  affine cone. It is a principal $\G_m$-bundle which is locally trivial by
  Hilberts Satz 90. Thus $\XS(k)\to X(k)$ is surjective and all
  $\G_m$-stable open subsets of $\XS$ descend to $X$.

  We recall some facts from \cite{KnopAB}. Let $Q_K(X)$ be the
  canonical parabolic of $X$ considered as a $K$-variety. Let
  $U_K:=R_uQ_K(X)$ be its unipotent radical. Let $\pi:T_X^*\to X$ be
  the cotangent bundle, and let $m:T_X^*\to\fg^*$ be its moment
  map. For $\alpha\in T_X^*$, let $x_\alpha:=\pi(\alpha)\in X$ and
  $L_\alpha:=C_G(m(\alpha))$. In \cite{KnopAB}, it was shown that
  \[
    m^{-1}(\fu_K^\perp)=\{\alpha\in T^*_X\mid
    \alpha(\fu_Kx_\alpha)=0\}
  \]
  contains an open $Q_K(X)$-stable subset $C$ such that $\pi:C\to X$ is
  dominant and, for all $\alpha\in C$, the pair $L_K:=L_\alpha$,
  $x:=x_\alpha$ satisfies properties \ref{it:flat1}--\ref{it:flat4}.
  More precisely, \ref{it:flat1} holds by \Cite{KnopAB}{Thm.\ 2.3 and
    \S3}.
  Assertion \ref{it:flat3} follows from \Cite{KnopAB}{Cor.\ 6.3} and
  \ref{it:flat4} from \Cite{KnopAB}{Cor.\ 7.6, Cor.\ 8.3}.

  Fix, for the moment, a canonical $k$-parabolic $Q_k:=Q_k(X)$ with
  Levi subgroup $L_k$. Since $X$ is quasi-affine there is a
  $Q_k$-semiinvariant $f$ with character $\chi$ such that $L_k$ is
  precisely the centralizer of $\chi\in\fl^*_k\subseteq\fg^*$. We may
  also assume that $f$ vanishes on the singular part of $X$. Then
  define $X_f:=\{x\in X\mid f(x)\ne0\}\subseteq X$ and
\[
  V:=m^{-1}(\fl^*_k)\cap T^*_{X_f}=
  \{\alpha\in T^*_{X_f}\mid
  \alpha(\fu_kx_\alpha)=\alpha(\fu_k^-x_\alpha)=0\}.
\]
Consider the open subset of $\fl^*_k$
\[
  \fl^r_k:=\{a\in\fl^*_k\mid C_G(a)\subseteq L_k\}.
\]
Since $\chi\in\fl_k^r$ the argument of the paragraph containing
\eqref{eq:QXX} shows that
$m^{-1}(\chi)\cap T^*_{X_f}\neq\leer$. Thus $V^r:=m^{-1}(\fl^r_k)\cap T^*_{X_f}$
is a non-empty open subset of $V$.
Now observe that $U_k\times \fl^r_k\to\fu_k^\perp\cong\Fq_k$ is an
open embedding. Hence
\[\label{eq:UkVr}
  U_k\times V^r\to C_k:=m^{-1}(\fu_k^\perp)\cap T^*_{X_f}=\{\alpha\in
  T^*_{X_f}\mid\alpha(\fu_kx_\alpha)=0\}
\]
is one as well. Since the action of $U_k$ on $X_f$ is free (by the
Local Structure \cref{kLST}), we see that $C_k$ is a vector bundle
over $X_f$, hence irreducible. It follows \emph{that $V^r$ is
  irreducible and that $U_kV^r$ is open and dense in $C_k$}.

Now recall the non-empty open subset $C\subseteq m^{-1}(\fu_K^\perp)$
from the beginning of the proof. It is irreducible (for the same
reason as $C_k$) and the projection $C\to X$ is dominant. After
temporarily replacing $Q_K(X)$ and $C$ by a $G(K)$-conjugate we may
assume $Q_K(X)\subseteq Q_k(X)$, hence
$\fu_K^\perp\subseteq\fu_k^\perp$. Then $C_f:=C\cap T^*_{X_f}$ is
open and dense in $C$ and is also contained in $C_k$. It was shown in
\Cite{KnopAB}{Thm.\ 3.2} that $G\cdot C$ is dense in $T^*_X$. Thus
$\overline{G\cdot V^r}=\overline{G\cdot U_kV^r}=\overline{G\cdot
  C_k}\supseteq \overline{G\cdot C_f}=\overline{G\cdot C}=T^*_X$ shows
that \emph{$G\cdot V^r$ is dense in $T^*_X$}.

Note that this implies that the intersection of $V^r$ with any non-empty
$G$-stable open subset of $T^*_X$ is a non-empty open subset of
$V^r$. We apply this to any $G$-stable dense open $k$-subset
$T^r_X\subseteq T^*_X$ such that for all $\alpha\in T^r_X$ there is a
conjugate $Q_K$ of the canonical parabolic $Q_K(X)$ such that parts
\ref{it:flat1}--\ref{it:flat4} hold for $L_K=G_{m(\alpha)}$ and
$Q_K$. We obtain this way a non-empty dense open subset
$V^1:=V^r\cap T^r_X$ of $V^r$.

From $X_f\cong U_k\times X\el$ (Local Structure Theorem) we get
that $T^*_{X\el}$ is the quotient of $C_k$ by $U_k$. Combined with
\eqref{eq:UkVr}, one gets the commutative diagram
\[
  \cxymatrix{V^r\ar[r]^\delta\ar[d]_\pi&T^*_{X\el}\ar[d]_{\pi\el}\\
    X_f\ar[r]_q&X\el}
\]
where $\delta$ is an open embedding.

Let $T^r_{X\el}\subseteq T^*_{X\el}$ be the analogue of
$T^r_X$ for the $L_k$-variety $X\el$. Then define the dense open
subset $V^0:=V^1\cap \delta^{-1}(T^r_{X\el})$ of $V^r$. Furthermore, we
define $X\el^0:=\pi\el(\delta(V^0))$, which is open, since
$\pi\el$ is an open morphism, and $X^0_f:=q^{-1}(X\el^0)=U_k\pi(V^0)$.

We claim that $X_f^0(k)=U_k(k)\pi(V^0(k))$. Indeed, let
$x\in X_f^0(k)$ and $\xq:=q(x)\in X\el^0(k)$. By construction, the
fiber $\delta(V^0)\cap\pi\el^{-1}(\xq)$ is a non-empty open subset of
an affine space, hence has a $k$-point $\oalpha$. Let
$\alpha\in V^0(k)$ with $\delta(\alpha)=\oalpha$. Since $x$ and
$x_\alpha$ lie in the same (free!) $U_k$-orbit there exists
$u\in U_k(k)$ with $x=ux_\alpha$, as claimed.

We proceed to show that $X^0:=GX^0_f$ satisfies the assertions of
\ref{it:flat2}. First, Zariski-density of $G(k)$ in $G$ and the
openness of $X^0_f$ imply that $X^0=G(k)X_f^0$ and therefore
$X^0(k)=G(k)X_f^0(k)$. Combined with the claim above, we get
$X^0(k)=G(k)U_k(k)\pi(V^0(k))=G(k)\pi(V^0(k))$. This shows that it
suffices to prove \ref{it:flat2} for the case that $x=x_\alpha$
with $\alpha\in V^0(k)$.

Let, as before, $\oalpha=\delta(\alpha)\in T^*_{X\el}$ and
$\xq=q(x)\in X\el^0$. Since, by construction, $\oalpha\in T^r_{X\el}$, 
we can apply \ref{it:flat1} to the $L_k$-variety $X\el^0$ and obtain a
$K$-parabolic $\Qq_K=Q_K(X\el)$ such that
$\Lq_K:=(L_k)_{m\el(\oalpha)}$ is a Levi complement. Put
$Q_K:=\Qq_KU_k\subseteq Q_k$.

Since $Q_K(X)$ can be read off from $\Xi_K(X)$ and clearly
$\Xi_K(X\el)=\Xi_K(X)$ we see that $Q_K$ is a canonical $K$-parabolic
of $X$. From $m(\alpha)\in\fl^r_k$ we see that
$G_{m(\alpha)}=(L_k)_{m(\alpha)}$. Moreover, $m(\alpha)=m\el(\oalpha)$
because $q$ is $L_k$-equivariant. Hence $(L_k)_{m(\alpha)}=\Lq_K$,
which is a Levi complement of $\Qq_K$ and therefore of $Q_K$. Put
$L_K:=G_{m(\alpha)}=\Lq_K$. Since $\alpha\in T^r_X$, properties
\ref{it:flat1}--\ref{it:flat4} hold for $L_K$ and $x$. Moreover, since
$\alpha\in V^0(k)$, the group $L_K=G_{m(\alpha)}$ is defined over
$k$. This shows the existence of the lower right corner of
\eqref{eq:square}.

  For the lower left square it suffices to show that $L_K$ contains
  $(L_k)\an$. To this end, we use that the action of $L_k$ on $X\el$
  is elementary, which means that $(L_k)\el$ acts trivially on
  $X\el$. Hence $(L_k)\el\subseteq L_K$. Since, on the other hand $Z(L_k)$ is
  contained in any Levi subgroup of $\Qq_K$ we obtain the assertion
  $(L_k)\an\subseteq Z(L_k)\,(L_k)\el\subseteq L_K$.

  For the upper left square let $A$ be a maximal split
  subtorus of $L_k$, hence of $G$. Then $A\subseteq L_k$ implies
  $A\subseteq (L_k)\an$ and therefore $A\subseteq L_K$. Thus, one can
  choose a maximal $k$-torus $T\subseteq L_K$ which contains $A$.

  For the upper right square choose any minimal $k$-parabolic
  $\Pq=MA\Nq$ of $L_k$ containing $A$. Put $P:=\Pq R_uQ_k(X)$.  Then
  $T\subseteq C_{L_k}(A)=MA$.

  Next, we show that there is a Borel subgroup
  $B\subseteq G$ which is contained in all three parabolics $Q_k(X)$,
  $Q_K(X)$, and $P$. For this, it suffices to show that $Q_K(X)\cap P$
  is parabolic in $Q_k(X)$. Consider the unique decomposition
  $L_k=L^i\;L^a$ where $L^i$ and $L^a$ are the products of the
  isotropic and anisotropic (possibly abelian) factors of $L_k$. Then
  $Q_k=L^iL^aU_k$ with $(Q_k)\an=L^iU_k$. Therefore, $P=P^iL^aU_k$
  where $P^i$ is a minimal parabolic of $L^i$. Likewise,
  $L^i=(L_k)\an\subseteq L_K\subseteq Q_K$ (see above) implies
  $Q_K=L^iQ^aU_k$ where $Q^a\subseteq L^a$ is a $K$-parabolic. Thus
  $Q_K\cap P=P^iQ^aU_k$ is a $K$-parabolic in $Q_k$.

  Finally, since $L_K$ and $x$ are defined over $k$ also $A_K(X)$
  obtains a $k$-structure. Moreover, the inclusion
  $A_k(X)\subseteq A_K(X)$ is defined over $k$ since it is defined in
  terms of a maximal torus $T$ which is contained in $MA$.
\end{proof}

\begin{corollary}\label{cor:Gaction}
  Let $X$, $L_K$, and $T$ be as in
  \cref{prop:flat}\ref{it:flat2}. Then $\Xi_K(X)\subseteq\Xi(T)$ is
  $\cG$-stable.
\end{corollary}

\begin{remarks}\label{rmk:flat}

  \emph{i)} The $k$-structure of $A_K(X)$ depends on the choice of
  $L_K$. It is not the same as the one from
  \cref{lemma:definedoverk} which is induced by the $*$-action on $T$.

  \emph{ii)} In \cite{KnopAB}, an orbit of the form
  $G_{m(\alpha)}x_\alpha\cong A_K(X)$ satisfying the properties
  \ref{it:flat1}--\ref{it:flat4} is called a \emph{flat of $X$} (at
  least if $X$ is quasi-affine; otherwise a flat is a projection of a
  flat in $\XS$ to $X$, where $\XS$ is a punctured affine cone over
  $X$).  Observe that a flat equals $Tx$ where $x=x_\alpha$ and
  $T\subseteq G_{m(\alpha)}$ is any maximal torus.

  \emph{iii)} Observe that flats are not required to be defined over
  $k$. Therefore, we are going to call them $K$-flats to emphasize
  this fact. If on the other hand $\alpha$ is defined over $k$ such
  that $x=\pi(\alpha)$ and $L_K=G_{m(\alpha)}$ satisfy part
  \ref{it:flat2} of the \lcnamecref{prop:flat} then we call $L_Kx$ a
  \emph{flat over $k$}. The $k$-split part of a flat over $k$ will be
  called a \emph{$k$-flat}. Every $k$-flat is isomorphic to $A_k(X)$
  and every $k$-point of the open subset $X^0$ is contained in a
  $k$-flat.  Moreover, in \cref{cor:flatclosure} below, we show that
  the closure of a $k$-flat inside a toroidal embedding $X(\cF)$ is
  isomorphic to the torus embedding $A_k(W_k(X)\cF)$.

  \emph{iv)} Observe that in general the set of flats depends on the
  global structure of $X$ (and not just on an open subset). For
  example, assume that $Y$ as in \ref{it:flat4} of the
  \lcnamecref{prop:flat} does not have an open orbit. Then one could
  make any specific flat ``bad'' by removing the $G$-orbit of the
  limit point $x_0\in Y$ from $X$ or by blowing it up. This problem
  does not occur for spherical varieties.

\end{remarks}

We are now in a position to show that $\cZ_k(X)$ is also the Weyl
chamber of a reflection group. For this, we define
\[
  N(\cN_k):=\{w\in W_K(X)\mid w\cN_k(X)=\cN_k(X)\}
\]
to be the normalizer of $\cN_k(X)$ in $W_K(X)$,
\[
  C(\cN_k):=\{w\in W_K(X)\mid w|_{\cN_k(X)}=\|id|_{\cN_k(X)}\},
\]
its centralizer, and
\[
  W_k(X):=N(\cN_k)/C(\cN_k),
\]
the restricted Weyl group.

\begin{theorem}\label{thm:WKWk}

  Let $X$ be a $k$-dense $G$-variety. Then $\cZ_k(X)$ is a fundamental
  domain for the action of $W_k(X)$ on $\cN_k(X)$.

\end{theorem}

\begin{proof}

  Since $\cZ_k(X)\subseteq\cZ_K(X)$, no two distinct elements of
  $\cZ_k(X)$ are $W_k(X)$-conjugate, as the same is true for
  $\cZ_K(X)$ and $W_K(X)$ (\Cite{KnopAB}{Thm.\ 7.4}). It remains to
    show that every element of $\cN_k(X)$ is $W_k(X)$-conjugate to an
    element of $\cZ_k(X)$.

  We claim that it suffices to show that every $v\in\cN_k(X)$ is
  $W_K(X)$-conjugate to an element of $\cZ_k(X)$. To see this, define
  \[
    \cN_k^0:=\cN_k(X)\setminus \bigcup_{w\in W_K(X)\setminus
      N(\cN_k)}w\cN_k(X).
  \]
  This is an open and dense subset of $\cN_k(X)$ with the property
  that if $w\in W_K(X)$, $v\in \cN_k^0$, and $wv\in \cN_k(X)$, then
  $w\in N(\cN_k)$. Thus, if for every $v\in \cN_k^0$ there exists
  $w\in W_K(X)$ with $wv\in \cZ_k(X)$, then automatically
  $w\in N(\cN_k)$. The claim follows by continuity.

  So let $v\in \cN_k(X)$, let $v_0\in\cZ_K(X)$ be the unique element
  in the $W_K(X)$-orbit of $v$, and let $(\Xq,D)$ be a model of
  $v_0$. It can be chosen to be defined over $k$ (see the fourth
  paragraph of \S\ref{sec:InvVal}). We claim that $D\cap\Xq(k)$ is
  dense in $D$. Since $D$ is a divisor, this implies in particular
  that $D$ is defined over $k$, and it would prove that
  $v_0\in\cZ_k(X)$ by \cref{thm:valdense}.

  Suppose $D\cap\Xq(k)$ were not dense in $D$. Then, after removing
  from $\Xq$ the Zariski closure of $D\cap\Xq(k)$, we may assume
  $D\cap\Xq(k)=\leer$. On the other
  hand, let $a:\G_m\to A_k(X)$ be a cocharacter in the ray spanned by
  $v$. Then, according to \cref{prop:flat}\ref{it:flat4}, there exists
  an $x\in X(k)$ such that the limit $x_0=\lim_{t\to\infty}a(t)x$ exists in
  $\Xq$ and lies in $D$. This yields the contradiction
  $x_0\in D\cap\Xq(k)$, proving the claim.
\end{proof}

Recall that a \emph{facet} of a convex cone is a face of codimension
one.

\begin{corollary}

  The group $W_k(X)$ is generated by the reflections across the facets
  of $\cZ_k(X)$.

\end{corollary}

\begin{proof}

  Indeed, only reflection groups have closed fundamental domains.
\end{proof}

Fundamental domains of finite reflection groups are very special:

\begin{corollary}\label{cor:cosimplicial}

  The valuation cone $\cZ_k(X)$ is cosimplicial, i.e., it is defined
  by a set of linearly independent linear inequalities.

\end{corollary}

Next, we determine the closure of a $k$-flat in a toroidal embedding.

\begin{proposition}\label{cor:flatclosure}
  Let $X$ be a $k$-dense $G$-variety and let $X(\cF)$ be the toroidal
  embedding corresponding to a fan $\cF$ supported in $\cZ_k(X)$. Let
  $A_k(X)x$ be a $k$-flat of $X(\cF)$ with $x\in X$. Then the closure
  of $A_k(X)x$ in $X(\cF)$ is isomorphic to $A_k(W_k(X)\cF)$.
\end{proposition}

\begin{proof}
  Consider the situation of \cref{prop:flat}. Put $W_k:=W_k(X)$ and
  $W_K:=W_K(X)$. We claim that $X(\cF)$ is toroidal
  with the same fan $\cF$ when $X$ is considered as a $K$-variety.
 This makes sense, since $\cZ_k(X)\subseteq\cZ_K(X)$ by
 \cref{prop:kKcentral}.
 
  Indeed, as in the proof of \cref{prop:flat}, we may choose an
  LST-slice $Y$ of $X$ over $K$ inside $X\el$ such that
  $\Uq_K\times Y\to X\el$ is an open embedding, where
  $\Uq_K:=R_u\Qq_K$ and $U_K:=U_k\Uq_K=R_uQ_K$. Since $(L_k)\el$ acts
  trivially on $X\el$ and $L_k=MA(L_k)\el$ by \eqref{eq:H=PHel}, the
  action of $A$ on $X\el$ commutes with that of $L_k$. Moreover, $A$
  acts on $Y$ through $A_k(X)\subseteq A_K(X)$
  (\cref{cor:AkAK}). Hence
  $U_K\times(Y\times^{A_K}A_K(\cF))=U_k\times\Uq_K\times(Y\times^{A_k}A_k(\cF))$
  is an open subset of $U_k\times X\el(\cF)$ meeting every
  stratum. Thus the toroidal embedding of $X$ over $K$ with fan $\cF$
  is an open subset of $X(\cF)$ meeting every stratum, and it
  coincides with $X(\cF)$ after shrinking $X\an$ (cf.\
  Remark~\ref{rem:XcF}\emph{ii)}).

  Now \ref{it:flat4} says that the
  closure of $A_K\cong Tx$ in $X(\cF)$ is isomorphic to
  $\Zq_K:=A_K(W_K\cF)$ and that $A_K(\cF)$ is an open subset of
  $\Zq_K$. Hence also $A_K(W_k\cF)=W_kA_K(\cF)$ is open in
  $\Zq_K$. Since $W_k\cF$ is supported in $\cN_k(X)$, the closure of
  $A_kx$ in $A_K(W_k\cF)$ is $\overline Z_k:=A_k(W_k\cF)$; hence
  $\Zq_k$ is locally closed in $\Zq_K$. If we show that $\Zq_k$ is
  closed in $\Zq_K$, then it is the closure of $A_kx$ in $X(\cF)$ and
  we are done.

  To this end, it suffices to show that every $1$-parameter subgroup
  $\lambda\in\cN_k(X)$ which converges in $\Zq_K$ converges also in
  $\Zq_k$ (follows, e.g., from Kempf \cref{thm:Kempf} and the fact
  that $\Zq_K$ is covered by $A_K$-stable open affine subsets). By
  $W_k$-equivariance and \cref{thm:WKWk}, we may assume
  $\lambda\in\cZ_k(X)$. Now convergence in $\Zq_K$ means
  $\lambda\in w\cC$ for $w\in W_K$ and $\cC\in\cF$.
  Then $w^{-1}\lambda\in\cC\subseteq\cZ_k(X)\subseteq\cZ_K(X)$. Since
  $\cZ_K(X)$ is a fundamental domain for $W_K$, this implies
  $\lambda=w^{-1}\lambda\in\cC$, so $\lambda$ converges in
  $\overline Z_k$.
\end{proof}

Another consequence of \cref{prop:flat} is:

\begin{proposition}[cf.~\cite{KnopAB}*{Cor.\ 7.9}]\label{prop:quasiaff}

  Let $X$ be a $k$-semisimple, $k$-spherical, and quasi-affine
  $G$-variety. Then $X$ is affine and homogeneous.

\end{proposition}

\begin{proof}

  Let $X_0\subseteq X$ be the open $G$-orbit and let $X\into\Xq$ be an
  equivariant affine embedding. Suppose $X_0\ne\Xq$. Then Kempf's
  \cref{thm:Kempf} implies that $\Xq\setminus X_0$ contains a
  $k$-rational point $y$. \cref{lem:centercentral} shows that
  $Y:=\overline{Gy}$ is the center of a $k$-central valuation
  $v\in\cZ_k(X)$. Now let $A_kx\subseteq A_Kx\subseteq X_0$ be a
  $k$-flat as in Remark \ref{rmk:flat}\emph{iii)} and let $\Aq$ be the
  closure of $A_kx$ in $\Xq$. Then $\Aq$ meets $Y$ by
  property~\ref{it:flat4} of \cref{prop:flat}. Therefore, the
  normalization of $\Aq$ is a non-trivial affine torus embedding and
  hence given by a non-zero strictly convex cone
  $\cC\subseteq\cN_k(X)$. Moreover, $\cC$ is $W_k(X)$-stable by
  \cref{prop:flat}\,\ref{it:flat3}. Such a cone can exist only when
  the fixed point set $\cN_k(X)^{W_k(X)}$ is non-zero. However, that
  set equals $\cZ_k^0$, which is $0$ by assumption. This contradiction
  shows that $X=X_0=\Xq$ is homogeneous and affine.
\end{proof}

Here is an application of \cref{prop:quasiaff}. Recall that a
homogeneous variety is called \emph{unimodular} if it admits an
invariant volume form. If $X=G/H$ (with $G$ reductive!), this is
equivalent to $H$ acting trivially on $\wedge^{\rm top}\fh$, i.e., to
$H$ being unimodular. We use the two expressions interchangeably.

\begin{corollary}

  Let $X=G/H$ be a $k$-spherical $G$-variety. Assume $N_G(H)/H$ is
  anisotropic (e.g., finite) and $H$ is unimodular. Then $H$ is
  reductive.

\end{corollary}

\begin{proof}

  The condition on $N_G(H)/H$ means that $X$ is $k$-semisimple
  (\cref{prop:anisoconvex}). Then $G/H^0$ is also $k$-semisimple
  (\cref{lem:GAcentral}) and $k$-spherical. Thus, we may assume
  without loss of generality that $H$ is connected.

  Now consider $\fh=\Lie H$ and let $\tau$ be a generator of the line
  $\wedge^d\fh$ (with $d=\dim\fh$). Let $N_\tau$ be its isotropy group
  in $N_G(H)$. Since $H$ is unimodular, $H\subseteq N_\tau$. By
  assumption, $N_G(H)/H$ is anisotropic; thus its subgroup $N_\tau/H$
  is also anisotropic. Therefore, it acts locally trivially on the
  fibers of $X\el\to X\an$, which implies that
  $\rk_kG/H=\rk_kX\el=\rk_kX\el/N_\tau=\rk_k G/N_\tau$. From
  \cref{lem:GAcentral}, we obtain $\cZ_k(G/H)=\cZ_k(G/N_\tau)$ and, in
  particular, that $G/N_\tau$ is also $k$-semisimple.

  We now consider $\tau$ as a point of $\wedge^d\fg$. The orbit map
  then yields an embedding
  \[
    \rho:G/N_\tau\into\wedge^d\fg:gN_\tau\mapsto g\tau
  \]
  which means that $G/N_\tau$ is quasi-affine. Since it is also
  $k$-spherical and $k$-semisimple, we find from \cref{prop:quasiaff}
  that $G/N_\tau$ is affine, which implies that $N_\tau$ is
  reductive. Its normal subgroup $H$ is then reductive as well.
\end{proof}

\section{The root system}\label{sec:RootSystem}

In this section, $X$ is a fixed $k$-dense $G$-variety. Accordingly,
when there is no danger of confusion, we drop the reference to $X$ in
most of the notation:
\[
  S\p,\ S_k\p,\ \Xi_k,\ A_k,\ \cN_k,\ \cZ_k,\ W_k
\]
where $k$ can also be replaced by $K$.

After assigning a $k$-Weyl group to $X$, we now construct an integral
root system for it. The simple roots should vanish on the facets of
the Weyl chamber $\cZ_k$ and are therefore unique up to a positive
scalar. Thus, we face an issue of normalization. The most obvious
normalization is as follows:

\begin{definition}

  Let $X$ be a $k$-dense $G$-variety.

  {\it i)} A \emph{primitive $k$-spherical root of $X$} is a primitive
  element $\sigma\in\Xi_k$ such that $\cZ_k\cap\{\sigma=0\}$ is a
  facet of $\cZ_k$. In other words, $\sigma$ is the unique shortest
  element of $\Xi_k$ that is an outward normal vector to a facet of
  $\cZ_k$.

  {\it ii)} The set of primitive $k$-spherical roots is denoted by
  $\Sigma_k\prim:=\Sigma_k\prim(X)$. It is in bijection with the
  facets of $\cZ_k$.

  {\it iii)} The \emph{primitive $k$-root system of $X$} is
  \[\label{eq:PhiKdef}
    \Phi_k\prim=\Phi_k\prim(X):=W_k\cdot\Sigma_k\prim\subseteq\Xi_k.
  \]

\end{definition}

\cref{cor:cosimplicial} implies that $\Sigma_k\prim$ is linearly
independent and that the valuation cone has the presentation
\[\label{eq:cZsigma}
  \cZ_k=\{a\in \cN_k\mid\sigma(a)\le0\text{ for all
  }\sigma\in\Sigma_k\prim\}.
\]
Recall the set $S_k\subset\Xi(A)$ of restricted simple roots of
$G$. Then \cref{prop:chamber} has the following reformulation:

\begin{lemma}\label{lem:non-negative}

  Every $\sigma\in\Sigma_k\prim$ is a linear combination
  $\sigma=\sum_{\alpha\in S_k}c_\alpha\alpha$ with
  $c_\alpha\in\QQ_{\ge0}$.

\end{lemma}

The set of $\alpha\in S_k$ with $c_\alpha>0$ is called the
\emph{support of $\sigma$} (denoted by $\supp\sigma$).

A \emph{weight lattice} for a root system $\Phi$ with Weyl group $W$
is a lattice $\Xi$ containing $\Phi$ such that
$(1-s_\sigma)\Xi\subseteq\ZZ\sigma$ for all $\sigma\in\Phi$. Here,
$s_\sigma$ denotes the reflection with respect to $\sigma$. If $\Phi$ is a
reduced root system, then this is equivalent to $W$ acting trivially
on $\Xi/\<\Phi\>_\ZZ$. This follows easily from
$\QQ\sigma\cap\<\Phi\>_\ZZ=\ZZ\sigma$. Yet another criterion is the
existence of a (unique) coroot $\sigma^\vee\in\Hom(\Xi,\ZZ)$ such that
\[
  s_\sigma(\chi)=\chi-\<\chi|\sigma^\vee\>\sigma\text{ for all
  }\chi\in\Xi.
\]
We record the following:

\begin{proposition}

  $\Phi_k\prim$ is a reduced root system with Weyl group $W_k$. The
  set $\Sigma_k\prim$ is a system of simple roots for $\Phi_k\prim$.
  The valuation cone $\cZ_k$ is the anti-dominant Weyl chamber with
  respect to $\Sigma_k\prim$. The lattice $\Xi_k$ is a weight lattice
  for $\Phi_k\prim$.

\end{proposition}

\begin{proof}

  Everything follows from the definitions and the fact that $\cZ_k$ is
  a fundamental domain for $W_k$ once it is established that $\Xi_k$
  is $W_k$-stable. Over $K$, this is \Cite{KnopAB}{Thm.\ 4.2}. For
  $k$, it now follows from \eqref{eq:resXiK}.
\end{proof}

Of course, all constructions above are valid with $k$ replaced by
$K$. In this case, several other normalizations of spherical roots are
in use (see, for example, the account in \cite{BvS}), but the
primitive one is the most common. Therefore, over $K$, we drop the
superscript $\prim$ and write $\Sigma_K=\Sigma_K(X):=\Sigma_K\prim$
and $\Phi_K=\Phi_K(X):=\Phi_K\prim$. In the following, we focus on the
relationship between $\Sigma_K$ and $\Sigma_k\prim$.

For any $\sigma\in\Xi_K$, let $\osigma:=\res_A\sigma$ be the
restriction to $A$ and let
\[
  \Xi_K^0=\Xi_K^0(X)=\{\sigma\in\Xi_K(X)\mid\osigma=0\}.
\]
Elements of
\[\label{eq:Sigma0}
  \Sigma_K^0=\Sigma_K^0(X):=\Sigma_K\cap\Xi_K^0=
  \{\sigma\in\Sigma_K\mid\osigma=0\}
\]
will be called \emph{($k$-)compact spherical roots}. From
\cref{lem:non-negative} (over $K$), it follows that the compact
spherical roots can be recovered from the compact simple roots
\[
  \Sigma_K^0=\{\sigma\in\Sigma_K\mid\supp(\sigma)\subseteq S^0\}.
\]
We will see below (\cref{cor:kernel}) that $\Sigma_K^0$ is the set of
$K$-spherical roots of $X\an$. Let
\[
  W_K^0=W_K^0(X)=\<s_\sigma\mid\sigma\in\Sigma_K^0\>\subseteq W_K
\]
be the subgroup of $W_K$ generated by all reflections with respect to
compact spherical roots. Using the notation \eqref{eq:resstrich}, we
define the set of restricted spherical roots as
\[
  \Sigma_k=\Sigma_k(X):=
  \res_A'\Sigma_K=\{\osigma\mid\sigma\in\Sigma_K,\ \osigma\ne0\}.
\]
Our goal is to show that $\Sigma_k$ is another system of simple roots
for $W_k$.

\begin{remark}

  In general, $\Sigma_k$ and $\Sigma_k\prim$ are different. This
  already happens in the group case $X=H\times H/H$ with $H=U(1,1)$
  and $k=\RR$. In that case, $\Xi_K(X)\cong\ZZ^2$ with Galois action
  $(n_1,n_2)\mapsto(-n_2,-n_1)$ and positive spherical root
  $\sigma=(1,-1)$. The restriction map is
  $\Xi_K=\ZZ^2\to\Xi_k=\ZZ:(n_1,n_2)\mapsto n_1-n_2$. Thus,
  $\osigma=\res_A\sigma=2$ is not primitive.

\end{remark}

Next, we study the action of the Galois group $\cG$. Recall the
situation of \cref{prop:flat}, in particular the maximal $k$-torus $T$
containing the maximal $k$-split torus $A$ and a point $x\in X(k)$
such that the orbit $Tx$ is a flat over $k$. Then the following
objects are considered:
\[
  \cxymatrix{\Sigma_K\xysubseteq&
    \Xi_K\ar@{>>}[r]\xysubseteqdown&\Xi_k\xysubseteqdown\\
    S\xysubseteq& \Xi(T)\ar@{>>}[r]&\Xi(A)}
\]
By \cref{cor:Gaction}, the sublattice $\Xi_K$ is $\cG$-stable and the
homomorphism $\Xi_K\auf\Xi_k$ is $\cG$-invariant. More important for
our purposes is the $*$-action of $\cG$ on $\Xi(T)$ defined by
$\gamma*\chi:=w_\gamma\,{}^\gamma\chi$ (see
\eqref{eq:staraction}). Since
$w_\gamma\in W^0=\<s_\alpha\mid\alpha\in S^0\>$ (see
\S\ref{sec:notation}), the restriction $\res_A:\Xi(T)\to\Xi(A)$ is
also $\cG^*$-invariant. Moreover, the subgroup $\Xi_K\subseteq\Xi(T)$
is $\cG^*$-stable by \cref{lemma:definedoverk}.

We define a $\cG$-action on valuations by
${}^\gamma v(f):=v({}^{\gamma^{-1}}f)$. Then it is easily seen that
the inclusion $\iota_K:\cZ_K\into\cN_K$ is
$\cG^*$-equivariant. Indeed, because of the $G$-invariance of $v$, we
have
\[
  \iota_K({}^\gamma v)(\chi)= {}^\gamma v(f_\chi)=
  v({}^{\gamma^{-1}}f_\chi)= v({}^{n_{\gamma^{-1}}\gamma^{-1}}f_\chi)=
  v(f_{\gamma^{-1}*\chi})= (\gamma*\iota_K(v))(\chi).
\]
This shows that $\Sigma_K$ is a $\cG^*$-stable subset of $\Xi_K$. In
particular, $\Phi_K$ and $W_K$ are $\cG^*$-stable.

\begin{lemma}\label{lem:wgamma}

  Let $\gamma\in\cG$. Then $w_\gamma\Xi_K=\Xi_K$, and the restriction
  of $w_\gamma$ to $\Xi_K$ is an element of $W_K^0$.

\end{lemma}

\begin{proof}

  The first assertion follows from the fact that $\Xi_K$ is stable
  under both $\gamma$ and $\gamma*$.

  Next, we show that there is an element $\wq_\gamma\in W_K$ with
  $w_\gamma v=\wq_\gamma v$ for all $v\in\cN_K$. For this, let
  $v\in\cZ_K$ and let $(\Xq,D)$ be the toroidal embedding
  corresponding to the ray $\QQ_{\ge0}v$. Then $\Xq$ does not contain
  any other $G$-stable divisor $D'\ne D$ with
  $v_{D'}\in\cZ_K$. Consider the closure $\overline{Tx}$ in $\Xq$
  where $Tx$ is a flat over $k$ (see Remark
  \ref{rmk:flat}\emph{iii)}). Then $\overline{Tx}$ equals the toroidal
  embedding $A_K(\cF)$, where $\cF$ is the fan consisting of $\{0\}$
  and all rays $\QQ_{\ge0}wv$ with $w\in W_K$. We now twist $\Xq$ by
  an element $\gamma\in\cG$. Then the embedding
  $X\into{}^\gamma\Xq=X\cup{}^\gamma D$ corresponds to the valuation
  $\gamma*v\in\cZ_K$. Using the fact that $Tx$ is $\cG$-stable and is
  therefore also a flat over $k$ of ${}^\gamma X$, we obtain
  \[
    \gamma\bigcup_w\QQ_{\ge0}wv=\bigcup_w\QQ_{\ge0}w(\gamma*v).
  \]
  Since $\cG$ preserves the primitive generators of the rays, we see
  that the orbit of $v$ is mapped to the orbit of $\gamma*v$:
  \[\label{eq:orbitgamma}
    \gamma(W_Kv)=W_K(\gamma*v).
  \]
  Every $v\in\cN_K$ is $W_K$-conjugate to an element of $\cZ_K$, and
  hence \eqref{eq:orbitgamma} holds for all $v\in\cN_K$. Now, using
  the fact that $W_K$ is normalized by $\cG^*$, we obtain
  \[
    w_\gamma W_Kv=(w_\gamma\gamma)(\gamma^{-1}W_Kv)
    =\gamma*W_K(\gamma^{-1}*v)=W_Kv.
  \]
  Thus, for every $v$ there is $\wq_\gamma\in W_K$ with
  $w_\gamma v=\wq_\gamma v$. The element $\wq_\gamma$ might depend on
  $v$. It is unique for $v$ in the interior $\cZ_K^0$ of $\cZ_K$, and
  there it depends continuously on $v$. Thus, $\wq_\gamma$ is constant
  for $v\in\cZ_K^0$, and hence constant everywhere by linearity.

  Since ${}^\gamma\sigma=\wq_\gamma^{-1}\gamma*\sigma$, we see
  that $\Phi_K$ is $\cG$-stable and $\wq_\gamma$ is the unique
  element of $W_K$ such that
  \[
    \wq_\gamma\,{}^\gamma\Sigma_K=\Sigma_K.
  \]
  We claim that $\wq_\gamma\in W_K^0$. Let $\Phi^0=W_K^0\Sigma_K^0$ be
  the root subsystem generated by $\Sigma_K^0$. It follows from
  \cref{lem:non-negative} over $K$ that
  $\Phi^0=\{\sigma\in\Phi_K\mid\res_A\sigma=0\}$. Since $\res_A$ is
  $\cG$-invariant, this implies that $\Phi^0$ is $\cG$-stable. Thus,
  for every $\gamma\in\cG$ there is a unique element
  $w_\gamma^0\in W_K^0$ such that
  \[
    w_\gamma^0\ {}^\gamma\Sigma_K^0=\Sigma_K^0.
  \]
  For any non-compact $\sigma\in\Sigma_K$, let
  $\tilde\sigma:=w_\gamma^0\,{}^\gamma\sigma$. Since $\res_A$ is
  invariant under both $\cG$ and $W_K^0$, we get
  $\res_A\tilde\sigma=\res_A\sigma$. This and the inclusion
  $\cZ_k\subseteq\cZ_K$ imply that $\tilde\sigma$ has at least one
  negative value on $\cZ_K$ and is therefore a positive root of
  $\Phi_K$. Together, this shows that $w_\gamma^0\ {}^\gamma\Sigma_K$
  consists entirely of positive roots and is therefore equal to
  $\Sigma_K$. Uniqueness implies $\wq_\gamma=w_\gamma^0\in W_K^0$.
\end{proof}

The following was established during the preceding proof:

\begin{corollary}

  The root system $\Phi_K$ is a $\cG$-stable subset of\/ $\Xi_K$, and
  $W_K$ is normalized by $\cG$.

\end{corollary}

We also obtain that the ordinary action and the $*$-action of $\cG$
are the same modulo compact roots:

\begin{corollary}\label{cor:tildeW}

  Let $\gamma\in\cG$ and $\chi\in\Xi_K$. Then
  ${}^\gamma\chi-\gamma*\chi\in\<\Sigma_K^0\>_\ZZ$.
\end{corollary}

\begin{proof}
  Because $\wq_\gamma\in W_K^0$, we have
  \[ {}^\gamma\chi-\gamma*\chi=
    (1-\wq_\gamma)({}^\gamma\chi)\in\<\Sigma_K^0\>_\ZZ.\qedhere\hfill\qed
  \]
\end{proof}

As a consequence, we can compute $\cN_k$ from $\Sigma_K^0$ and the
$\cG^*$-action on $\Xi_K$:

\begin{corollary}\label{cor:compN}

  The subspace $\cN_k$ of $\cN_K$ is defined by the equations
  \[\label{eq:equations}
    \{\sigma=0\mid\sigma\in\Sigma_K^0\}\cup
    \{\chi-\gamma*\chi=0\mid\chi\in\Xi_K,\ \gamma\in\cG\}.
  \]

\end{corollary}

\begin{proof}

  Since $\cN_k=\cN_K^\cG$, the space $\cN_k$ is defined by the
  equations $\chi-{}^\gamma\chi=0$ with $\chi\in\Xi_K$,
  $\gamma\in\cG$. Moreover, all equations $\sigma=0$ for
  $\sigma\in\Sigma_K^0$ hold on $\cN_k$ by definition. Modulo those
  equations, we have
  \[
    \chi-{}^\gamma\chi=\chi-\gamma*\chi,
  \]
  which proves the assertion.
\end{proof}

Observe that the weight lattice $\Xi_k$ can also be computed, since it
is the image of $\Xi_K$ in $\cN_k^*$. The equations for $\cN_k$ can be
made slightly more precise. There is a unique $W_k$-stable
decomposition $\cN_k=\cN_k^0\oplus\cN_k^1$ with
$\cN_k^0:=\cN_k^{W_k}=\cZ_k^0$ and such that $\Sigma_k$ forms a basis
of $(\cN_k^1)^*$. Then
\[
  \cN_k^0=(\cN_K^0)^{\cG}=(\cN_K^0)^{\cG^*}.
\]
Moreover, it suffices to let $\chi$ in \eqref{eq:equations} range over
a generating set. Thus, the subspace $\cN_k^1\subseteq\cN_K^1$ is
defined by the equations $\sigma=0$ for all compact spherical roots
and $\sigma-\tau=0$ for all pairs of non-compact spherical roots that
are in the same $\cG^*$-orbit. This shows:

\begin{corollary}\label{lem:linindep}

  The set $\Sigma_k\subseteq\Xi_k$ is linearly independent. Moreover,
  $\osigma=\otau\ne0$ for $\sigma,\tau\in\Sigma_K$ if and only if
  $\sigma$ and $\tau$ are in the same $\cG^*$-orbit.

\end{corollary}

From this, we obtain:

\begin{corollary}\label{cor:SSprim}

  There is a map $\Sigma_k\prim\to\ZZ_{>0}:\sigma\mapsto n_\sigma$
  such that $\Sigma_k=\{n_\sigma\sigma\mid\sigma\in\Sigma_k\prim\}$.

\end{corollary}

\begin{proof}

  It follows from \eqref{eq:diagXi22} that $\cZ_k$ is also defined by
  the inequalities $\osigma\le0$ with $\sigma\in\Sigma_K$. Since they
  are linearly independent, they form a minimal set of
  inequalities. Consequently, every $\osigma$ is an integral multiple
  of an element of $\Sigma_k\prim$.
\end{proof}

Since elements of $\Sigma_K$ and $\Sigma_k$ correspond to facets of
$\cZ_K$ and $\cZ_k$, respectively, we obtain:

\begin{corollary}\label{lem:facet}

  Let $\cF\subseteq\cZ_K$ be a facet and
  $\overline\cF:=\cF\cap\cZ_k$. Then either $\overline\cF=\cZ_k$ or
  $\overline\cF$ is a facet of $\cZ_k$.
\end{corollary}

A formula for the fundamental coweights due to Borel and Tits
\Cite{BorelTits}{Thm.\ 6.13} also generalizes almost verbatim. Let
$\omega_\sigma^\vee\in\cN_k^1$ be the dual basis of $\Sigma_k$ such
that
\[
  \cZ_k=\cN_k^0+
  \sum_{\sigma\in\Sigma_k}\QQ_{\le0}\,\omega_\sigma^\vee.
\]

\begin{corollary}

  For any non-compact $\sigma\in\Sigma_K$, the following formula
  holds:
  \[
    \omega_\osigma^\vee=\sum_{\tau\in\cG*\sigma} \omega_\tau^\vee.
  \]

\end{corollary}

Now we proceed with the construction of a root system. For this, we
define
\[
  \Phi_k:=W_k\Sigma_k=W_k\res_A'\Sigma_K,
\]
generalizing \eqref{eq:PhiKdef}. One can also consider the restriction
of the $K$-root system
\[
  \Phi_k^{\rm res}:=\res_A'\Phi_K=\res_A'W_K\Sigma_K
\]
but this is, in general, not a root system:

\begin{example}\label{rem:g2index}

  Let $k=\RR$ and $G=Sp(4,2)$. Then $\rk_\RR G=1$ with
  $S^0=\{\alpha_1,\alpha_3\}$. Let $H\subseteq G$ be a real form of
  item $4'$ of \Cite{Wasserman}{Table~C}. More precisely, let
  $N'\cong \G_a^3$ be the commutator subgroup of $N$. Then
  $H:=MAN'\subseteq G$. Put $X=G/H$. From Wasserman's table, we obtain
  \[
    \Sigma_\CC=\{\sigma_1:=\alpha_1+\alpha_3,\sigma_2:=\alpha_2\}.
  \]
  Thus, $\Sigma_\CC$ generates a root system of type $\sG_2$, with the
  long root $\sigma_1$ being compact. In particular,
  $\sigma_1+n\sigma_2\in\Phi_\CC$ for $n=1,2,3$. Therefore,
  $\Sigma_\RR=\{\osigma_2\}$, but
  \[
    \Phi_\RR^{\res}=\{\pm\osigma_2,\pm2\osigma_2,\pm3\osigma_2\}
  \]
  which is not a root system.

\end{example}

Nevertheless, we have:

\begin{theorem}\label{thm:PhiIndiv}

  Let $X$ be a $k$-dense $G$-variety. Then

  \begin{enumerate}

  \item\label{it:PhiIndiv1} $(\Phi_k,\Xi_k)$ is an integral root
    system. Its Weyl group is $W_k$, and $\Sigma_k$ is a system of
    simple roots.

  \item\label{it:PhiIndiv3} $n_\sigma\in\{1,2\}$ for all
    $\sigma\in\Sigma_k\prim$ (see \cref{cor:SSprim}).

  \item\label{it:PhiIndiv2} $\Phi_k$ consists precisely of the
    indivisible elements of $\Phi_k^{\rm res}$, i.e., those
    $\sigma\in\Phi_k^{\rm res}$ such that
    $\frac1n\sigma\not\in\Phi_k^{\rm res}$ for any $n\in\ZZ_{\ge2}$.

  \end{enumerate}

\end{theorem}

\begin{proof}

  \ref{it:PhiIndiv1} Let $R_K:=\<\Sigma_K\>_\ZZ$ and
  $R_k:=\res_A R_K=\<\Sigma_k\>_\ZZ$ be the root lattices. Since $R_K$
  is $W_K$-stable and therefore $N(\cN_k)$-stable, we see that $R_k$
  is $W_k$-stable. Since the elements of $\Sigma_k$ are primitive in
  $R_k$ by \cref{lem:linindep}, we conclude that $\Phi_k$ is a root
  system for $W_k$ and that $\Sigma_k$ is a set of simple roots.

  It remains to be shown that $\Xi_k$ is a set of integral weights for
  $\Phi_k$. Since $\Xi_k$ is $W_k$-stable, this means that $W_k$ acts
  trivially on $\Xi_k/R_k$. Clearly, this holds over $K$. The
  assertion now follows from the fact that $\Xi_k/R_k$ is a quotient
  of $\Xi_K/R_K$.

  \ref{it:PhiIndiv3} Let $\sigma\in\Sigma_k\prim$ with
  $\sigmaS:=n_\sigma\sigma\in\Sigma_k$. Since $\sigma\in\Xi_k$, it
  follows from \ref{it:PhiIndiv1} that
  \[
    \textstyle{\frac2{n_\sigma}}=
    \<\textstyle{\frac1{n_\sigma}}\sigmaS,\sigmaS^\vee\>=
    \<\sigma,\sigmaS^\vee\>\in\ZZ.
  \]

  \ref{it:PhiIndiv2} Clearly
  $\Phi_k\subseteq\Phi_k^{\rm res}\subseteq\res_A R_K=R_k$. Since all
  elements of $\Phi_k$ are primitive in $R_k$, they are all
  indivisible in $\Phi_k^{\rm res}$. Conversely, let $\sigma\in\Phi_K$
  such that $\osigma=\res_A\sigma\in\Phi_k^{\rm res}$ is
  indivisible. Let $\cH=\{\osigma=0\}\subseteq\cN_k$ be the hyperplane
  defined by $\osigma$, choose $a\in\cH$ not lying on any other
  reflection hyperplane, and an element $\wq\in W_k$ with
  $\wq a\in\cZ_k$. Then $\cF:=\wq\cH\cap\cZ_k$ is of codimension $1$
  in $\cZ_k$. Now lift $\wq$ to an element $w\in N(\cN_k)$. Then $\cF$
  is the intersection of $\cZ_k$ with $\cZ_K\cap\{w\sigma=0\}$. The
  latter is a face of $\cZ_K$, hence $\cF$ is a facet of $\cZ_k$. Let
  $\otau\in\Sigma_k$ correspond to $\cF$. Then $\wq\osigma$ is a
  multiple of $\otau$, and therefore equal to $\otau$. We conclude
  $\osigma=\wq^{-1}\otau\in\Phi_k$.
\end{proof}

Summarizing, to any $K$-variety $X$ with $G$-action, one can associate
the following combinatorial objects:
\[\label{eq:Luna-K}
  S,\ \Xi(T),\ \Xi_K,\ \Sigma_K,\ S\p.
\]
If $X$ is a $k$-dense $G$-variety, one additionally has a subset
$S^0\subseteq S$ of compact roots and a $\cG^*$-action on all
objects. By restriction to a maximal split $k$-torus (computable from
the data above by \cref{cor:compN}), one obtains the restricted
objects
\[\label{eq:Luna-k}
  S_k,\ \Xi(A),\ \Xi_k,\ \Sigma_k,\ S\p_k.
\]
We compute the spherical roots in a special case:

\begin{proposition}\label{prop:boundary2}
  Let $X$ be a $k$-dense $G$-variety, and let $\cF$ be a fan supported
  in $\cZ_k(X)$. For any $\cC\in\cF$, consider the stratum $Y:=X(\cC)$
  of $X(\cF)$. Then
  \begin{align*}\hspace{20pt}
    &\Sigma_k(Y)=\Sigma_k(X)\cap\<\cC\>^\perp,\\
    &\Sigma_K(Y)=\Sigma_K(X)\cap\<\cC\>^\perp=\{\sigma\in\Sigma_K(X)\mid
      \res_A\sigma\in\Sigma_k(Y)\cup\{0\}\}.
  \end{align*}
\end{proposition}

\begin{proof}

  \cref{prop:boundary}\,\ref{it:boundary4} implies
  $\Sigma_k\prim(Y)=\Sigma_k\prim(X)\cap\<\cC\>^\perp$. For $k=K$,
  this means $\Sigma_K(Y)=\Sigma_K(X)\cap\<\cC\>^\perp$. From this, we
  deduce $\Sigma_k(Y)=\Sigma_k(X)\cap\<\cC\>^\perp$ by restriction to
  $A$. Finally, let $\sigma\in\Sigma_K(X)$ and let
  $\osigma:=\res_A\sigma$. Then $\sigma\in\Sigma_K(Y)$ if and only if
  $\sigma\in\<\cC\>^\perp$ if and only if $\osigma\in\<\cC\>^\perp$
  (since $\cC\subseteq\cN(A)$) if and only if
  $\osigma\in\Sigma_k(Y)\cup\{0\}$.
\end{proof}

From this, we obtain:

\begin{corollary}\label{cor:kernel}

  Let $X$ be a $k$-dense $G$-variety. Then
  \begin{align*}\hspace{20pt}
    \Xi_K(X\an)&=\Xi_K^0(X),&\Sigma_K(X\an)&=\Sigma_K^0(X),&S\p(X\an)&=S\p(X),\\
    \Xi_K(X\el)&=\Xi_K(X),&\Sigma_K(X\el)&=\Sigma_K^0(X),&S\p(X\el)&=S\p(X).
  \end{align*}

\end{corollary}

\begin{proof}

  Let $\cC\subseteq\cZ_k(X)$ be a strictly convex cone of maximal
  dimension. Then $X(\cC)$ is induced from $X\an$, and the data for
  $X\an$ follow from \cref{prop:boundary} and
  \cref{prop:boundary2}. The first equation for $X\el$ follows from
  the $k$-LST, and the other two from the anisotropic case.
\end{proof}

\begin{remark}

  Since all orbits in $X\an$ are affine varieties, there are strong
  constraints on the pair $(\Sigma_K^0,S\p)$. For example, if $X$ is
  absolutely spherical, then $X\an$ is an affine, homogeneous,
  $K$-spherical variety. These have been classified by \cite{Kraemer},
  \cite{Mikityuk}, and \cite{BrionClassification}. See
  \cite{BraviPezzini} for the calculation of the corresponding
  spherical roots.
\end{remark}

The triple $(S_k,S_k\p,\Sigma_k)$ looks very much like the spherical
system of a spherical variety over $K$ (see \cite{LunaTypA}). This is
indeed often the case, but not always.

\begin{examples}

  \emph{i)} Let $k=\RR$ and $G=SU(n,n)$ with $n\ge2$. Then the
  restricted root system of $G$ is of type $\sC_n$ with simple roots
  $\beta_i=\oalpha_i=\oalpha_{2n-i}$ for $i=1,\ldots,n$. Consider the
  subgroup $H:=U(n,n-1)$ and $X=G/H$. Then $X$ is absolutely spherical
  of real rank $1$ with $\Sigma_\CC=\{\sigma\}$, where
  \[
    \sigma:=\alpha_1+\ldots+\alpha_{2n-1}.
  \]
  Thus, $\Sigma_\RR=\{\osigma\}$ with
  \[
    \osigma=2\beta_1+\ldots+2\beta_{n-1}+\beta_n.
  \]
  However, $\osigma$ is not a spherical root of any spherical variety
  over $\CC$ (see, e.g., \Cite{Wasserman}{Table\ 1} for a list of
  possible roots). Observe that rescaling $\beta_n$ to $\half\beta_n$
  converts $S_\RR$ to type $\sB_n$ and $\osigma$ to the spherical root
  of $SO(2n+1)/O(2n)$. Thus, one might wonder whether the
  normalization of $S_\RR$ is chosen correctly. The next example
  invalidates this objection.

  \emph{ii)} Let $k=\RR$ and let $G$ be the quasi-split group of type
  $\sE_6$ (type $EII$). Its restricted simple roots are
  \[
    \beta_1=\oalpha_2,\beta_2=\oalpha_4,\beta_3=\oalpha_3=\oalpha_5,
    \beta_4=\oalpha_1=\oalpha_6
  \]
  (Bourbaki notation \cite{Bourbaki}), which span a root system of
  type $\sF_4$. According to Berger \cite{Berger}, $G$ contains a
  subgroup $H$ that is isogenous to $SO(6,4)\times SO(2)$. Its spherical
  roots over $\CC$ are
  \[
    \sigma_1=\alpha_1+\alpha_3+\alpha_4+\alpha_5+\alpha_6,\
    \sigma_2=\alpha_2+\alpha_4+\half\alpha_3+\half\alpha_5
  \]
  (see \cite{BraviPezzini}). Their restrictions to $A$ are
  \[
    \osigma_1=\beta_2+2\beta_3+2\beta_4,\
    \osigma_2=\beta_1+\beta_2+\beta_3.
  \]
  They generate a root system of type $\sB_2$. The root $\osigma_1$ is
  again not a spherical root of any $\CC$-variety, but, worse,
  Wasserman's tables \cite{Wasserman} show that $\osigma_1$ and
  $\osigma_2$ are not the spherical roots of any $\sF_4$-variety over
  $\CC$, even if one allows rescaling of the $\beta_i$ or the
  $\osigma_j$.

\end{examples}

One can also look at the triple
\[
  \Xi_K,\Sigma_K,\Sigma_K^0
\]
together with the $\cG^*$-action, called the \emph{spherical index} of
$X$, which very much resembles, and indeed very often is, a
Satake--Tits diagram. However, there are counterexamples: Let $X=G/H$
be as in \cref{rem:g2index}. Then $\Sigma_\CC$ is of type $\sG_2$ and
one of the roots is compact. This never occurs for groups over any
ground field (see \Cite{Tits}{p.~61}).

Nevertheless, as in the group case (see \Cite{BorelTits}{4.9}), the
same major constraint for the structure of indices is still valid:

\begin{lemma}\label{thm:opposition}

  The spherical index $(\Xi_K,\Sigma_K,\Sigma_K^0)$ is invariant under
  the opposition map $\iota=-w_0$ (where $w_0\in W_K$ is the longest
  element).

\end{lemma}

\begin{proof}

  The $\cG^*$-action stabilizes both $W_K$ and $\Sigma_K$. This
  implies that $\cG^*$ commutes with $w_0$ and therefore with
  $\iota$. It remains to be shown that
  $\iota\Sigma_K^0=\Sigma_K^0$. Let
  $\Xi^0:=\ker(\res_A)\subseteq\Xi_K$. Then $N(\cN_k)$ (the normalizer
  of $\cN_k$ in $W_K$) acts on $\Xi^0$. Let $\wq_0\in W_k$ be the
  longest element and let $n_0\in N(\cN_k)$ be a lift. Then
  $\cZ_k=-\wq_0\cZ_k\subseteq -n_0\cZ_K$. Let $\cF\subseteq\cZ_K$ be
  the face corresponding to $\Sigma_K^0$. Since $\cZ_k$ contains an
  interior point of $\cF$, we obtain
  $\cF\subseteq\cZ_K\cap(-n_0\cZ_K)$. Thus, there exists
  $w^0\in W_K^0$ such that $-w^0n_0\cZ_K=\cZ_K$. By uniqueness, we
  obtain $w_0=w^0n_0$. Hence,
  \[
    \iota\Sigma_K^0=
    -w^0n_0\Sigma_K^0\subseteq\Sigma_K\cap\Xi^0=\Sigma_K^0. \qedhere\hfill\qed
  \]
\end{proof}

\section{The standard embedding}\label{sec:wonderful}

In this section, we use our results on root systems to investigate the
smoothness properties of the standard embedding of a $k$-semisimple
variety. As a direct consequence of the fact that the valuation cone
is simplicial (\cref{cor:cosimplicial}), we obtain from
\cref{prop:boundary}:

\begin{proposition}\label{thm:standard}

  Let $X$ be a $k$-semisimple $G$-variety with $\rk_kX=r$. Then the
  standard embedding $X\st$ contains $r$ irreducible divisors
  $D_1,\ldots,D_r$ (namely, the closures of the codimension-one
  strata) such that all strata of $X\st$ are set-theoretically equal
  to
  \[\label{eq:stratumI}
    X(I):=\bigcap_{i\not\in I}D_i\setminus\bigcup_{i\in I}D_i
  \]
  where $I$ ranges over all subsets of $\{1,\ldots,r\}$.

\end{proposition}

A more canonical way to parameterize strata is by their set of
spherical roots. To this end, let
$\Sigma_k(X)=\{\sigma_1,\ldots,\sigma_r\}$. Then $D_i$ is equal to the
closure of $X(\cR_i)$, where $\cR_i$ is the extremal ray of $\cZ_k(X)$
such that
$\sigma_i(\cR_i)\ne0$ or, equivalently, $\sigma_j(\cR_i)=0$ for all
$j\neq i$. Thus, $I$ can be viewed as a subset of $\Sigma_k(X)$.

\begin{proposition}\label{cor:localization}

  Let $X$ be a $k$-semisimple $G$-variety and
  $I\subseteq\Sigma_k(X)$. Then $X(I)=X(\cZ^I)$, where $\cZ^I$ is the
  face
  $\cZ_k\cap\bigcap_{\sigma\in I}\{\sigma=0\}=\cZ_k\cap\<I\>^\perp$ of
  $\cZ_k$. Hence, $X(I)$ is characterized by $\Sigma_k(X(I))=I$.

\end{proposition}

\begin{proof}
  The stratum $X(I)$ corresponds by construction to the face $\cC$ of
  $\cZ_k$ that is spanned by all rays $\cR_\sigma$ with
  $\sigma\not\in I$. This is equivalent to $\sigma(\cC)=0$ for all
  $\sigma\in I$. Thus, $\cC=\cZ^I$. Now \cref{prop:boundary2} implies
  $\Sigma_k(X(I))=I$.
\end{proof}

\begin{remark}\label{rmk:XJXI}

  The equality $\cZ^I=\cZ_k\cap\<I\>^\perp$ defines a face of $\cZ_k$
  even if $\cZ_k$ is not strictly convex. Observe also that
  $I\mapsto\cZ^I$ is an inclusion-reversing bijection between subsets
  of $\Sigma_k(X)$ and faces of $\cZ_k(X)$.

\end{remark}

We now address the problem of determining when the standard embedding
is smooth. Clearly, this is the case if and only if the torus embedding
$A_k(\cF\st)$ is smooth. A well-known criterion in toroidal geometry
asserts that this holds if and only if the cone $\cZ_k(X)$ is
spanned by an integral basis of
$\Hom(\Xi_k(X),\ZZ)\subseteq\cN_k(X)$. Dualizing, this leads to the
following definition:

\begin{definition}

  A $k$-dense $G$-variety $X$ is called \emph{\kwonderful} if
  $\Sigma_k\prim(X)$ forms a $\ZZ$-basis of $\Xi_k(X)$.

\end{definition}

Observe that \kwonderful{} varieties are automatically $k$-semisimple
and therefore have a standard embedding. Thus, we have:

\begin{proposition}\label{prop:smooth}

  Let $X$ be a $k$-dense $G$-variety. Then $X$ is \kwonderful{} if and
  only if $X$ is $k$-semisimple and the standard embedding is
  smooth. In this case, the union of the divisors $D_i$ of
  \cref{thm:standard} is a normal crossings divisor, and Equation
  \eqref{eq:stratumI} holds even scheme-theoretically. In particular,
  the closures of all strata are smooth. If $X$ is $k$-spherical, this
  means that all orbit closures in $X\st$ are smooth.

\end{proposition}

\begin{remarks} {\it i)} \kwonderful{} varieties are analogous to
  semisimple groups of adjoint type, since the latter are also
  characterized by the weight lattice coinciding with the root
  lattice. Note, however, that the concept of $k$-adjointness does not
  quite generalize this. Indeed, let $k=\CC$, $G=SL(2,\CC)$, and
  $X=G/\{e\}$. Then $G$ is not of adjoint type as a group, but it is
  $\CC$-adjoint as a $G$-variety. The reason for this is a matter of
  normalization: even though the simple root $\alpha$ of $G$ is not
  primitive in the weight lattice, the unique spherical root of $X$ is
  $\frac12\alpha$, which generates the weight lattice.
  
  {\it ii)} A typical example of a non-\kwonderful{}, $k$-semisimple
  variety is $X=SL(n,\CC)/SO(n,\CC)$, $n\ge3$, over $k=K=\CC$. Then
  $\Xi_k(X)$ equals twice the weight lattice of $SL(n,\CC)$, which is
  not a root lattice for any root system with Weyl group $W_k=S_n$.

  {\it iii)} Let $k=K$ be algebraically closed and $G$ be semisimple
  of adjoint type. Then, for $X=G/H$ a
  symmetric variety, the existence of a complete embedding as in
  \cref{prop:smooth} was established by De Concini and Procesi in
  \cite{DeConciniProcesi}. They named it the \emph{wonderful embedding
    of $X$}.

  {\it iv)} Subsequently, Luna, still assuming $k=K$, developed the
  concept of a \emph{wonderful variety}, defined as a smooth, complete
  $G$-variety $\Xq$ with open orbit $X$ such that $\Xq\setminus X$ is
  a normal crossings divisor with the property that all $X(J)$ as in
  \eqref{eq:stratumI} are orbits. Luna does not \emph{a priori}
  require $X$ to be spherical, but deduces it as a consequence in
  \cite{LunaWonderful}. Thus, over algebraically closed fields, Luna's
  wonderful varieties are precisely the standard embeddings of
  $k$-adjoint, $k$-spherical varieties. It is not known whether this
  generalizes to arbitrary fields of characteristic zero.

\end{remarks}

Any reductive group can be made adjoint by taking the quotient by its
center. A similar procedure exists for $G$-varieties $X$. Here, the
role of the center is played by the group $\fA_k=\fA_k(X)$ of
$k$-central automorphisms of $X$ (\cref{def:central}). In the
following discussion, we tacitly assume that the orbit space $X/\fA_k$
exists. This is certainly the case if $X$ is homogeneous. In general,
this can be achieved by replacing $X$ with a suitable dense open
subset. Here is the main result of this section:

\begin{theorem}\label{thm:Autoroot}

  Let $X$ be a $k$-dense $G$-variety. Then the quotient $X/\fA_k$ is
  \kwonderful{}.

\end{theorem}

The theorem will be proved simultaneously with \cref{prop:Autoroot}
below. Before doing so, we state a simple consequence which is the
counterpart to the sufficient condition \cref{cor:centralsplit} for
$k$-semisimplicity. The proof is similar.

\begin{corollary}\label{cor:centralunity}

  Let $X=G/H$ be homogeneous and assume that the center of
  $\Aut^G(X)=N_G(H)/H$ does not contain any subgroup isomorphic to
  $\mu_n$ (the group of $n$-th roots of unity) for $n\ge2$ (e.g., if
  $H$ is self-normalizing). Then $X$ is $k$-adjoint.

\end{corollary}

\begin{remarks} {\it i)} Recall that $\fA_k$ is finite if and only if
  $X$ is $k$-semisimple (\cref{cor:finiteconvex}). Thus, for any
  $k$-semisimple variety $X$, there exists a finite abelian group $E$ of
  automorphisms (namely, $E=\fA_k$) such that $X/E$ is \kwonderful{}.

  {\it ii)} Observe that even in the absolutely spherical case, there
  are examples where $\fA_k(X/\fA_k(X))\ne1$. The simplest one occurs
  for $k=\CC$ and $X=GL(2,\CC)/GL(1,\CC)$. Thus, \cref{thm:Autoroot}
  cannot be directly reduced to the case $\fA_k=1$.

\end{remarks}

Let
\[
  \Gamma_k(X):=\Xi_k(X/\fA_k)=\ker\big(\Xi_k(X)\to\Xi(\fA_k)\big).
\]
Another way to phrase the theorem is that $\Gamma_k(X)$ is the root
lattice of a root system. In fact, $\Gamma_k(X)$ is generated by
\[\label{eq:autSigma}
  \Sigma_k\aut(X):=\Sigma_k\prim(X/\fA_k),
\]
which forms a set of simple roots of the root system
\[
  \Phi_k\aut(X):=\Phi\prim_k(X/\fA_k)=W_k(X)\Sigma_k\aut(X)\subseteq\Xi_k(X).
\]
We record some additional properties:

\begin{proposition}\label{prop:Autoroot}

  Let $X$ be a $k$-dense $G$-variety. Then:

  \begin{enumerate}

  \item\label{it:Autoroot1} $\Xi_k(X)$ is a weight lattice for
    $\Phi_k\aut(X)$. In particular, there exists a map
    $\Sigma_k\prim(X)\to\{1,2\}:\sigma\mapsto n\aut_\sigma$ such that
    $\Sigma_k\aut(X)=\{n\aut_\sigma\sigma\mid\sigma\in\Sigma_k\prim(X)\}$.

  \item\label{it:Autoroot2} $\Gamma_k(X)=\res_A\Gamma_K(X)$ and
    $\Sigma_k\aut=\res_A'\Sigma_K\aut$.

  \item\label{it:Autoroot3} If $X$ is \Kwonderful{}, then it is also
    \kwonderful{}.
 
  \end{enumerate}

\end{proposition}

\begin{proof}

  For $k=K$, \cref{thm:Autoroot} is exactly \Cite{KnopAut}{Cor.\ 6.5 b)}. From
  this, we derive the general case of \cref{thm:Autoroot} and
  \cref{prop:Autoroot}.

  Let $X_k:=X/\fA_k(X)$. By \cref{lemma:definedoverk}, the group
  $\fA_K(X)$ and its action on $X$ are defined over $k$. Therefore,
  the variety $X_K:=X/\fA_K(X)$ is also defined over $k$. Since
  $\fA_k\subseteq\fA_K$, these spaces are connected by the surjective
  morphisms
  \[
    X\auf X_k\auf X_K.
  \]
  We have by definition $\Xi_k(X_k)=\Gamma_k:=\Gamma_k(X)$ and
  $\Xi_K(X_K)=\Gamma_K$. The equality $\fA_k=\fA_K\cap A_k$
  (Eq.~\eqref{eq:fAA}) implies that $A_k/\fA_k\to A_K/\fA_K$ is
  injective. Passing to character groups, this means that
  \[
    \res_A:\Gamma_K\to\Gamma_k
  \]
  is surjective, which proves the first half of
  \ref{it:Autoroot2}. From
  \[
    \Gamma_K=\Xi_K(X_K)\auf\Xi_k(X_K)\subseteq\Xi_k(X_k)=\Gamma_k,
  \]
  we obtain $\Xi_k(X_K)=\Gamma_k$. Now let
  \[
    \Sigma'_k:=\res'_A\Sigma_K\aut(X)=\res'_A\Sigma_K(X_K)=\Sigma_k(X_K),
  \]
  which is a linearly independent subset of $\Gamma_k$. Furthermore,
  \Cite{KnopAut}{Cor.~6.5 b)} asserts that $\Sigma_K(X_K)$ is a
  $\ZZ$-basis of $\Gamma_K$. It follows that $\Sigma'_k$ generates
  $\Gamma_k$. Hence, it is also a $\ZZ$-basis of $\Gamma_k$; therefore,
  $\Sigma'_k=\Sigma_k\prim(X_k)=\Sigma_k\aut(X)$. This proves
  \cref{thm:Autoroot} and the second half of \ref{it:Autoroot2}.

  An even simpler argument proves \ref{it:Autoroot3}: if $X$ is
  \Kwonderful{}, then $\Sigma_K(X)$ is a $\ZZ$-basis of
  $\Xi_K(X)$. Restriction to $A$ shows that $\Sigma_k(X)$ generates,
  and therefore is, a $\ZZ$-basis of $\Xi_k(X)$. Since
  $\Sigma_k(X)=\Sigma_k\prim(X)$, it follows that $X$ is \kwonderful{}
  as well.

  Finally, it is known that $W_K$ acts trivially on $\Xi_K/\Gamma_K$
  (by \Cite{KnopAut}{Cor.~6.5 c)}). It follows that $W_k$ acts
  trivially on the quotient $\Xi_k/\Gamma_k$. This means that $\Xi_k$
  is a weight lattice for $\Phi_k\aut$ and proves \ref{it:Autoroot1}.
\end{proof}

\begin{remark}

  All elements of $\Sigma_K\aut(X)$ are integral multiples of elements
  of $\Sigma_K(X)=\Sigma_K\prim(X)$. It follows by restriction that
  all elements of $\Sigma_k\aut(X)$ are integral multiples of elements
  of $\Sigma_k(X)$. This implies $n_\sigma\le n\aut_\sigma$, where
  $n_\sigma$ was defined in \cref{cor:SSprim}.

\end{remark}

\numberwithin{equation}{subsection}\swapnumbers

\section{Boundary orbits and boundary degenerations}
\label{sec:BoundaryDegenerations}

In this section, we work with a fixed homogeneous $k$-spherical
variety $X$. Therefore, we drop references to $X$, e.g.,
$\Xi_k:=\Xi_k(X)$, $\cN_k:=\cN_k(X)$, etc.

We study the homogeneous varieties that occur as orbits in a toroidal
embedding of $X$. These will be called \emph{boundary orbits of $X$}.

\subsection{\it The minimal boundary orbits}\

First, we show that there is a finite family of homogeneous
$G$-varieties such that every boundary orbit is a principal
split-torus bundle over one member of this family.

The definition of this family is very simple. Let $\fA_k^0$ be the
identity component of $\fA_k$, and consider the quotient
$\hat X:=X/\fA_k^0$. Its valuation cone is $\hat\cZ_k=\cZ_k/\cN_k^0$,
which is therefore strictly convex. Hence $\hat X$, being
$k$-semisimple, admits a standard embedding $\hat X\st$. Observe that
$\Sigma_k(\hat X)=\Sigma_k(X)$. Thus, for any $I\subseteq\Sigma_k$, we
\emph{define $\hat X(I)$ to be the orbit corresponding to
  $\hat\cZ^I:=\cZ^I/\cN_k^0$ inside $\hat X\st$.}

To describe the relation of $\hat X(I)$ to boundary orbits, we need
some notation. For any subspace $V\subseteq \cN_k$, let
$A_V\subseteq A_k$ be the unique subtorus with $\cN_k(A_V)=V$, and
define $\Xi_k^V:=\Xi_k\cap V^\perp=\Xi_k(A_k/A_V)$ as well as
$\Sigma_k^V:=\Sigma_k\cap\Xi_k^V=\Sigma_k\cap V^\perp$. For a subspace
$U\subseteq\Xi_k\otimes_\ZZ\QQ$, define, similarly,
$A^U:=A_{U^\perp}=\bigcap_{\chi\in\Xi_k\cap U}\ker\chi$.

Every face of $\cZ_k$ is of the form $\cZ^I=\cZ_k\cap\<I\>^\perp$ for
some $I\subseteq\Sigma_k$. The torus $A^{\<I\>}$ is the identity
component of the kernel of $(\sigma)_{\sigma\in I}:A_k\to\G_m^I$.

\begin{stheorem}\label{prop:HC-HI}
  Let $\cF$ be a fan supported on $\cZ_k$, and let $\cC\in\cF$ be a face. Put
  $I:=\Sigma_k^{\<\cC\>}$. Then the orbit $X(\cC)\subseteq X(\cF)$ is
  a $G$-equivariant principal fiber bundle over $\hat X(I)$ for the
  torus $A^{\<I\>}/A_{\<\cC\>}$. This means, in particular, that if
  $x$ is a $k$-rational point of $X(\cC)$ and $\hat x$ is its image in
  $\hat X(I)$, then $G_x$ is normal in $G_{\hat x}$ with
  $A^{\<I\>}/A_{\<\cC\>}=G_{\hat x}/G_x$.
\end{stheorem}

\begin{proof}
  We may assume that $\cF$ is the fan of faces of $\cC$. Then
  $X(\cC)\subseteq X(\cF)$ is the unique closed orbit. The morphism
  $f:X\to\hat X$ maps $\cC$ into the face $\hat\cZ^I$, and $I$ is
  maximal with this property. Hence $f$ extends, for instance, by
  \cref{prop:morphism2}, to a morphism $X(\cF)\to\hat X\st$ such that
  $X(\cC)$ is mapped onto $\hat X(I)$. By the Local Structure Theorem
  applied to $\hat X(I)$, the fibers of $X(\cC)\to\hat X(I)$ are
  isomorphic to the fibers of
  \[
    A_k/A_{\<\cC\>}=A_k(X(\cC))\to A_k(\hat
    X(I))=(A_k/\fA_k^0)/(A^{\<I\>}/\fA_k^0)=A_k/A^{\<I\>},
  \]
  which implies the assertion. Finally, the last identification holds
  because both sides are isomorphic to the fiber of
  $X(\cC)\to\hat X(I)$ over $\hat x$.
\end{proof}

Boundary orbits also have upper bounds. More precisely, for every
boundary orbit $X(\cC)$ we construct a $k$-spherical $G$-variety $X_I$
(with $I=\Sigma_k^{\<\cC\>}$) that is a principal bundle over $X(\cC)$
for a split torus. These varieties $X_I$ are called \emph{boundary
  degenerations of $X$} since they are related to $X$ via a
flat---indeed smooth---degeneration. Over algebraically closed fields,
other names are also used, e.g., \emph{localizations in $\Sigma$}
\cite{KnopLoc}, or \emph{satellites} \cite{BatyrevMoreau}. In the
following, we present several degeneration processes in order of
increasing generality.
    
\subsection{\it Boundary degenerations along a one-parameter
  subgroup}\label{subsection:Bda}\

In the first construction, we deform $X$ over the affine line
$\A_k^1$. To this end, we start with a $k$-flat $F\subseteq X$ with
base point $x_1\in F(k)$ and a one-parameter subgroup
$\lambda:\G_m\to A_k$ lying in $\cZ_k$. For any $s\in\G_m$, set
$x_s:=\lambda(s^{-1})x_1\in F$. Let $H_s:=G_{x_s}$ be the
isotropy group. The degeneration of $H=H_1$ in $G$ is defined as
$H_0=\lim_{s\to0}H_s$. The boundary degeneration is then
$X_\lambda:=G/H_0$.
    
To make this precise, consider the $G\b:=G\times\G_m$-variety
$X\b:=X\times\G_m$ and let $p:X\b\to\G_m$ be the projection to the
second factor. Then $X\b$ is $k$-spherical with valuation cone
$\cZ_k\b:=\cZ_k(X\b)=\cZ_k\times\QQ$, where the isomorphism
$\cN_k(X\b)\cong\cN_k\times\QQ$ is normalized such that the second
component of $a\in\cN_k(X\b)$ equals $v_a(p)$. Thus $(\lambda,1)$
corresponds to the one-parameter subgroup
$t\mapsto(\lambda(t),t^{-1})$.  The ray $\cC$ in $\cZ_k\b$ generated
by $(\lambda,1)$ defines a toroidal two-orbit embedding
$X\b\into\Xq\b_\lambda$. Since $p$, considered as a function, has a
valuation of $1$ with respect to the valuation corresponding to
$(\lambda,1)$, the morphism $p$ extends to a smooth morphism
$\pq:\Xq\b_\lambda\to S\b:=\A_k^1$ such that the closed orbit
$X_\lambda$ is the fiber over $0$. Any other fiber over a $k$-rational
point is isomorphic to $X$.
\[
  \cxymatrix{X\times\G_m\ar[dr]\ar@{=}[r]&X\b\ar[d]^<<<p\ar@{^(->}[r]&\Xq\b_\lambda\ar[d]^<<<\pq&X_\lambda\ar[d]\ar@{_(->}[l]\\
    &\G_m\ar@{^(->}[r]&\A_k^1&0\ar@{_(->}[l]}
\]
The next \namecref{lemma:Zqd} states the main properties of
$X_\lambda$. It will be proved as part of the more general
\cref{lemma:smoothhomogeneous} below.

\begin{sproposition}\label{lemma:Zqd}
  $X_\lambda$ is a $k$-spherical $G$-variety with
  $ \Xi_k(X_\lambda)=\Xi_k$ and
  $\Sigma_k(X_\lambda)=\Sigma_k^{\<\lambda\>}$. The section
  $\G_m\to X\b:s\mapsto(x_s,s)$ of $p$ extends to a section
  $\delta:\A_k^1\to \Xq\b_\lambda$ of $\pq$. Then
  $x_0:=\delta(0)\in X_\lambda(k)$ and $X_\lambda\cong G/H_0$, where
  $H_0=G_{x_0}\subseteq G$.
\end{sproposition}

\subsection{\it Boundary degenerations along the valuation cone}\

Next, we describe a construction that yields all degenerations
$X_\lambda$ simultaneously. This will show, in particular, that
$X_\lambda$ depends, up to isomorphism, only on
$I=\Sigma_k^{\<\lambda\>}$. The naive approach would be to replace the
parameter space $\A_k^1$ by the embedding of $A_k$ corresponding to
the cone $\cZ_k$. If $\cZ_k$ is not strictly convex, however, such an
embedding does not exist. Therefore, we take the quotient by the edge
$\cN_k^0=\cZ_k\cap(-\cZ_k)$ or, more geometrically, the identity
component $\fA_k^0$ of $\fA_k$ (see \cref{thm:auto}).

More precisely, let $S\red$ be the toroidal embedding of
$\hat A_k:=A_k(\hat X)=A_k/\fA_k^0$ corresponding to the strictly
convex cone $\hat\cZ_k=\cZ_k/\cN_k^0$. The $A_k$-orbits of $S\red$ are
in bijection with the faces of $\hat\cZ_k$ or, equivalently, with
those of $\cZ_k$. Specifically, for every one-parameter subgroup
$\lambda$ in the relative interior of $\cZ^I$, the points $\lambda(s^{-1})$
converge as $s\to0$ to a $k$-rational point $a_I\in S\red$ in the
orbit corresponding to $I$.

Since $\fA_k^0$ acts on $X$, we can form the quotient
$X\red:=(X\times A_k)/\fA_k^0$. It is a $k$-spherical variety for the
group $G\red:=G\times A_k$. The projection to the second factor yields a
morphism $p:X\red\to\hat A_k$. The diagonal copy $\|diag|\hat\cZ_k$ of
$\hat\cZ_k$ inside $\cZ_k(X\red)=(\cZ_k\times\cN_k)/\Delta\cN_k^0$
defines a toroidal embedding $X\red\into\Xq\red$ yielding the diagram
\[
  \cxymatrix{X\times^{\fA_k^0}A_k\ar@{=}[r]\ar[dr]&X\red\ar[d]^<<<p\ar@{^(->}[r]&\Xq\red\ar[d]^<<<\pq&X_I\ar[d]\ar@{_(->}[l]\\
    &\hat A_k\ar@{^(->}[r]&S\red&a_I\ar@{_(->}[l]}
\]
Now the boundary degeneration $X_I$ of $X$ is defined as the fiber
\[
  X_I:=\pq^{-1}(a_I)\subseteq\Xq\red.
\]

\begin{slemma}\label{lemma:smoothhomogeneous}
  All boundary degenerations $X_I$ are $k$-spherical, homogeneous
  $G$-varieties. The invariants $Q_k$, $X\an$, $X\el$, $\Xi_k$, and
  $A_k$ are the same for $X$ and $X_I$, whereas $\Sigma_k(X_I)=I$.
  Moreover, if $\lambda$ is a one-parameter subgroup in the relative
  interior of $\cZ^I$, then there is a canonical isomorphism
  $X_\lambda\cong X_I$.
\end{slemma}

\begin{proof}
  Each flat in $X\red$ is isomorphic to the $A_k\times A_k$-variety
  $F=(A_k\times A_k)/\fA_k^0$. Let $\FQ$ be its embedding
  corresponding to $\|diag|\hat\cZ_k$. Then, according to
  \cref{prop:boundary}, the $G\red$-variety $\Xq\red$ contains an open
  subset $\Xq_0\red$ that is fibered over $U\times X\red\an$ with
  general fiber $\FQ$. The morphism $p:X\red\to\hat A_k$ restricts on
  each fiber $F$ to the projection
  $p_2:F\to\hat A_k:[a_1,a_2]\mapsto[a_2]$ to the second factor. The
  isomorphism
  \[\label{eq:FAAA}
    F=(A_k\times A_k)/\fA_k^0\overset\sim\to A_k\times (A_k/\fA_k^0):
    [x,a]\mapsto(a^{-1}x,[a])
  \]
  (with $[\cdot]$ denoting cosets) maps $\|diag|\hat\cZ_k$ to
  $0\times\hat\cZ_k$, which yields an isomorphism of embeddings
  \[\label{eq:FQ}
    f:\FQ\overset\sim\to A_k\times S\red
  \]
  such that $\pq\circ f^{-1}:\FQ\to S\red$ is the projection to the
  second factor. Since the projection is smooth, this shows that
  $p:X\red\to\hat A_k$ can be extended to a smooth morphism
  $\pq:\Xq\red\to S\red$. Observe that the $A_k\times1$-action on $F$
  is still multiplication on the first factor under the isomorphism
  \eqref{eq:FAAA}. This implies that $G$ acts transitively and $Q_k$
  has a dense orbit on each fiber of $p$. Therefore, all fibers of
  $\pq$ over $k$-rational points, and, in particular, $X_I$, are
  homogeneous $k$-spherical $G$-varieties. Moreover, all these fibers
  have the same elementary kernel, showing $(X_I)\el=X\el$ and,
  therefore, the first five equalities.

  For the last one, let $Y:=\hat A_k a_I\subseteq S\red$. Then
  $\pq^{-1}(Y)\cong X_I\times Y$ is the $G\red$-orbit
  $X\red(I):=X\red(\hat\cZ^I)$ of $\Xq\red$ corresponding to the face
  $\hat\cZ^I$ of $\|diag|\hat\cZ_k$. Hence
  $\Sigma_k(X_I\times Y)=\Sigma_k(X\red(I))=I$ by
  \cref{prop:boundary2}, which implies that $\Sigma_k(X_I)=I$.

  Finally, we show that $\Xq_\lambda\b$ is the pullback of $\Xq\red$
  via $\ell:\A_k^1\to S\red$, i.e., that the following diagram is
  Cartesian, where $\ell$ is the extension of
  $\G_m\to\hat A_k:s\mapsto[\lambda(s^{-1})]$, which exists because
  $[\lambda]\in\hat\cZ_k$.
  \[\label{eq:ZXq1}
    \cxymatrix{\Xq\b_\lambda\ar[d]^\pq\ar[r]&\Xq\red\ar[d]^\pq\\
      \A_k^1\ar[r]^\ell&S\red}
  \]
  This shows that $X_\lambda\cong X_I$, since $0\in\A_k^1$ is mapped
  to $a_I\in S\red$.

  To this end, we define the top arrow as the extension of
  $X\times\G_m=X\b\to X\red:(x,s)\mapsto[x,\lambda(s^{-1})]$. This
  extension exists because the ray spanned by $(\lambda,1)$ is mapped
  to the ray spanned by $[\lambda,\lambda]\in\|diag|\hat\cZ_k$. The
  commutativity of the diagram can be checked on the open orbits,
  where it is obvious.

  To show that the diagram is Cartesian, we argue as above: from the
  Local Structure Theorem, we get open subsets of $\Xq\b_\lambda$ and
  $\Xq\red$ that are fibered over $X\an$ with fibers
  $\overline{A_k\times\G_m}$ and $\FQ$, respectively. Thus, it
  suffices to show that the middle square of the diagram
  \[\label{eq:ZXq2}
    \xymatrix{A_k\times\A_k^1\ar[rr]^\sim\ar[d]_p&&\overline{A_k\times\G_m}\ar[d]_p\ar[rr]^{[a_1,\lambda(s^{-1})]}&&\FQ\ar[d]^{[a_2]}\ar[rr]^\sim&&A_k\times S\red\ar[d]^{[a_2]}\\
      \A_k^1\ar@{=}[rr]&&\A_k^1\ar[rr]^{[\lambda(s^{-1})]}&&S\red\ar@{=}[rr]&&S\red
    }
  \]
  is Cartesian. The right-hand square is induced by the isomorphism
  \eqref{eq:FQ}, while the left-hand square uses the analogous
  isomorphism
  \[
    A_k\times\G_m\overset\sim\to
    A_k\times\G_m:(x,t)\mapsto(\lambda(t^{-1})x,t).
  \]
  We conclude by noting that the outer rectangle is clearly Cartesian.
\end{proof}

We now relate boundary degenerations to boundary orbits.

\begin{sproposition}\label{prop:BdegBorb}
  Let $\cC\subseteq\cZ_k$ be a strictly convex cone, and let
  $I=\Sigma_k^{\<\cC\>}$. Then $A^{\<I\>}$ acts freely on
  $X_I$. Moreover, $\hat X(I)=X_I/A^{\<I\>}$ and
  $X(\cC)=X_I/A_{\<\cC\>}$, which induces an equivariant fibration
  \[
    X_I \auf X(\cC)\auf \hat X(I).
  \]
\end{sproposition}

\begin{proof}
  Since $X\red=(X\times A_k)/\fA_k^0$, the first projection yields a
  morphism $X\red\to\hat X$. As a result, $\|diag|\hat\cZ_k$ is mapped
  isomorphically to $\hat\cZ_k$. Thus, we get an extension
  $\Xq\red\to\hat X\st$. Moreover, the orbit $X_I$ is mapped to
  $\hat X(I)$. From $\Sigma_k(X_I)=I$, it follows that
  $\cN_k^0(X_I)=\<I\>^\perp=\cN_k(A^{\<I\>})$. Thus, $A^{\<I\>}$
  indeed acts freely on $X_I$. Moreover, $\Xi_k(X_I)=\Xi_k(X)$, while
  $\Xi_k(\hat X(I))=\Xi_k(X)\cap\<\cZ^I\>^\perp=
  \Xi_k(X)\cap\<I\>_\QQ=\Xi_k(A_k/A^{\<I\>})$. This implies that
  $\Xi_k(X_I/A^{\<I\>})=\Xi_k(\hat X(I))$ and therefore
  $X_I/A^{\<I\>}=\hat X(I)$. We have already seen that
  $X(\cC)/(A^{\<I\>}/A_{\<\cC\>})=\hat X(I)$, which implies
  $X(\cC)=X_I/A_{\<\cC\>}$.
\end{proof}

\begin{scorollary}\label{cor:twoboundaryorbits1}
  Two boundary orbits $X(\cC)$ and $X(\cC')$ are isomorphic if and
  only if $\<\cC\>=\<\cC'\>$. Moreover,
  $X(\cC)=\hat X(I)$ if and only
  if $\<\cC\>=\<I\>^\perp$, i.e., $\cC$ is a subcone of maximal
  dimension of the face $\cZ^I\subseteq\cZ_k$. In particular, every
  $\hat X(I)$ is a boundary orbit of $X$.
\end{scorollary}

On the level of isotropy groups, \cref{prop:BdegBorb} yields:

\begin{scorollary}
  The choice of a base point $x_I\in X_I(k)$ yields $k$-rational base
  points $x(\cC)\in X(\cC)$ and $\hat x(I)\in\hat X(I)$ with isotropy
  groups $H_I$, $H(\cC)$, and $\hat H(I)$, respectively. Then $H_I$ is
  normal in $\hat H(I)$ with $\hat H(I)/H_I=A^{\<I\>}$. Moreover,
  $H_I\subseteq H(\cC)\subseteq\hat H(I)$ with
  $H(\cC)/H_I=A_{\<\cC\>}$.
\end{scorollary}

If we specialize to $I=\leer$, the minimal boundary orbit $\hat X(I)$
has rank $0$, and therefore $X_\leer$ is a $k$-horospherical variety
$G/H_0$ with $R_uH_0=U^-$ and $H_0\cap L=H\cap L$ (recall that
$Q_k(X)=LU$). Since all other minimal boundary orbits are not
horospherical, we get:

\begin{scorollary}\label{cor:interior}
  Let the base point $x_0$ of the homogeneous $k$-spherical variety
  $G/H$ be chosen such that $Q_k(X)x_0$ is open in $X$ and $A_kx_0$ is
  a flat. Let $\lambda:\G_m\to A_k$ be a one-parameter subgroup of
  $A_k$ lying in $\cZ_k$. Then $X_\leer=G/(H\cap L)\ltimes U^-$ (with
  $Q_k=LU$), and
  \[
    \lim_{t\to0}\lambda(t^{-1})\cdot\fh=(\fh\cap\fl)\oplus\fu^-
  \]
  (inside the Graßmannian of $\fg$) if and only if $\lambda$ lies in
  the interior of $\cZ_k$.
\end{scorollary}

Following \cite{BrionSymetrique}, this can be used to construct a
distinguished basis for $\fh$ from which one can read off $\fA_k$ and
$\cZ_k$.

\begin{stheorem}
  Let $X=G/H$ and $Q_k=LU$ be as in \cref{cor:interior}. Let
  $L_0:=Q_k\cap H=L\cap H$. Then there is an $L_0$-equivariant linear
  map $\Upsilon:\fu^-\to\Fq$ such that
  \[\label{eq:fhfl0}
    \fh=\fl_0\oplus\{\xi+\Upsilon(\xi)\mid\xi\in\fu^-\}.
  \]
  Moreover, let $\Upsilon_{\alpha\beta}:\fu_\alpha^-\to\Fq_\beta$ be
  the induced maps between weight spaces with respect to a maximal
  split torus $A$ of $Q_k$ and let
  \[
    \cM:=\{\beta-\alpha\mid\Upsilon_{\alpha\beta}\ne0\}.
  \]
  Then $\cM$ and $\Sigma_k\aut$ generate the same saturated monoid or,
  equivalently,
  \[\label{eq:mua1}
    \fA_k=\{a\in A_k\mid\mu(a)=1\text{ for all }\mu\in\cM\}
  \]
  and
  \[\label{eq:mua2}
    \cZ_k=\{u\in\cN_k\mid \mu(u)\le0\text{ for all }\mu\in\cM\}.
  \]
\end{stheorem}

\begin{proof}
  Let $\overline\fh\subseteq\fh$ be an $L_0$-invariant complement to
  $\fl_0$. Then $\overline\fh\oplus\Fq=\fh+\Fq=\fg=\fu^-\oplus\Fq$
  implies that for every $\xi\in\fu^-$, there is a unique
  $\Upsilon(\xi)\in\Fq$ with $\xi+\Upsilon(\xi)\in\overline\fh$. This
  yields \eqref{eq:fhfl0}.

  The saturated monoid generated by $\cM$ is the intersection of the
  group $\<\cM\>_\ZZ$ and the convex cone $\QQ_{\ge0}\cM$. Now
  \eqref{eq:mua1} and \eqref{eq:mua2} are equivalent to
  $\<\cM\>_\ZZ=\Xi(A_k/\fA_k)$ and $\QQ_{\ge0}\cM=-\cZ_k^\vee$,
  respectively. On the other hand,
  $\Xi(A_k/\fA_k)=\<\Sigma_k\aut\>_\ZZ$ (by \eqref{eq:autSigma} and
  \cref{thm:Autoroot}), and $-\cZ_k^\vee=\QQ_{\ge0}\Sigma_k$ (by
  \eqref{eq:cZsigma}). Thus, the assertion about saturated monoids is
  indeed equivalent to equations \eqref{eq:mua1} and \eqref{eq:mua2}.

  In order to prove these equations, observe first that the
  $L_0$-equivariance of $\Upsilon$ implies that every $\mu\in\cM$ is
  trivial on $A\cap L_0$. This implies that $\cM$ is actually a subset
  of $\Xi_k$.

  The group $\fAS:=\{a\in A_k\mid\mu(a)=1\text{ for all }\mu\in\cM\}$
  is the set of all $a\in A_k$ such that $\Upsilon$ is
  $\|Ad|(a)$-equivariant. This is equivalent to $\fh$ being normalized
  by $a$. Thus, $\fA_k\subseteq\fAS$. Conversely, let
  $a\in\fAS(k)$. Then $a$ induces an automorphism of
  $X^0=G/H^0$. Since $a$ acts trivially on $X\an=L/(A\cap L_0)A_k$, it
  induces a central automorphism of $X^0$. On the other hand,
  $X=X^0/E$, where $E\subseteq\Aut^G(X^0)$ is a (finite)
  subgroup. Thus, since $a$ is central, it commutes with $E$ by
  \Cite{KnopAut}{Cor.\ 5.6} and, therefore, induces a central
  automorphism of $X$, proving $\fAS\subseteq\fA_k$.

  For \eqref{eq:mua2}, let
  $\cZS_k:=\{u\in\cN_k\mid \mu(u)\le0\text{ for all }\mu\in\cM\}$, and
  let $\lambda\in\|Hom|(\G_m,A_k)\cap\cZ_k$. Then we know from
  \cref{cor:interior} that $\lambda\in\cZ_k^0$ if and only if
  $\lim_{t\to0}\lambda(t^{-1})\fh=\fl_0\oplus\fu^-$, where the limit is to
  be taken in the Graßmannian of $\fg$. The latter condition is
  equivalent to $\lim_{t\to0}\lambda(t)^{-1}\Upsilon\lambda(t)=0$ in
  $\|Hom|_k(\fu^-,\Fq)$, which in turn means
  $\lambda\in\cZS_k^0$. Thus, we get
  \[
    \cZS_k^0\cap\cZ_k=\cZ_k^0.
  \]
  It follows that $\cZ_k^0$ is both closed and open in $\cZS_k^0$ and,
  therefore, $\cZ_k^0=\cZS_k^0$. Taking closures yields the assertion
  $\cZ_k=\cZS_k$.
\end{proof}

\subsection{\it Boundary degenerations as deformation to the normal
  bundle}\

The boundary degeneration $X_I$ admits a more geometric construction
using the technique of degeneration to the normal bundle.

We recall from \Cite{Fulton}{Chap.~5}: let $Y$ be a smooth closed
subvariety of a smooth variety $X$. Let $X'\to X\times\A_k^1$ be the
blow-up of $X\times\A_k^1$ along $Y\times\{0\}$. Moreover, let
$X\nor:=X'\setminus E$, where $E\subseteq X'$ is the proper transform
of the divisor $X\times\{0\}$. Then the composition
$\pi:X\nor\to X\times\A_k^1\to\A_k^1$ has the property that it is
smooth (and, thus, flat), that the fibers over $t\in\G_m(k)$ are
isomorphic to $X$, and that the fiber over $0\in\A_k^1$ is the normal
bundle $N_YX$.

Now we apply this construction to a toroidal embedding $\Xq=X(\cF)$ of
a $k$-spherical variety $X$, where $\cF$ is the fan of faces of a cone
$\cC\subseteq\cZ_k$, and to $Y=X(\cC)$, the unique closed orbit. In
order for $\Xq$ to be smooth, we assume that $\cC$ is generated by
elements $a_1,\ldots,a_s$ that are part of a basis of
$\|Hom|(\G_m,A_k)$. Then:

\begin{sproposition}
  Let $\Xq\nor\to\A_k^1$ be the degeneration of $\Xq$ to the normal
  bundle $N_Y\Xq$ as above, and let $\lambda:\G_m\to A_k$ be a
  one-parameter subgroup in the relative interior of $\cC$.

  \begin{enumerate}

  \item\label{it:Flat1} The limit
    $n_0:=\lim_{t\to0}\,(\lambda(t^{-1})x_1,t)$ exists in $\Xq\nor$, is
    independent of the choice of $\lambda$, and lies in $N_Y\Xq$.

  \item\label{it:Flat2} The orbit $G\cdot n_0$ is open in $N_Y\Xq$ and
    is isomorphic to the boundary degeneration $X_I$, where
    $I=\Sigma_k^{\<\cC\>}$.

  \item\label{it:Flat3} Let $D_i\subset\Xq$ be the $G$-invariant
    divisor corresponding to the extremal ray
    $\QQ_{\ge0}a_i\subseteq \cC$. Then
    $G\cdot n_0=N_Y\Xq\setminus \bigcup_iN_YD_i$.

  \item\label{it:Flat4} Let $\chi_i$ be the character of $H_0=G_{x_0}$
    acting on the normal space $N_{D_i,x_0}\Xq$ of the divisor $D_i$
    in $\Xq$ at $x_0$. Then $G_{n_0}$ is the
    common kernel of the $\chi_i$.

  \end{enumerate}
\end{sproposition}

\begin{proof}
  Observe that $\Xq\nor$ is a toroidal embedding of $X\b=X\times\G_m$
  for the group $G\b=G\times\G_m$. More precisely, we claim that
  $\Xq\nor=X\b(\cC\nor)$, where $\cC\nor\subseteq\cN_k\times\QQ$ (normalized as in \S\ref{subsection:Bda}) is
  the cone spanned by $\cC\times\{0\}$ and $(\aq,1)$, where
  $\aq:=\sum_ia_i$.

  Indeed, using the Local Structure Theorem, the calculation is
  reduced to the case of toric varieties and, therefore, to the case
  of $Y=\{(0,0)\}\subseteq\A_k^1\times\A_k^s=\A_k^{s+1}$. An explicit
  computation shows that the blow-up of $\A_k^{s+1}$ at the origin is
  described by the fan whose maximal cones are spanned by
  $\{e_0,\ldots,\widehat{e_i},\ldots,e_s\}\cup\{\eq\}$,
  $i=0,\ldots,s$, where the $e_i$ are the canonical basis vectors of
  $\ZZ^{s+1}$ and $\eq:=\sum_ie_i$. The exceptional divisor
  corresponds to the ray $\QQ_{\ge0}\eq$. The open orbit inside the
  strict transform of $\{0\}\times\A_k^s$ corresponds to the ray
  $\QQ_{\ge0}e_0$. Thus, its closure is given by all cones containing
  $e_0$. Removing these leaves the cone spanned by
  $e_1,\ldots,e_s,\eq$. The claim is proved.

  One now proves \ref{it:Flat1} and \ref{it:Flat2} in the same way as
  \cref{lemma:Zqd} and \cref{lemma:smoothhomogeneous},
  respectively. The last two assertions follow from the fact that
  $H_0$ leaves all hyperplanes $N_{Y,x_0}D_i\subset N_{Y,x_0}\Xq$
  invariant and, therefore, factors through the group of diagonal
  matrices.
\end{proof}

\subsection{\it Boundary degenerations in arbitrary directions}\

In this subsection, we extend the degenerations of $X$ over
$\G_m\subset\A_k^1$ and $A_k/\fA_k^0\subseteq S\red$ to a still larger
base. Among other things, this allows the treatment of
degenerations along arbitrary curves inside $X$.

Note that $X/\fA_k$ is also a $k$-spherical $G$-variety with the
strictly convex valuation cone $\hat\cZ_k=\cZ_k/\cN_k^0$. Thus, it has
a standard embedding, which we denote by $S\2$.

Now define $X\2:=(X\times X)/\fA_k$. This is a homogeneous
$k$-spherical variety for the group $G\2:=G\times G$. The projection to
the second factor yields a morphism $p:X\2\to X/\fA_k$. The valuation
cone of $X\2$ is $\cZ_k\2:=(\cZ_k\times\cZ_k)/\|diag|\cN_k^0$. It
contains a diagonal copy $\|diag|\hat\cZ_k$ of $\hat\cZ_k$. Let
$\Xq\2:=X\2(\|diag|\hat\cZ_k)$ be the corresponding embedding of
$X\2$.

\begin{enumerate}

\item \emph{$S\2$ is the smooth standard embedding of a $k$-spherical
    $G$-variety. In particular, its orbits $S\2_I$ are parameterized
    by subsets $I$ of $\Sigma_k$ such that
    $\Sigma_k(S\2_I)=I$. Moreover, each orbit closure is smooth.} This
  follows by construction and \cref{thm:Autoroot}.

\item \emph{$\Xq\2$ is a smooth $k$-spherical $G\2$-variety.} Again,
  this follows by construction and from the fact that
  $A_k(X\2)=(A_k\times A_k)/\fA_k$ contains a diagonal copy of
  $A\2_k=A_k/\fA_k$ and that $\cZ_k\2$ is generated by the dual basis
  of $\Sigma_k\aut$ inside $\cN(A_k\2)$.

\item \emph{The morphism $p$ extends to a smooth morphism
    $\pq:\Xq\2\to S\2$, yielding the Cartesian diagram
    \[
      \cxymatrix{X\times^{\fA_k}X\ar@{=}[r]\ar[dr]&X\2\ar[d]^<<<p\ar@{^(->}[r]&\Xq\2\ar[d]^<<<\pq\\
        &X/\fA_k\ar@{^(->}[r]&S\2}
    \]} This is shown as in the case of $S\red$.

\item \emph{The $G\2$-orbits of $\Xq\2$ are the preimages of the
    $G$-orbits in $S\2$.} This follows from the fact that the fan is
  supported on a diagonal copy of $\hat\cZ_k$.
  
\item \emph{The diagonal morphism $X/\fA_k\to(X\times X)/\fA_k$
    extends to a section $\delta:S\2\to\Xq\2$.}
  
\item \emph{Let $a\in S_I\2(k)$. Then the fiber $X_a$ of $\pq$ over
    $a$ is a homogeneous $k$-spherical $G$-variety with
    $Q_k(X_a)=Q_k(X)$, $(X_a)\el=X\el$, and $\Sigma_k(X_a)=I$.} This
  follows in the same way as above, using the isomorphism
  $(A_k\times A_k)/\fA_k\overset\sim\to A_k\times
  (A_k/\fA_k):[a,b]\mapsto (ab^{-1},[b])$.

\item \emph{There is a Cartesian diagram
    \[
      \cxymatrix{\Xq\red\ar[d]\ar[r]&\Xq\2\ar[d]\\
        S\red\ar[r]^f&S\2}
    \]
    where $f$ is induced by a flat $A_k\into X$.} See the proof of
  \eqref{eq:ZXq1}.
    
\end{enumerate}

\begin{sremark}\label{rem:boundarydeg}
  In contrast to $A_k/\fA_k^0$, we took the quotient $X/\fA_k$ by the
  full group of central automorphisms of $X$. This has the advantage
  that $\pq$ is a smooth morphism between smooth varieties and that
  all orbit closures are smooth as well. The drawback is that two
  fibers $X_*$ of $\pq$ may not be isomorphic, even if the base points
  lie in the same $G$-orbit $S_I\2$. This is because $G(k)$ does not,
  in general, act transitively on $S_I\2(k)$. On the other hand, since
  all isotropy groups in a torus embedding are connected, the
  $A_k$-orbits in $S\red$ are of the form $A_k/A'$, where $A'$ is
  connected. Thus, one can find a complement $A_k=A'\times A''$ with
  $A''$ acting simply transitively. This implies that $A_k(k)$ acts
  transitively on $S_I\red(k)$. Thus, every fiber of $\Xq\red$ over a
  $k$-rational point is isomorphic to one of the standard fibers
  $X_I$. A final advantage of $\fA_k^0$ over $\fA_k$ is that it is
  more readily computable, since it depends only on the little Weyl
  group of $X$, as opposed to its root system.
\end{sremark}

There are two applications of the big boundary degeneration.

Let $F=A_kx\subseteq X/\fA_k$ be a flat with closure
$\FQ\subseteq S\2$. By \cref{prop:flat}, we know that $\FQ$ is a
smooth, projective, toroidal variety. Now consider the pullback
$\Xq\2\times_{S\2}\FQ\to\FQ$. This is a smooth deformation of $X$ that
is defined over all of $\FQ$ and that coincides with $\Xq\red$ after
the quasi-finite pullback $S\red\to\FQ$. In other words, we have
constructed a deformation of $X$ in all directions of $\cN_k(X)$ and
not only those in $\cZ_k$.

The second application concerns the \emph{Demazure morphism}. For
$X=G/H$, this is the morphism $\Phi_X$ from $X$ to the Graßmannian
$\|Gr|(\fg)$ that assigns to every point $x=gH\in X$ its isotropy
subalgebra $\fg_x=\|Ad|(g)\fh$. Let
$\Xq^{\text{Dem}}\subseteq\|Gr|(\fg)$ be the closure of the image
$\Phi_X(X)$. In the case $k=K$, Losev \cite{Losev} has shown that
$\Xq^{\text{Dem}}$ coincides with the wonderful embedding of
$X/\fA_k$. In general, we can show the following:

\begin{stheorem}
  Let $X=G/H$ be a homogeneous $k$-spherical $G$-variety, and let $\cF$
  be a fan supported in $\cZ_k(X)$. Then the Demazure morphism
  $\Phi_X:X\to\|Gr|\fg$ extends to a morphism
  $\Phi_{X(\cF)}:X(\cF)\to\Xq^{\text{Dem}}\subseteq\|Gr|(\fg)$. Moreover,
  this morphism factors through $(X/\fA_k)\st$ and satisfies
  $\Phi_{X(\cF)}(x)\subseteq\fg_x$ for all $x\in X(\cF)$.
\end{stheorem}

\begin{proof}
  The fibers of the smooth morphism $\pq:\Xq\2\to S\2$ are precisely
  the $G$-orbits, all of which have the same dimension. Thus, the
  isotropy subalgebras $\fg_x$ with $x\in\Xq\2$ form a subvector
  bundle $\fg_*$ of the trivial vector bundle
  $\Xq\2\times\fg\to\Xq\2$. This defines a morphism
  $\Xq\2\to\|Gr|(\fg)$. Composing with the section
  $\delta:S\2\to\Xq\2$, we obtain a morphism
  $\overline\Phi_X:S\2\to\|Gr|(\fg)$ with
  $\overline\Phi_X([x])=\fg_{\delta([x])}=\fg_x$ for all $x\in X$, where
  $[x]=x\fA_k\in X/\fA_k$. Thus, for every embedding $X\into\Xq$ that
  factors through $S\2$, there is an extension of $\Phi_X$ to
  $\Xq$. This holds, in particular, for $X(\cF)$.
\end{proof}

\numberwithin{equation}{section}\swapnumbers

\section{The weak polar decomposition}
\label{sec:polar}

We present an application of our theory to local fields.

Let $X$ be a homogeneous $k$-spherical variety, let $Q_k(X)=LU$ be a
Levi decomposition, and let $A_k=A_k(X)$. The LST yields a
$Q_k(X)$-equivariant embedding $U\times X\el\into X$ and an
$L$-equivariant principal $A_k$-bundle $\pi:X\el\to X\an$. The action
of $L$ on $X\an$ is anisotropic and transitive. We say that a subset
$E$ of $X\el(k)$ is \emph{representing} if $X\an(k)=L(k)\pi(E)$. Since
the map $X\el(k)\to X\an(k)$ is surjective by Hilberts Satz 90,
(applied to the split torus $A_k$), representing sets exist. If $k$ is
a local field, the orbit space $X\an(k)/M(k)$ is finite by
\cite{BorelSerre}. Thus, there exist finite representing sets.

Let $k$ be a local field with absolute value $|\cdot|$. We write
$\sA_k:=A_k(k)$ for the group of $k$-rational points of
$A_k=A_k(X)$ and define
\[
  \sA_k^-:=\{x\in \sA_k\mid|\sigma(x)|\le1\text{ for all
  }\sigma\in\Sigma_k(X)\}
\]
to be the \emph{compression domain} in $\sA_k$.

\begin{theorem}\label{thm:weak2}
  Let $k$ be a local field, let $X$ be a homogeneous $k$-spherical
  variety, and let $E\subseteq X\el(k)$ be representing. Then there
  exists a compact subset $\Omega\subseteq G(k)$ such that
  \[
    X(k)=\Omega\cdot \sA_k^-E.
  \]
\end{theorem}

\begin{proof}
  Let $\cF$ be a fan whose support is $\cZ_k(X)$. We also choose $\cF$
  to be smooth, i.e., each cone in $\cF$ is generated by a subset of
  an integral basis. Let $\Xq:=X(\cF)$, $\Xq\el:=X\el(\cF)$, and
  $\Aq_k:=A_k(\cF)$. All of these varieties are smooth.

  Let $\cC\in\cF$ be a cone of maximal dimension. Then the orbit
  $A_k(\cC)\subseteq \Aq_k$ consists of a single fixed point
  $a_\cC$. Let $\cC$ be defined by the inequalities
  $\beta_1\le0,\ldots,\beta_r\le0$. Assuming that the $\beta_i$ are
  primitive in $\Xi_k(X)$, they form an integral basis (by
  smoothness). Hence, $a_\cC$ has an $A_k$-stable (Zariski)
  neighborhood that is isomorphic to the affine space $\A_k^r$, where
  $A_k$ acts via the characters $\beta_i$. Thus, if we set
  \[
    \sA_k(\cC)^-:=\{x\in A_k(k)\mid|\beta_i(x)|\le1\text{ for all
    }i=1,\ldots,r\},
  \]
  then its closure $\overline{\sA_k(\cC)^-}$ in $\sAq_k:=\Aq_k(k)$ is
  isomorphic to $\sD^r\subseteq k^r=\A_k^r(k)$, where $\sD\subseteq k$
  is the closed unit disk. This implies that $\overline{\sA_k(\cC)^-}$
  is a closed neighborhood of $a_\cC$ in $\sAq_k$.

  Since the characters $\beta_i$ form an integral basis of $\Xi_k(X)$,
  one can express the spherical roots in terms of the $\beta_i$:
  \[
    \sigma=\sum_{i=1}^r\nu_i(\sigma)\beta_i\quad\text{ with
    }\nu_i(\sigma)\in\ZZ.
  \]
  From the inclusion $\cC\subseteq\cZ_k(X)$, we obtain
  $\nu_i(\sigma)\ge0$, and therefore $\sA_k(\cC)^-\subseteq
  \sA_k^-$. Thus, we have proved that the closure $\overline{\sA_k^-}$ of
  $\sA_k^-$ in $\sAq_k$ is a neighborhood of every fixed point
  $a_\cC\in\sAq_k$.

  Note that the $A_k$-bundle $X\el\to X\an$ extends to an
  $\Aq_k$-bundle $\Xq\el\to X\an$ such that every closed orbit is
  mapped isomorphically onto $X\an$. Thus, any $x\in X\el(k)$ defines
  an embedding $\Aq_k\into\Xq\el$ such that $a_\cC$ corresponds to a
  point $x_\cC\in\Xq\el$. Note that $M\subseteq L$ also acts
  transitively on $X\an$, so $X\an(k)$ consists of finitely many
  $\sM$-orbits where $\sM:=M(k)$. Since $\Xq\el\to
  X\an$ is locally trivial, this implies that
  $\sM \cdot\overline{\sA_k^-x}$ is a neighborhood of the orbit
  $\sM x_\cC$ in
  $\Xq\el(k)$. Thus, there exists a finite subset $F_\cC\subseteq
  L(k)$ such that $F_\cC
  \sM \cdot\overline{\sA_k^-x}$ is a neighborhood of
  $L(k)x_\cC$ in $\Xq\el(k)$. Let $\overline{\sA_k^-E}:=\bigcup_{x\in
    E}\overline{\sA_k^-x}$ and $F:=\cup_\cC F_\cC$. Then, since
  $E$ is representing,
  $F\sM \cdot\overline{\sA_k^-E}$ is a neighborhood of
  $X\el(\cC)(k)$ inside $\Xq\el(k)$ for every maximal
  $\cC\in\cF$. Let $\sU\subseteq
  U(k)$ be a compact neighborhood of
  $1$. Then the Local Structure Theorem implies that $\Omega_0:=\sU
  F\sM \cdot\overline{\sA_k^-E}$ is a compact neighborhood of
  $X\el(\cC)(k)$ inside $(U\times\Xq\el)(k)$ for every
  $\cC$. Since $U\times\Xq\el$ is open in
  $\Xq$, we see that the interior $\Omega_0^0$ of
  $\Omega_0$ is also an open neighborhood of
  $X\el(\cC)(k)$ inside $\Xq(k)$.

  Next, we show that for every $x\in\Xq(k)$ there exists $g\in G(k)$
  such that $gx$ lies in $\Omega_0^0$. This clearly implies the
  theorem: by the compactness of $\Xq(k)$ (see \cref{cor:compact}), it
  can be covered by finitely many left translates
  $g^{-1}\Omega_0^0$. Thus, there exists a finite set
  $F'\subseteq G(k)$ such that $\Xq(k)=\Omega\cdot\overline{\sA_k^-E}$, with
  $\Omega=F'\sU F\sM \subseteq G(k)$ compact. Restricting
  to $X(k)$ yields the assertion.

  Since $U\times\Xq\el$ is Zariski open in $\Xq$ and meets every
  $G$-orbit, there exists $g\in G(k)$ such that $gx\in
  U\times\Xq\el$. Thus, there exists $u\in U(k)$ such that
  $ugx\in\Xq\el(k)$. Then one can find $m\in F\sM $ and $x'\in E$ such
  that $mugx\in\overline{\sA_k^-x'}$. Let $\phi$ denote the embedding
  $\Aq_k\into\Xq\el$ corresponding to $x'$, and choose $a\in\Aq_k(k)$
  with $\phi(a)=mugx$. Since every $A_k$-orbit of $\Aq_k$ contains a
  fixed point $a_\cC$ in its closure, there exists a homomorphism
  $\lambda:\G_m\to A_k$ such that $\lim_{t\to\infty}\lambda(t)a=a_\cC$
  (even in the Hausdorff topology). A positive power $\lambda^n$ of
  $\lambda$ admits a lift to a homomorphism
  $\widetilde\lambda:\G_m\to A$. Since
  $\lim_{t\to\infty}\widetilde\lambda(t)\phi(a)=
  \phi(\lim_{t\to\infty}\lambda(t)^na)=x_\cC$, there exists $t\in\G_m(k)$
  such that $\widetilde\lambda(t)mugx\in\Omega_0^0$.
\end{proof}

The following lemma is well known.

\begin{lemma}\label{prop:transitive}
  For $k=\RR$, let $X$ be a homogeneous anisotropic $G$-variety. Then
  $G(\RR)$ acts transitively on $X(\RR)$.
\end{lemma}

\begin{proof} Equation \eqref{eq:H=PHan} implies that $M\subseteq G$
  acts transitively on $X$. Thus, $X$ is an affine variety by
  \cref{prop:affineorbit}. Moreover, since $M$ is $\RR$-anisotropic,
  $M(\RR)$ is a compact Lie group and $X(\RR)$ is compact. The
  Stone--Weierstraß theorem shows that the space $\CC[X]$ of regular
  functions is dense in the space $\cC(X(\RR))$ of continuous
  functions; hence $\CC[X]^M=\CC[X]^{M(\RR)}$ is dense in
  $\cC(X(\RR))^{M(\RR)}$. Therefore
  $\cC(X(\RR))^{M(\RR)}=\CC[X]^M=\CC$, i.e., every $M(\RR)$-invariant
  continuous function on $X(\RR)$ is constant. This implies that
  $M(\RR)$ acts transitively on $X(\RR)$ and proves the assertion.
\end{proof}

It follows that for $k=\RR$ we can choose $E$ to consist of a single
element. The same holds trivially if $G$ is split, since $M=1$ in that
case. Thus:

\begin{corollary}\label{cor:weak2}
  
  Let $k=\RR$, or let $k$ be a local field with $G$ split. Let $X$ be
  a homogeneous $k$-spherical variety. Then for every $x\in X\el(k)$
  there exists a compact subset $\Omega\subseteq G(k)$ such that
  $X(k)=\Omega\cdot \sA_k^-x$.
  
\end{corollary}

\begin{remarks}

  \emph{i)} For $k=\RR$, the corollary was proved in \cite{KKSS} in
  the more refined version that $\Omega=F\sK$, where
  $F\subseteq G(\RR)$ is a finite subset and $\sK\subseteq G(\RR)$ is
  a maximal compact subgroup (see the proof of \emph{loc.\ cit.},
  Prop.\ 5.12).

  \emph{ii)} For $k=\RR$, one has $\sA_\RR^-=\sA[2]\exp\fa_\RR^-$, where
  $\sA[2]\subseteq\sA_\RR$ is the subgroup of $2$-torsion elements and
  $\fa_\RR^-$ is the \emph{compression cone}
  \[
    \fa_\RR^-=\{\xi\in\fa_\RR\mid{\rm d}\sigma(\xi)\le0\text{ for all
    }\sigma\in\Sigma_\RR\}.
  \]
  Since $\fa_\RR=\cN_\RR\otimes_\QQ\RR$, the compression cone is
  nothing but the closure of the valuation cone $\cZ_\RR$ in
  $\fa_\RR$. Combining this with \emph{i)}, we obtain
  $X(\RR)=F\sK\cdot\sA[2]\exp\fa^-_\RR x$.

\end{remarks}

Let $X$ be a homogeneous $k$-spherical $G$-variety. Then $Q_k(X)$ has
an open orbit $X_0\subseteq X$. The next lemma shows that the
$Q_k(X)$-isotropy groups of $k$-rational points in $X_0$ are
conjugate, even though $G(k)$ does not act transitively on $X(k)$.

\begin{lemma}\label{lemma:conjugation}

  Assume that $k=\RR$ or that $G$ is split (and $k$ is arbitrary). Let
  $X$ be a homogeneous $k$-spherical $G$-variety. For $i\in\{1,2\}$,
  let $x_i\in X(k)$ be a point with isotropy group $H_i:=G_{x_i}$. Let
  $Q_1, Q_2\subseteq G$ be $G(k)$-conjugates of $Q_k(X)$ such that
  $x_i$ lies in the open $Q_i$-orbit. Then there exist Levi
  decompositions $Q_i=L_iU_i$ and an element $g\in G(K)$ such that
  conjugation by $g$ maps $H_1$, $Q_1$, $L_1$, and $H_1\cap Q_1$ to
  $H_2$, $Q_2$, $L_2$, and $H_2\cap Q_2$, respectively, and such that
  the diagram
  \[\label{eq:HQL}
    \cxymatrix{H_1\cap Q_1\ar@{^(->}[r]\ar[d]^\sim&L_1\ar[d]^\sim\\
      H_2\cap Q_2\ar@{^(->}[r]&L_2}
  \]
  is defined over $k$.

\end{lemma}

\begin{proof}

  The two parabolics $Q_1$ and $Q_2$ are conjugate by an element
  $t\in G(k)$. After replacing $x_2$ with $tx_2$, we may assume that
  $Q_1=Q_2=:Q=LU$. Let $X_0\subseteq X$ be the open $Q$-orbit. Then
  $x_1,x_2\in X_0(k)$. By the Generic Structure Theorem,
  $X_0\cong U\times X\el$. After replacing $x_i$ with $u_ix_i$ for
  $u_i\in U(k)$, we may assume that $x_i\in X\el(k)$. Now consider the
  quotient $\pi:X\el\to X\an$. There exists $l\in L(k)$ such that
  $\pi(x_2)=l\pi(x_1)=\pi(lx_1)$ (by \cref{prop:transitive} in the
  case where $k=\RR$). Thus, by replacing $x_1$ with $lx_1$, we may
  assume that there exists $a\in A_k(k)$ such that $x_2=ax_1$. Let
  $g\in A(K)$ be a lift of $a$, and let $\phi$ be conjugation by
  $g$. Then $\phi$ maps $H_1$ to $H_2$ and preserves $Q$ and
  $L$. Since $H_i\cap Q=H_i\cap L$ and $g$ centralizes $L$, the
  diagram \eqref{eq:HQL} is defined over $k$.
\end{proof}

We use this lemma to show that, under certain circumstances, at least
the Lie algebra of $H\cap Q$ depends only on the $G(K)$-conjugacy
class of $H$.

\begin{corollary}

  Assume that $k=\RR$ or that $G$ is split. Let
  $\fh_1,\fh_2\subseteq\fg$ be two self-normalizing $k$-spherical
  subalgebras, and let $Q_1,Q_2\subseteq G$ be two $k$-parabolic
  subgroups conjugate to $Q_k(X)$ with $\fh_i+\Fq_i=\fg$. Assume that
  there exists an \emph{inner} $K$-automorphism $\phi$ of $G$ mapping
  $\fh_1\otimes_kK$ to $\fh_2\otimes_kK$. Then
  $\fh_1\cap\Fq_1\cong\fh_2\cap\Fq_2$.

\end{corollary}

\begin{proof}

  Let $N_i$ be the normalizer of $\fh_i$ in $G$. Since $\fh_i$ is
  self-normalizing, it follows that $\Lie N_i=\fh_i$. This, in turn,
  implies that $N_i$ is self-normalizing in $G$. For the homogeneous
  space $X_i=G/N_i$, this means that the group of $G$-automorphisms is
  trivial. The existence of $\phi$ means that $X_1$ and $X_2$ are
  isomorphic over $K$, i.e., $X_2$ is a $k$-form of $X_1$. Such forms
  are classified by $H^1(k,\Aut^G(X_1))$, which vanishes; hence, $X_1$
  and $X_2$ are isomorphic even over $k$. This means that $X_1(k)$
  contains two points $x_1,x_2$ whose isotropy Lie algebras are
  $\fh_1$ and $\fh_2$, respectively. Now the assertion follows from
  \cref{lemma:conjugation}.
\end{proof}

We end this section with some statements specific to $k=\RR$. To this
end, we introduce the notation $\sG:=G(\RR)^0$, $\sP:=P(\RR)^0$, and
$\sX:=X(\RR)$.

\begin{theorem}

  Let $X=G/H$ be a homogeneous $\RR$-spherical variety. Assume that
  $X$ is $\RR$-adjoint, and let $X\into X\st$ be its smooth standard
  embedding. Let $Y\subseteq X\st$ be the closed $G$-orbit. Then

  \begin{enumerate}

  \item\label{it:Rwonderful2} $\sX\st:=X\st(\RR)$ is a compact,
    connected manifold.

  \item\label{it:Rwonderful3} $\sY:=Y(\RR)$ is the only closed
    $\sG$-orbit of $\sX\st$. In particular, it is connected and
    $G(\RR)$-stable.

  \item\label{it:Rwonderful1} $\sP$ has at most $2^{\rk_\RR X}$ open
    orbits in $\sX$ (or, equivalently, in $\sX\st$). They all contain
    $\sY$ in their closure.

  \end{enumerate}

\end{theorem}

\begin{proof}
  The orbit $Y$ is of rank zero by \cref{prop:boundary}. It follows
  from \cref{thm:rank0} that $Y$ is of the form $Y=G/H$ with
  $Q\an\subseteq H\subseteq Q=LU$. Thus, we get a map
  $\sY\to (G/Q)(\RR)$ with fiber $(Q/H)(\RR)$. It is known that
  $(G/Q)(\RR)$ is connected (e.g., by the Bruhat decomposition); hence
  $\sG$ acts transitively on it. The action of $Q$ on $Q/H$ is
  anisotropic. Hence, the connected group $M(\RR)\subseteq\sG$ acts
  transitively on $(Q/H)(\RR)$. It follows that $\sY$ is connected and
  the $\sG$-action on it is transitive.
  
  \cref{prop:boundary} yields an open embedding
  $U\times\Xq\el\to X\st$, where $\Xq\el$ is an $\Aq_\RR$-fibration
  over $X\an$. Let $r:=\rk_\RR X$. Since $X\st$ is smooth (see
  \cref{prop:smooth}), we have $\Aq_\RR\cong\A_\RR^r$ with
  $A_\RR\cong\G_m^r$ acting on it. Since $M(\RR)\subseteq\sP$ acts
  transitively on $X\an(\RR)$, we see that there are at most as many
  open $\sP$-orbits in $X(\RR)$ as there are open
  $(\G_m(\RR)^r)^0=\RR_{>0}^r$-orbits in $\RR^r$. All of these orbits
  contain $0\in\RR^r$ in their closure, which proves
  \ref{it:Rwonderful1}.

  Now \ref{it:Rwonderful2} follows easily: $\sX\st$ is compact because
  of \cref{cor:compact}. It is a manifold because $X\st$ is
  smooth. Finally, it is connected, since every connected component
  would contain an open $\sP$-orbit, and all of them contain $\sY$ in
  their closure.

  Regarding \ref{it:Rwonderful3}, we already know that $\sY$ is a
  closed $\sG$-orbit. Conversely, let $\sZ=\sG x$ be a closed orbit in
  $\sX\st$. The (Zariski) closure $Z$ of $Gx$ in $X\st$ is also a
  standard embedding containing $Y$ as its only closed orbit. After
  replacing $X$ with $Gx$, we may assume that $\sZ$ is contained in
  $X(\RR)$. Then $\sZ$ is open and contains an open $\sP$-orbit. The
  latter contains $\sY$ in its closure. Hence $\sZ=\sY$.
\end{proof}

\begin{remark}
  For split groups $G$, the theorem is due to Akhiezer and
  Cupit-Foutou \cite{AkhiezerCupitFoutou}.
\end{remark}

Let $\sM$ be a $\sK$-manifold, where $\sK$ is a Lie group. Recall that
a subgroup $\sH\subseteq\sK$ is called a \emph{principal
  $\sK$-isotropy subgroup of $\sM$} if the isotropy group $\sK_x$ is
conjugate to $\sH$ for $x$ in an open and \emph{dense} subset $\sM_0$
of $\sM$. The largest such subset $\sM_0$ is denoted by
$\sM^{\rm princ}_\sK$.

Similarly, a subalgebra $\fh\subseteq\fk:=\Lie\sK$ is a
\emph{principal $\fk$-isotropy subalgebra of $\sM$} if the isotropy
subalgebra $\fk_x$ is $\Ad(\sK)$-conjugate to $\fh$ for $x$ in an open
and dense subset $\sM_0$ of $\sM$. The largest such subset $\sM_0$ is
denoted by $\sM^{\rm reg}_\sK$. Clearly, the principal subalgebra is
simply the Lie algebra of the principal subgroup whenever the latter
exists; in that case,
$\sM^{\rm princ}_\sK\subseteq \sM^{\rm reg}_\sK$. It is well known
that $\sM$ has a principal isotropy subgroup or subalgebra whenever
$\sK$ is compact and $\sM$ is connected; see \cite{DK} for details.

\begin{corollary}\label{cor:pricreg}
  
  Let $X=G/H$ be a homogeneous $\RR$-spherical variety, and let
  $X\into\Xq$ be a smooth toroidal embedding. Let
  $\sK\subseteq G(\RR)$ be a compact subgroup. Then all $\sG$-orbits
  of $\sXq:=\Xq(\RR)$ intersect $\sXq^{\rm reg}_\sK$
  non-trivially. Moreover, $\sX^{\rm princ}_\sK$ exists and meets all
  $\sG$-orbits of $\sX$.

\end{corollary}

\begin{proof}

  The principal locus $\sXq^{\rm princ}_\sK$ exists because $\sXq$ is
  a connected manifold. Since it is dense, we find that
  $\sX^{\rm princ}_\sK=\sXq^{\rm princ}_\sK\cap\sX$ exists as
  well. The latter meets all $\sG$-orbits of $\sX$ since they are all
  open in $\sX$. This proves the second assertion.

  The first assertion is proved by induction on $\dim X$. Let
  $\sO\subseteq\sXq$ be a $\sG$-orbit. It is open in $\sY=Y(\RR)$,
  where $Y\subseteq\Xq$ is a $G$-orbit. If $Y=X$, then $\sO$ is open
  in $\sXq$ and therefore meets the regular locus. Otherwise, there
  exists a $G$-orbit $Z\subseteq\Xq$ of codimension one such that
  $Y\subseteq \Zq\subseteq\Xq$, where $\Zq$ denotes the closure of
  $Z$. It is well known that in a smooth toroidal embedding, all orbit
  closures are smooth. Thus, $\Zq$ is smooth. Now recall from
  (\Cite{DK}{Thm.\ 2.8.5}) that the complement of $\sXq^{\rm reg}_\sK$
  has codimension at least $2$. Therefore,
  $\sXq^{\rm reg}_\sK\cap\Zq(\RR)$ is non-empty and, hence, coincides
  with $\Zq(\RR)^{\rm reg}_\sK$. The assertion follows by induction.
\end{proof}

\begin{example}
  Consider $G=SL(2,\RR)$ acting transitively on
  $X=\A_\RR^2\setminus\{(0,0)\}$. Then one obtains a smooth toroidal
  embedding $\Xq$ of $X$ by blowing up $\A_\RR^2$ at the origin. The
  $G$-orbits of $\Xq$ are $X$ and the exceptional divisor
  $E\subset\Xq$. Now, $\sK=SO(2)$ is a maximal compact subgroup of
  $\sG$. Then $\sK$ acts freely on $X$, while $-\mathbf1_2\in\sK$ acts
  trivially on $E$. This shows that $\sXq^{\rm princ}_\sK$ does not
  meet the $\sG$-orbit $E(\RR)$. See \cref{prop:princ} for how to
  remedy this situation.
\end{example}

\section{The Satake compactification}\label{sec:Satake}

In this final section, we compare our compactifications with those
already appearing in the literature.

Suppose that the $k$-dense variety $X$ is also $K$-semisimple. One can
obtain another embedding by regarding $X$ as a $K$-variety, taking its
standard embedding, and observing that it is defined over $k$. The
resulting variety, denoted by $X\st^K$, is another canonical embedding
of $X$. The advantages of $X\st^K$ are its independence from $k$ and
the fact that it is well-studied. For symmetric varieties, it was
introduced by De Concini and Procesi in \cite{DeConciniProcesi} and,
more generally, for $K$-spherical varieties by Brion and Pauer
\cite{BrionPauer} (based on the embedding theory of Luna and Vust
\cite{LunaVust}). If $k=\RR$ and $X\st^\CC$ is smooth, then the set
$X\st^\CC(\RR)$ of real points is clearly a compact $G(\RR)$-manifold
containing $\sX$ as a dense open subset.

A major disadvantage of $X\st^K$ is that it does not behave well when
$X$ is not absolutely spherical; in that case, $X\st^K$ may contain
infinitely many orbits and is therefore defined only up to birational
isomorphism. Moreover, its boundary is difficult to control. In
particular, the union of all $k$-dense orbits need not be open (see
Example \ref{it:compare3} {\it iii)} below).

The relation between the two standard embeddings is as follows:

\begin{proposition}

  Let $X$ be a $K$-semisimple $k$-dense $G$-variety. Then its
  $K$-standard model $X\st^K$ can be chosen to be defined over
  $k$. Moreover, $X$ is $k$-semisimple, and a $k$-standard model
  $X\st$ can be chosen such that the identity on $X$ extends to a
  birational morphism
  \[
    \iota_X:X\st\to X\st^K.
  \]
  Then the following are equivalent:

  \begin{enumerate}[noitemsep]

  \item\label{it:Kwond0} $\cZ_k$ is a face of $\cZ_K$.

  \item\label{it:Kwond1} $\iota_X$ is injective on $k$-rational
    points.

  \item\label{it:Kwond2} $\iota_X$ is an open embedding.

  \item\label{it:Kwond3} $\cG^*$ acts trivially on
    $\Sigma_K(X)\setminus\Sigma_K^0(X)$ (see \eqref{eq:Sigma0}).

  \end{enumerate}

  If these conditions hold and $X$ is $k$-spherical, then $\iota_X$ is
  bijective on $k$-rational points.

\end{proposition}

\begin{proof}

  The $k$-standard embedding corresponds to the cone $\cZ_k$. An
  analogous statement holds for $K$. Thus, the existence of $\iota_X$
  follows from the inclusion $\cZ_k\subseteq\cZ_K$.

  The space $\Xi_K(X)\otimes_\ZZ\QQ$ has $\Sigma_K(X)$ as a basis
  because $X$ is $K$-semisimple. Moreover, $\cG^*$ permutes this
  basis. Now \eqref{eq:diagXi22}, together with \cref{cor:compN},
  implies that $\cZ_k$ is obtained from $\cZ_K$ by imposing the
  equations $\sigma=\tau$ whenever
  $\sigma,\tau\in\Sigma_K\setminus\Sigma_K^0$ lie in the same
  $\cG^*$-orbit. This shows the equivalence
  \ref{it:Kwond0}$\Leftrightarrow$\ref{it:Kwond3}.

  Now assume that $\cZ_k$ is a face of $\cZ_K$. Then $\cC=\cZ_k$
  corresponds to a stratum $X^K(\cC)$ of $X\st^K$. Every such stratum
  is closed in the open chart $X_0$, which corresponds to the fan of
  faces of $\cC$. This shows that $X\st=X_0\subseteq X\st^K$ is open,
  which proves \ref{it:Kwond0}$\Rightarrow$\ref{it:Kwond2}.

  Since the implication \ref{it:Kwond2}$\Rightarrow$\ref{it:Kwond1} is
  trivial, it remains to prove
  \ref{it:Kwond1}$\Rightarrow$\ref{it:Kwond0}. To this end, suppose
  that $\cZ_k$ is not a face of $\cZ_K$. Then there exists an extremal
  ray $\cR$ of $\cZ_k$ which is not an extremal ray of $\cZ_K$. Let
  $\cC$ be the smallest face of $\cZ_K$ containing $\cR$. Then
  $\iota_X$ maps the stratum $X(\cR)$ of $X\st$ onto the stratum
  $X^K(\cC)$ of $X\st^K$. The fibers through $k$-rational points are
  tori of dimension $\dim\cC-\dim\cR>0$. This shows that $\iota_X$ is
  not injective on $k$-rational points.

  Finally, if $\iota_X$ is an open embedding and $X$ is $k$-spherical,
  then \cref{cor:rational} implies that $X\st(k)=X\st^K(k)$.
\end{proof}

\begin{examples}\label{ex:compare}

  {\it i)} Condition \ref{it:Kwond3} above is satisfied if $G$ is an
  inner form (e.g., split), since the $\cG^*$-action on $S$ is then
  trivial. It even suffices for $\cG^*$ to act trivially on
  $S\setminus S^0$, which, in the real case, means that the Satake
  diagram of $G$ does not contain any arrows (e.g., all anisotropic
  groups, or $SO(p,q)$ with $k=\RR$ and $|p-q|\ne2$).

  {\it ii)} Even if $\cG^*$ acts non-trivially on $S\setminus S^0$,
  its action on $\Sigma_K\setminus\Sigma_K^0$ may be trivial. This
  happens, for example, for all Riemannian symmetric spaces (see
  below), or for the spaces
  $U(p_1+p_2,q_1+q_2)/U(p_1,q_1)\times U(p_2,q_2)$ with $k=\RR$.

  \phantomsection\label{it:compare3}{\it iii)} The simplest example
  where the two standard embeddings differ is when $k=\RR$ and
  $X:=SU(2,1)/\mu_3SO(2,1)$, where $\mu_3$ is the center of $G$
  (included to ensure that $X\st^\CC$ is smooth). Then $X(\CC)$ is the
  space of smooth quadrics in $\P^2(\CC)$, and $X\st^\CC$ is called
  the space of complete quadrics. It is smooth and consists of four
  orbits: the open orbit $X$, two orbits $D_1$ and $D_2$ of codimension
  $1$, and a closed orbit $Y$. The group $SU(2,1)$ is quasi-split and
  $\cG^*$ swaps the two simple roots of $S$. Thus, the two orbits
  $D_1$ and $D_2$ are swapped by conjugation; in particular, they
  contain no real points. For the real points of the entire space,
  this implies $X\st^\CC(\RR)=\sX\cup \sY$. Thus, all $G(\RR)$-orbits
  are either open or closed, and the closed orbit $\sY$ has real
  codimension $2$. On the other hand, the closed orbit of $\sX\st$ has
  real codimension $1$. This shows that
  \[
    \iota_X:\sX\st\to X\st^\CC(\RR)
  \]
  is surjective but not injective. In fact, $X\st$ is the blow-up of
  $X\st^\CC$ along $Y$, with the proper transforms of $D_1$ and $D_2$
  removed. This example also shows that the union of the two $k$-dense
  orbits, namely $X\cup Y$, is not open but only constructible.

\end{examples}

\begin{remark}

  In \cite{BorelJi}, Borel and Ji describe various constructions for
  compactifying symmetric and locally symmetric spaces. In particular,
  \S II.10 deals with the compactification of semisimple (also known
  as non-Riemannian or pseudo-Riemannian) symmetric spaces, i.e., the
  real points of $X=G/H$, where $G$ is real semisimple and $H$ is the
  fixed-point group of an involution, which are compactified by
  $X\st^\CC(\RR)$. Unfortunately, the exposition is not entirely
  accurate. In particular, Thm. II.10.5 is incorrect as stated. Part
  {\it iv)} of the theorem asserts, for example, that the boundary is
  of codimension $1$, which is not true in general, as Example
  \ref{it:compare3} {\it iii)} above shows. The error seems to lie in
  the transition from (II.10.1) to (II.10.2), where, without
  justification, ``$\sigma$'' is replaced by ``$o$''. Coincidentally,
  the discussion from Eq. (II.10.2) to the end of the section provides
  a description of our $\sX\st$ instead of $X\st^\CC(\RR)$. In
  particular, \Cite{BorelJi}{Thm. II.10.5} becomes correct when
  interpreted as a description of the $G(\RR)$-orbits in $\sX\st$.

\end{remark}

For future reference, we briefly establish the connection between our
theory and the classical theory of symmetric spaces of non-compact
type. Let $k=\RR$, let $G$ be a semisimple group, let $\sG=G(\RR)^0$
be the identity component of the group of real points, let
$\sK\subseteq \sG$ be a maximal compact subgroup, and let
$H\subseteq G$ be its Zariski closure in $G$. Then $\sX=\sG/\sK$ is a
symmetric space, and $X=G/H$ is an $\RR$-variety with
$\sX\subseteq X(\RR)$. If $G$ is of adjoint type, $\sX$ has a
distinguished compactification by a manifold with corners, called the
\emph{maximal Satake compactification}. See \Cite{BorelJi}{\S I.10}
for details.

\begin{proposition}

  Let $\sX=\sG/\sK$ be a symmetric space. Assume that $G$ is of
  adjoint type. Then the closure of $\sX$ in $\sX\st$ is the maximal
  Satake compactification of $\sX$.

\end{proposition}

\begin{proof}

  The spherical roots of $X$ are the restricted roots of $G$ (see
  \cite{Vust}). In particular, the $\cG^*$-action on them is
  trivial. It follows that $\sX$ has isomorphic closures in $\sX\st$
  and $X\st^\CC(\RR)$. The closure in the latter is the maximal Satake
  compactification by \Cite{BorelJi}{I.17.9, II.14.10}.
\end{proof}

It is well known that the maximal Satake compactification of a
symmetric space has the structure of a manifold with corners. Recall
that a \emph{manifold with corners} is a space $\sXq$ with a
$\cC^\infty$-structure that is locally modeled on
$\RR_{\ge0}^l\times\RR^{n-l}$ for a fixed $n$ and varying $l$. See,
e.g., \cite{Joyce} for an introduction to manifolds with corners. The
points corresponding to
$0^l\times\RR^{n-l}\subseteq\RR_{\ge0}^l\times\RR^{n-l}$ are said to
have \emph{depth $l$}. Let $\sXq_l:=\{x\in\sXq\mid\text{$x$ is of
  depth $l$}\}$. A \emph{boundary component} is the closure of a
connected component of $\sXq_1$.

In this generality, pathologies such as self-intersections of boundary
components may occur. The teardrop (see \Cite{Joyce}{Fig.~2.1}) is the
simplest example. For this reason, Jänich \cite{Jaenich} introduced
the concept of an $\<N\>$-manifold, which is a manifold with corners
together with an assignment of a number in $\{1,\ldots,N\}$ to every
boundary component such that, at each point, all boundary components
meeting there have distinct numbers (see \cites{Jaenich,Joyce} for
precise definitions). This definition prevents self-intersections. In
particular, all boundary components of $\sXq$ are themselves manifolds
with corners.

Clearly, the set $\{1,\ldots,N\}$ can be replaced by any set
$\Sigma$. In that case, we speak of a $\Sigma$-manifold $\sXq$ such
that every boundary component $\sY$ is of type $\sigma\in\Sigma$. Then
$\sY$ is a $\Sigma\setminus\{\sigma\}$-manifold. More generally, let
$\sY$ be a connected component of $\sXq_l$ for some $l$. Then $\sY$ is
contained in precisely $l$ boundary components; let
$\Sigma_0\subseteq\Sigma$ denote the set of their types. The closure
$\sYq$ of $\sY$ has the structure of a
$\Sigma\setminus\Sigma_0$-manifold. All subsets of the form $\sY$ form
a stratification which we call the \emph{boundary stratification} of
$\sXq$.

One can construct $\Sigma$-manifolds by cutting manifolds along
hypersurfaces. More precisely, let $\sXs$ be a manifold without
boundary, and for each $\sigma\in\Sigma$ let $D_\sigma\subset\sXs$ be
a smooth closed hypersurface (possibly disconnected or empty). Assume
that all $D_\sigma$ intersect each other transversely. For each
$x\in\sXs$, let $\|dp|(x)$ be the number of $\sigma\in\Sigma$ such
that $x\in D_\sigma$, and let $\sX$ be the set of all points of depth
$0$, i.e., the complement of the union of all $D_\sigma$.

Let $x\in\sXs$. Then the hypersurfaces $D_\sigma$ containing $x$
divide a neighborhood of $x$ into $2^{\|dp|(x)}$ subsets called
orientations of $x$. More precisely, an orientation is an element of
$\varprojlim_U\pi_0(U\cap\sX)$, where $U$ runs through all open
neighborhoods of $x$ in $\sXs$.

Now let $\sXq$ be the set of all pairs $(x,\eta)$, where $x\in\sXs$
and $\eta$ is an orientation of $x$. It can be topologized so that the
forgetful map $\pi:\sXq\to\sXs:(x,\eta)\mapsto x$ is proper and
surjective, and the fiber over $x$ has cardinality
$2^{\|dp|(x)}$. Then each boundary component of $\sXq$ maps to a
specific hypersurface $D_\sigma$, which makes $\sXq$ into a
$\Sigma$-manifold.

We now apply this construction to a homogeneous $\RR$-spherical
variety $X$ or specifically to its real points $\sX$. Choose any
smooth fan $\cF$ whose support is $\cZ_\RR(X)$, and let $\sXs$ be the
set of real points of $X(\cF)$. Then $\sXs$ is a smooth
compactification of $\sX$. Its boundary is a union of smooth
hypersurfaces indexed by the set $\Sigma$ of one-dimensional cones in
$\cF$. The manifold with corners $\sXq$ obtained by cutting along the
boundary is, therefore, a compact $\Sigma$-manifold. It is
$G(\RR)$-equivariant. A distinguishing feature of $\sXq$ is that each
connected component contains exactly one open $G(\RR)^0$-orbit. If
that orbit is a (Riemannian) symmetric space and $\cF$ is the fan of
faces of $\cZ_\RR(X)$, then the component is precisely its maximal
Satake compactification. Note that if $X$ is $\RR$-adjoint, one can
choose $\cF$ to be the standard fan. In that case, $\Sigma$ can be
identified with $\Sigma_\RR(X)$.

\begin{proposition}

  Let $\sG:=G(\RR)^0$. Then the boundary strata of $\sXq$ are
  precisely the $\sG$-orbits.

\end{proposition}

\begin{proof}

  The boundary strata are connected and $\sG$-stable. Moreover, the
  action of $\sG$ on each of them is locally transitive. Hence, they
  are $\sG$-orbits.
\end{proof}

\begin{examples}

  In all examples, we set $k=\RR$, take $G=SL(2,\CC)$ as an algebraic
  group, and let $\sG=G(\RR)=SL(2,\RR)$ be the group of its real
  points. All examples have $\RR$-rank $1$ and are not horospherical
  except for the last one. Every $X=G/H$ can be compactified by adding
  either one closed orbit $Y$ or, in the last case, two closed orbits.

  {\it i)} $H_1=SO(2)$. Then $G/H_1$ is the space of quadratic forms
  $x_1\xi^2+2x_2\xi\eta+x_3\eta^2$ with discriminant
  $x_2^2-x_1x_3=-1$. Therefore, $\sX$ is a two-sheeted
  hyperboloid. Its closure in $\P^3$ is the smooth quadric
  $x_2^2-x_1x_3=-x_0^2$ with signature $(3,1)$. Thus,
  $\sX\st\cong S^2$, the $2$-sphere. The boundary $\sY=Y(\RR)$ is
  given by $x_0=0$, which defines a great circle $S^1\subseteq S^2$
  that separates $\sX\st$ into two parts. The embedding
  $\sX\st$ can be identified with the Riemann sphere. Cutting $\sX\st$
  along $\sY$ yields the disjoint union of two closed disks. Thus,
  $\sG$ has two open and two closed orbits in $\sXq$.

  {\it ii)} $H_2=SO(1,1)$. In this case, $G/H_2$ is the space of
  quadratic forms $x_1\xi^2+2x_2\xi\eta+x_3\eta^2$ with discriminant
  $x_2^2-x_1x_3=+1$. Therefore, $\sX$ is a one-sheeted
  hyperboloid. Its closure in $\P^3$ is the smooth quadric
  $x_2^2-x_1x_3=x_0^2$ with signature $(2,2)$. Thus,
  $\sX\st\cong \P^1(\RR)\times\P^1(\RR)\cong S^1\times S^1$, a
  $2$-torus. The boundary $\sY$ corresponds to the diagonal in
  $S^1\times S^1$, which does not separate $\sX\st$. Cutting along
  $\sY$ results in a cylinder $S^1\times[0,1]$. Thus, $\sG$ has one
  open and two closed orbits in $\sXq$. Note that $\sXq$ is \emph{not}
  the closure of an open $\sG$-orbit in $\sX\st$, in contrast to the
  maximal Satake compactification of a Riemannian symmetric
  space. This example also shows that connected components of $\sXq$
  can have more than one closed $\sG$-orbit.

  {\it iii)} $H_3=N_G(SO(2))$ or $H_4=N_G(SO(1,1))$ with real points
  $\sH_i:=H_i(\RR)$. Note that $H_3\ne H_1$ but $\sH_3=\sH_1$, meaning
  $H_3\setminus H_1$ has no real points. These two subgroups lead to
  the same $\RR$-variety. In fact, $X$ is the space of similitude
  classes of quadratic forms $x_1\xi^2+2x_2\xi\eta+x_3\eta^2$ with
  discriminant $x_2^2-x_1x_3\ne0$. This shows that $\sX\st$ is the
  projective space $\P^2(\RR)$. The boundary is a smooth conic that
  separates $\P^2(\RR)$ into two parts: one, $\cong \sG/\sH_3$, is
  diffeomorphic to an open disk; the other, $\cong \sG/\sH_4$, is
  diffeomorphic to an open Möbius strip. Cutting along $\sY$ results
  in the disjoint union of a closed disk (the compactification of
  $\sG/\sH_1$) and a closed Möbius strip (the compactification of
  $\sG/\sH_4$). Thus, $\sG$ has two open and two closed orbits in
  $\sXq$.

  {\it iv)} In the last example, let $H_5=N$ be a maximal unipotent
  subgroup. It is horospherical; hence, the compactification contains
  two closed orbits $Y_1$ and $Y_2$. If we take for $\cF$ the two
  half-lines in $\cZ_\RR(X)=\QQ$ (there is no other possibility), then
  $X(\cF)$ equals the projective space $\P^2$ blown up at the
  origin. Thus, $\sXs$ is the non-orientable $S^1$-bundle over $S^1$,
  i.e., a Klein bottle. The two closed orbits form two sections of
  this bundle. Since the normal bundle is non-trivial, cutting along
  either of them does not result in two circles. Instead, one gets a
  single circle which projects as a two-fold cover to $S^1$. Thus,
  $\sXq$ consists of one open orbit and two closed
  orbits. Topologically, $\sXq$ is a closed annulus. Note that the two
  closed orbits are no longer sets of real points of algebraic
  varieties. Indeed, the isotropy group is the identity component
  $\sP=P(\RR)^0$, which is a non-algebraic subgroup of $SL(2,\RR)$.

\end{examples}

We conclude this section by stating a property for which the Satake
compactification behaves better than the toroidal ones.

\begin{proposition}\label{prop:princ}

  Let $\Xq$ be a smooth toroidal embedding of the $\RR$-spherical
  variety $X=G/H$, and let $\sXq$ be the embedding of $\sX=X(\RR)$
  obtained by cutting $\Xq(\RR)$ along all boundary divisors. Let
  $\sK\subseteq\sG=G(\RR)^0$ be a compact subgroup. Then the principal
  $\sK$-isotropy groups of the $\sG$-orbits in $\sXq$ are conjugate to
  each other.
  
\end{proposition}

\begin{proof}

  Let $\sG\xq$ be some $\sG$-orbit in $\sXq$ such that $\xq$ is in its
  principal $\sK$-locus. We show that $\xq$ is also principal in
  $\sXq$. By the slice theorem for compact groups, it suffices to show
  that $\sK_\xq$ acts trivially on the normal space $V$ of the orbit
  $\sG\xq$ in $\sXq$. Here the normal space at $\xq$ is just the
  normal space of its image $x$ in $\Xq(\RR)$. Now $\xq$ is of the
  form $(x,\eta)$, where $x\in\Xq(\RR)$ and $\eta$ is an orientation
  at $x$. The isotropy group $\sK_x$ acts on $V$ such that all
  hyperplanes tangent to the boundary divisors passing through $x$ are
  fixed. This means that the action of $\sK_x$ on $V$ is
  diagonal. Moreover, by compactness, all eigenvalues are $\pm1$. This
  implies that if $g\in\sK_x$ fixes one orientation, it acts trivially
  on $V$. The assertion follows.
\end{proof}

\end{document}